%% file: lms.tex
\documentclass[a4paper,11pt,reqno]{article}

\usepackage[utf8]{inputenc}
\usepackage[T1]{fontenc}
\usepackage{lmodern}
\usepackage[english]{babel}
\usepackage{microtype,csquotes}
\usepackage[normalem]{ulem}

\usepackage[backend=biber,hyperref=true,style=alphabetic,giveninits=true,maxbibnames=99]{biblatex}

\usepackage{tikz}
\usetikzlibrary{math}
\definecolor{midnightblue}{rgb}{0.1, 0.1, 0.44}

\usepackage[pdfusetitle,colorlinks,allcolors=midnightblue]{hyperref}
\AtBeginDocument{}
\AtBeginDocument{}
\AtBeginDocument{}
\hypersetup{
	pdfsubject={Minimization problems with linear and nonlinear total variation.},
	pdfkeywords={total variation, nonlinear total variation, energy minimization, regularity of minimizers,
		optimality conditions, fidelity, piecewise monotone, Cantor, Hölder}}

\usepackage[a4paper,hmargin=2.5cm,bottom=3cm,top=3cm,footskip=3\baselineskip,
	marginparwidth=2.4cm,marginparsep=0.05cm]{geometry}

\usepackage{enumitem}
\setlist[enumerate,1]{label=(\roman*)}
\setlist[enumerate,2]{ref=\theenumi(\theenumii)}

\usepackage{booktabs}

\usepackage{fsmath}
\usepackage{bm}
\usepackage{stmaryrd}

\DeclareMathOperator{\BV}{BV}
\DeclareMathOperator{\SBV}{SBV}
\DeclareMathOperator{\reg}{reg}
\newcommand{\BVreg}{\BV_{\reg}}
\newcommand{\TVreg}{\TV_{\reg}}

\newcommand{\TVphitheta}{\TV\!\!_{\phi,\theta}}
\newcommand{\dx}{\d x}
\newcommand{\dy}{\d y}
\newcommand{\dz}{\d z}

\newcommand{\orcidlink}[1]{\href{https://orcid.org/#1}{#1}}

\title{Minimizers of one-dimensional regularization problems \\ with linear and nonlinear total variation}
\author{
	Luca Lussardi%
	\thanks{Politecnico di Torino, Torino,
	\url{luca.lussardi@polito.it}, \textsc{orcid}: \orcidlink{0000-0001-9130-3573}}\and
	Marco Morandotti%
	\thanks{Politecnico di Torino, Torino,
		\url{marco.morandotti@polito.it}, \textsc{orcid}: \orcidlink{0000-0003-3528-6152}}\and
	Federico Stra%
	\thanks{Politecnico di Torino, Torino,
		\url{federico.stra@polito.it}, \textsc{orcid}: \orcidlink{0000-0001-8151-1045}}
}

\begin{document}

\maketitle

\begin{abstract}
	A study of the minimizers of one-dimensional Rudin--Osher--Fatemi-type functionals with linear or
	nonlinear total variation and fidelity term is undertaken.
	Conditions on the input datum and the parameters of the functional are found that force the
	input itself to be the minimizer or not.
	In the linear setting the discriminating conditions on the parameters are complementary
	highlighting the sharpness of the results.
	The nonlinear setting is substantially different:
	while the non-minimality of the datum is treated in analogy with the linear case,
	for the minimality of the input only partial answers are found.
	The results in the nonlinear setting hinge on auxiliary constrained or penalized minimization problems
	investigating the behavior of optimal transitions with prescribed height.
	Additionally, they are complemented by some numerical examples.
\end{abstract}


\tableofcontents

\section{Introduction}

In the framework of image processing, the total variation energy functional plays a fundamental role. Since the pioneering work of Rudin, Osher, and Fatemi \cite{ROF}, total variation regularization has been extensively studied both from the theoretical and applied viewpoints, leading to a wide literature on qualitative properties of minimizers and on nonlinear variants of the classical model (see the survey \cite{CCN_survey} and the references therein; in particular, see also \cite{Chambolle1997} and \cite{Chambolle2004}).

Given an interval $I\subseteq\setR$, a datum $g: I\to\setR$ describing the gray level of an input image, $\lambda,\gamma>0$, and $p\in [1,\infty]$, we study one-dimensional minimization problems for functionals of the form
\[
	F(u) = \TV(u;I) + \lambda\norm{u-g}_p^\gamma
\]
as well as their nonlinear counterpart obtained by replacing the classical total variation $\TV$ with a generalized functional $\TV_{\phi,\theta}$. Precisely, in \autoref{def:TV} we define $\TV(u;I)$ as the \emph{pointwise} version of the total variation of the function $u$ over the set $I$. Up to a choice of good representatives, this corresponds to the quantity $\abs{\D u}(I)$ defined for $u\in\BV(I)$ according to \cite{AFP}. For $u\in\BV(I)$, we define also the nonlinear total variation as
\[
	\TVphitheta(u;I)
	= \int_I \phi\bigl(\abs{u'}(x)\bigr)\dx + \abs{\D^cu}(I) + \sum_{x\in J_u} \theta\bigl(\abs{[u](x)}\bigr),
\]
where $\phi,\theta:[0,\infty)\to[0,\infty)$; see \autoref{def:TV_phi_theta} for the precise definition, including the requirements on $\phi$ and $\theta$.

Among the relevant questions concerning this optimization problem, we are motivated by understanding
the regularity of the minimizers and the degeneracy of the problem.
In particular, in this article we address these questions:
\begin{enumerate}[label=Q\arabic*:]
	\item Can the minimizer exhibit a Cantor part $\D^cu\neq0$?
	\item Is the minimizer $u=g$?
\end{enumerate}

We will obtain and answer to the former question as a byproduct of our investigation of the latter.
With regard to Q2, our analysis shows that the behavior of $g$ near the boundary of the domain is crucial.
More specifically, under the assumption of at least or at most $\alpha$-growth at the endpoints (see \autoref{def:alpha-growth}), we determine conditions on the parameters $p,\gamma$ which cause the input $g$ to be a minimizer for sufficiently large $\lambda$, or cause it not to be a minimizer for any $\lambda$.
The threshold $1/\gamma\lessgtr1+1/(p\alpha)$ that we find for the two behaviors is sharp.

The linear case turns out to admit a more detailed analysis. We obtain several sharp results characterizing both the failure and the validity of minimality of the datum. On the one hand, we prove that functions exhibiting sufficiently fast growth near the boundary cannot be minimizers. On the other hand, monotone or piecewise monotone data with controlled growth at the endpoints are shown to be essentially unique minimizers for sufficiently large values of the fidelity parameter $\lambda$.
Moreover, the Cantor staircase function provides an interesting example showcasing that the minimizers may exhibit only the Cantor part of the derivative, thus answering Q1 positively.

By contrast, the nonlinear setting is considerably more difficult and remains only partially understood. Although we are able to establish some preliminary structural results and derive partial analogues of the linear theory, the nonlinear nature of $\TVphitheta$ leads to significant differences and obstacles with respect to the linear problem; as a consequence, most of the techniques developed in the linear case do not extend directly to the nonlinear case, mostly due to the fact that the functional exhibits a worse behavior with respect to local truncations. To work around this issue, we resort to the auxiliary boundary value problem \eqref{eq:E-def}, which allows us to characterize optimal transitions with prescribed height. Moreover, the failure of the analogues to the linear theorems is highlighted by the behavior of the minimizers of the penalized functional \eqref{eq:penalized-functional}. Yet, our treatment of the nonlinear problem is only partial and we regard this part of the paper as mainly exploratory, with some examples suggesting that the nonlinear theory may exhibit genuinely different phenomena. This section is complemented by numerical examples to showcase and visualize some possible behaviors of the minimizers of \eqref{eq:penalized-functional}.

Our proofs are mainly based on fine constructions of competitors for the minimization problems
heavily exploiting the convexity of the functionals. In addition, we resort to a range of techniques
from calculus of variations and convex analysis, including a detailed study of Euler--Lagrange
equations (for the nonlinear problem), the notion of perspective of a convex function, Legendre
transform, and other inequalities inspired by optimal transport on the real line.

The article is organized as follows. In \autoref{sec:preliminaries} we collect the necessary
preliminaries on functions of bounded variation and introduce the nonlinear total variation
functional. \autoref{sec:linear} contains the main results for the linear model, distinguishing
between situations where the datum fails to be a minimizer and situations where it is optimal. In
\autoref{sec:nonlinear} we turn to the nonlinear setting, where we analyze the associated boundary
value problem, obtain partial results which leave some open questions, and showcase the behavior of
optimal profiles by means of numerical examples. The approach is based on insight about the
minimization problem developed in the appendix \autoref{sec:first-variation-simple}.

\subsection{Main results and open questions}

We summarize here the main results contained in the paper.

\begin{itemize}
	\item Linear total variation. After the simple and illustrative example \autoref{thm:not-min-cantor-function} with the Cantor staircase function, in \autoref{thm:not-min-alpha-growth} we show that if $g$ has sufficiently fast growth near the endpoints and $\gamma$ is large enough then $g$ is not a minimizer. In \autoref{sec:linear-min} we provide progressively more general results of minimality of the input datum: we start with the simple situation in \autoref{thm:min-flat} where $g$ is constant near the endpoints; we move on to the case in \autoref{thm:min-monotone} where $g$ is monotone and has controlled growth; finally, with the help of the fundamental \autoref{prop:u_contained_in_g}, in \autoref{thm:min-piecewise-monotone} we extend the result to piecewise monotone functions. The complementarity of the ranges for the exponent $\gamma$ in \autoref{thm:not-min-alpha-growth} and \autoref{thm:min-monotone} demonstrates the sharpness of our results.
	\item Nonlinear total variation. In \autoref{thm:not-min-alpha-growth-nonlinear} we prove the analogue of \autoref{thm:not-min-alpha-growth} under the assumption of a lower bound on the boundary value problem \eqref{eq:E-def}. \autoref{sec:nonlinear-min} is more exploratory: as it turns out surprisingly, the desired nonlinear statement corresponding to \autoref{prop:u_contained_in_g} is false, a fact which we demonstrate with a very simple counterexample in \autoref{prop:nonlinear-competitors}. The culprit is the non minimality of constant functions for a functional penalizing jumps at the right endpoint of the interval, as stated in \autoref{prop:not-min-zero-nonlinear}. We demonstrate the behavior of minimizers of this penalized problem with some numerical examples in \autoref{sec:numerical-examples}.
\end{itemize}

We stress the fact that our study of the linear case (\autoref{sec:linear}) provides more precise
statements than what could be derived from the nonlinear case (\autoref{sec:nonlinear}) by
specialising to $\phi=\theta=\Id$. This is because the \autoref{def:TV} of $\TV$ is a pointwise
version of total variation which is sensitive to pointwise modifications of the function, in
contrast with $\TVphitheta$ of \autoref{def:TV_phi_theta}, which is invariant in the $L^1$
equivalence classes.

Our partial treatment of the regularity issue leaves open some questions more advanced than Q1,
for instance whether there exists a minimizer $u\in\SBV$ when the input is $g\in\SBV$.
This could have implications in regularization of images.

Finally, in the nonlinear setting, future research may help identify sufficient conditions for the
minimality of the input, potentially based on the behavior in $0$ and at $\infty$ of the
non-linearities $\phi$ and $\theta$. While our results show that there are situations where
the input performs better that certain specific competitors (see \autoref{prop:nonlinear-competitors}),
the exposition in \autoref{sec:nonlinear-min} demonstrates that the situation is substantially more
involved than the linear case and a more detailed analysis is required to fully resolve this matter.

\section{Preliminaries}\label{sec:preliminaries}

\subsection{Definitions and basic properties of \texorpdfstring{$\BV$}{BV} functions}

\begin{definition}\label{def:TV}
	Given a subset $I\subseteq\setR$ and a function $u:I\to\setR$, the \emph{total variation} of $u$ over $I$ is
	\[
		\TV(u;I)
		= \sup\set*{\sum_{i=1}^n \abs{u(x_i)-u(x_{i-1})}}
		{x_i\in I\ \forall i\in\{0,\dots,n\} \wedge x_{i-1}\leq x_i\ \forall i\in\{1,\dots,n\}}.
	\]
	The vector space of \emph{functions with finite total variation} over $I$ is
	\[
		\TV(I) = \set{u:I\to\setR}{\TV(u;I)<\infty}.
	\]
\end{definition}

In the sequel, we may use the abbreviated notation $\TV(u)$ instead of $\TV(u;I)$ when it is clear
from the context.

\begin{definition}\label{def:regulated}
	A function $u:I\to\setR$ is said to be \emph{regulated} if the left limit $u^-(x)$ exists and is
	finite at any $x\in\setR$ which is an accumulation point of $I$ from the left, and the right limit
	$u^-(x)$ exists and is finite at any $x\in\setR$ which is an accumulation point of $I$ from the
	right.
\end{definition}

\begin{proposition}\label{prop:BV is regulated}
	Every function $u\in\TV(I)$ is regulated.
\end{proposition}

\begin{definition}\label{def:regular}
	A regulated function $u:I\to\setR$ is said to be \emph{regular} if
	\[
		u(x) \in [u^-(x)\wedge u^+(x),u^-(x)\vee u^+(x)]
	\]
	for every $x\in I$ which is a bilateral accumulation point of $I$ and $u(x)=u^\pm(x)$ at any point
	$x\in I$ which is a unilateral accumulation point of $I$.
	The set of \emph{regular functions with finite total variations} over $I$ is
	\[
		\TVreg(I) = \set{u\in\TV(I)}{\text{$u$ is regular}}.
	\]
\end{definition}

The total variation $\TV(u;I)$ and the space of functions $\TV(I)$ are the pointwise version of
$\abs{\D u}(I)$ and $\BV(I)$ as defined in \cite{AFP}. More explicitly, $\TV(u;I)$ is precisely
the \emph{pointwise variation} defined in \cite[Definition 3.26]{AFP}; \cite[Theorem 3.27]{AFP}
shows that the variation $\abs{\D u}(I)$ is equal to the relaxation of $\TV(u,I)$ in the $\leb^1$
equivalence class (see also \cite[(3.23)]{AFP}); \cite[Theorem 3.28]{AFP} shows that any function
$u\in\BV(I)$ admits a \emph{good representative} $v\in\TVreg(I)$ such that $\abs{\D u}(I)=\TV(v;I)$.

The following properties follow directly from the definitions.

\begin{proposition}\label{prop:basic-properties}
	Let $I\subseteq\setR$ be an interval.
	\begin{enumerate}
		\item If $u\in\TV(I)$, then $u$ has at most countably many discontinuities (and they are jumps).
		\item If $u\in\TV(I)$, then $u$ is regular if and only if
		      $\liminf_{x'\to x}u(x')\leq u(x)\leq\limsup_{x'\to x}u(x')$ for every accumulation point $x\in I$.
		\item If $u,v\in\TVreg(I)$ coincide $\leb^1$-a.e., then they share the same set of jump points
		      and they coincide elsewhere.
		\item Every $u\in\BV(I)$ admits an equivalent function $v\in\TVreg(I)$ which coincides
		      $\leb^1$-a.e.\ with $u$.
		\item \label{it:regularizations}
		      If $u\in\TV(I)$, then a function $v:I\to\setR$ is regular and coincides with $u$
		      $\leb^1$-a.e.\ if and only if
		      $\liminf_{x'\to x}u(x')\leq v(x)\leq\limsup_{x'\to x}u(x')$ for every $x\in I$.
		\item If $u$ is regulated and $x$ is a bilateral accumulation point of $I$, then
		      $u(x) \in [u^-(x)\wedge u^+(x),u^-(x)\vee u^+(x)]$ if and only if the ``pointwise variation''
		      $\llbracket u\rrbracket = \abs{u(x)-u^-(x)}+\abs{u^+(x)-u(x)}$ is equal to the jump
		      $\abs{[u](x)} = \abs{u^+(x)-u^-(x)}$.
	\end{enumerate}
\end{proposition}

Regulated, regular and $\TV$ functions are well behaved with regard to composition with continuous
or Lipschitz functions.

\begin{proposition}\label{prop:tv-composition}
	Let $I\subseteq\setR$.
	\begin{enumerate}
		\item If $u:I\to\setR$ is regulated and $f\in C\bigl(\overline{u(I)}\bigr)$, then $f\circ u$ is
		      regulated.
		\item If $u:I\to\setR$ is regular and $f\in C\bigl(\overline{u(I)}\bigr)$ is monotone, then
		      $f\circ u$ is regular.
		\item \label{it:lip-TV} If $u\in\TV(I)$ and $f\in\Lip\bigl(u(I)\bigr)$, then $f\circ u\in\TV(I)$ and
		      $\TV(f\circ u;I)\leq\Lip(f)\TV(u;I)$.
		\item If $u\in\TVreg(I)$ and $f\in\Lip\bigl(u(I)\bigr)$ is monotone, then $f\circ u\in\TVreg(I)$.
	\end{enumerate}
\end{proposition}

\begin{proof}
	Since the rest are straightforward, we prove only \ref{it:lip-TV}, which will play a crucial role
	in what follows.
	\begin{itemize}
		\item[(iii)] For every choice of sorted points $x_0<\dots<x_n$ in $I$ we have
			\[
				\sum_{i=1}^n \abs*{f\bigl(u(x_i)\bigr)-f\bigl(u(x_{i-1})\bigr)}
				\leq \Lip(f) \sum_{i=1}^n \abs{u(x_i)-u(x_{i-1})}
				\leq \Lip(f) \TV(u;I);
			\]
			taking the supremum on the left hand side proves the thesis. \qedhere
	\end{itemize}
\end{proof}

As a special application of the last two points of \autoref{prop:tv-composition} we get the
following useful property.

\begin{remark}\label{rmk:BV-truncation}
	Given $m,M\in\setR$, the functions $f_m,f^M,f_m^M\in\Lip(\setR)$ defined as
	\begin{align*}
		f_m(t)   & = t\vee m,        &
		f^M(t)   & = t\wedge M,      &
		f_m^M(t) & = t\vee m\wedge M
	\end{align*}
	are $1$-Lipschitz and monotone, therefore if $u\in\TV(I)$ then the truncated functions
	$u\vee m$, $u\wedge M$ and $u\vee m\wedge M$ are also in $\TV(I)$ and
	\[
		\TV(u\vee m;I),\ \TV(u\wedge M;I),\ \TV(u\vee m\wedge M;I) \leq \TV(u;I).
	\]
	If $u\in\TVreg(I)$, the truncated functions are also in $\TVreg(I)$.
\end{remark}

We collect here some further elementary properties of the pointwise total variation.

\begin{proposition}\label{prop:further-properties}
	Let $I\subseteq\setR$ be an interval and $u\in\TV(I)$. Then the following hold true.
	\begin{enumerate}
		\item \label{it:tv-monotone-diff} $u$ can be expressed as the difference of the two
		      non-decreasing functions $x\mapsto\TV(u;(-\infty,x])$ and $x\mapsto\TV(u;(-\infty,x])-u(x)$;
		\item $\TV(u;[a,c])=\TV(u;[a,b])+\TV(u,[b,c])$ for every $a\leq b\leq c$ in $I$;
		\item $\TV(u;[a,c])=\abs{u(a)-u^+(a)}+\TV(u,(a,c])$ for every $a<c$ in $I$;
		\item \label{it:regular-better} if $\tilde u:I\to\setR$ is such that
		      \[
			      \liminf_{y\to x} u(y) \leq \tilde u(x) \leq \limsup_{y\to x} u(y),
		      \]
		      then $\TV(\tilde u;I) \leq \TV(u;I)$;
	\end{enumerate}
\end{proposition}

\begin{proof}
	Property \ref{it:tv-monotone-diff} is a well known fact;
	the next two splitting properties are elementary;
	therefore, we only prove \ref{it:regular-better}.
	\begin{enumerate}
		\item[(iv)]
			Let $x_0<x_1<\dots<x_n$ be a sorted partition of $I$. Given $\eps>0$, for every $i\in\{0,\dots,n\}$ there exist distinct points
			$x_{i,\eps}^\pm\in I$ such that
			\begin{align}
				u(x_{i,\eps}^-)-\eps & \leq \tilde u(x_i) \leq u(x_{i,\eps}^+)+\eps &  & \forall i\in\{0,\dots,n\}, \label{eq:inf-sup-points} \\
				x_{i-1,\eps}^\pm     & < x_{i,\eps}^\pm                             &  & \forall i\in\{1,\dots,n\}. \notag
			\end{align}
			For every $i\in\{0,\dots,n\}$, let $x_{i,\eps}'<x_{i,\eps}''$ such that
			$\{x_{i,\eps}',x_{i,\eps}''\}=\{x_{i,\eps}^-,x_{i,\eps}^+\}$.
			The sequence of points $x_{0,\eps}'<x_{0,\eps}''<x_{1,\eps}'<\dots<x_{n,\eps}'<x_{n,\eps}''$ forms
			a partition of $I$.
			By the triangle inequality and \eqref{eq:inf-sup-points} we get
			\[
				\begin{split}
					\abs{u(x_{i,\eps}')-\tilde u(x_i)} + \abs{\tilde u(x_i)-u(x_{i,\eps}'')}
					&= \abs{u(x_{i,\eps}^+)-\tilde u(x_i)} + \abs{\tilde u(x_i)-u(x_{i,\eps}^-)} \\
					&\leq \abs*{\bigl(u(x_{i,\eps}^+)+\eps\bigr)-\tilde u(x_i)}
					+ \abs*{\tilde u(x_i)-\bigl(u(x_{i,\eps}^-)-\eps\bigr)} + 2\eps \\
					&= \abs*{\bigl(u(x_{i,\eps}^+)+\eps\bigr)-\bigl(u(x_{i,\eps}^-)-\eps\bigr)} + 2\eps \\
					&\leq \abs{u(x_{i,\eps}^+)-u(x_{i,\eps}^-)} + 4\eps
					= \abs{u(x_{i,\eps}')-u(x_{i,\eps}'')} + 4\eps,
				\end{split}
			\]
			from which we obtain
			\[
				\begin{split}
					\sum_{i=1}^n \abs{\tilde u(x_{i-1})-\tilde u(x_i)}
					&\leq \sum_{i=1}^n \left(\abs{\tilde u(x_{i-1})-u(x_{i-1,\eps}'')}
					+ \abs{u(x_{i-1,\eps}'')-u(x_{i,\eps}')}
					+ \abs{u(x_{i,\eps}')-\tilde u(x_i)}\right) \\
					&\leq \sum_{i=0}^n \left(\abs{u(x_{i,\eps}')-\tilde u(x_i)} + \abs{\tilde u(x_i)-u(x_{i,\eps}'')}\right)
					+ \sum_{i=1}^n \abs{u(x_{i-1,\eps}'')-u(x_{i,\eps}')} \\
					&\leq \sum_{i=0}^n \abs{u(x_{i,\eps}')-u(x_{i,\eps}'')}
					+ \sum_{i=1}^n \abs{u(x_{i-1,\eps}'')-u(x_{i,\eps}')} + 4(n+1)\eps.
				\end{split}
			\]
			By definition of total variation, the right hand side is less than or equal to $\TV(u;I)+4(n+1)\eps$.
			Sending $\eps\to0$ and taking the supremum of the left hand side over the partition $(x_i)_{0\leq i\leq n}$ we deduce the thesis. \qedhere
	\end{enumerate}
\end{proof}

We recall the classical decomposition $\D u=\D^au+\D^cu+\D^ju$ of the derivative of a function
$u\in\BV(I)$; see \cite[Corollary 3.33]{AFP}. The (linear) total variation of a function $u\in\TVreg(I)$ can be represented as
\[
	\TV(u;I) = \int_I \abs{u'(x)}\dx + \abs{\D^cu}(I) + \sum_{x\in J_u} \abs{[u](x)},
\]
where $J_u$ is the jump set of $u$ and $[u](x)=g^+(x)-g^-(x)$ is the jump value.
This leads us to introduce the nonlinear total variation $\TVphitheta(u;I)$ as follows.

\begin{definition}\label{def:TV_phi_theta}
	Given two non-decreasing functions $\phi,\theta:[0,\infty)\to[0,\infty)$ such that $\phi$ is convex,
	$\theta$ is concave, $\phi(0)=\theta(0)=0$ and
	\[
		\lim_{t\to\infty} \frac{\phi(t)}t = \lim_{t\to0^+} \frac{\theta(t)}t = \theta'(0) = 1,
	\]
	the \emph{nonlinear total variation} of $u\in\TV(I)$ is
	\[
		\TVphitheta(u;I)
		= \int_I \phi\bigl(\abs{u'}(x)\bigr)\dx + \abs{\D^cu}(I) + \sum_{x\in J_u} \theta\bigl(\abs{[u](x)}\bigr).
	\]
\end{definition}

It is simple to verify that the following conditions are equivalent:
\begin{align*}
	 & \phi(t)=t,\ \forall t\in[0,\infty) &
	 & \Leftrightarrow                    &
	 & \exists t\in(0,\infty): \phi(t)=t  &
	 & \Leftrightarrow                    &
	 & \phi'(0)=1.
\end{align*}

Our main theorems rely on the following property of the input datum near the endpoints.
The precise assumption will be anyway recalled explicitly in their statements.

\begin{definition}\label{def:alpha-growth}
	Given $\alpha>0$, a function $u:(a,b)\to\setR$ which admits the limit $g^+(a)$ is said to have
	\emph{at most $\alpha$-growth} at the endpoint $a$ if there exists $C>0$ such that
	$\abs{g(x)-g^+(a)}\leq C(x-a)^\alpha$ for $x\in(a,b)$. It is said to have \emph{at least $\alpha$-growth}
	if the reverse inequality holds.
\end{definition}

\subsection{Technical results}

\begin{definition}
	Given a measurable function $u:\Omega\to[0,\infty]$,
	let $\mu_u:[0,\infty]\to[0,\abs\Omega]$ be its \emph{super-level set measure}
	\[
		\mu_u(t) = \abs{\set{x\in\Omega}{u(x)>t}}.
	\]
	This function is non-increasing and lower semi-continuous, therefore it admits a non-increasing
	lower semi-continuous pseudo-inverse $(\mu_t)^{-1}:[0,\abs\Omega]\to[0,\infty]$ given by
	\[
		(\mu_u)^{-1}(s) = \inf\set{t}{\mu_u(t)\leq s}.
	\]
	We define the \emph{non-increasing} and \emph{non-decreasing rearrangements}
	$u_\dagger,u^\dagger:[0,\abs\Omega]\to[0,\infty]$ of $u$ as
	\[
		u_\dagger(s) = (\mu_u)^{-1}(s), \qquad
		u^\dagger(s) = (\mu_u)^{-1}(\abs\Omega-s), \qquad
		\forall s\in[0,\abs\Omega].
	\]
\end{definition}

The functions $\mu_u$ and $(\mu_u)^{-1}$ are related by the crucial property
\[
	\mu_u(t) \leq s \Longleftrightarrow (\mu_u)^{-1}(s) \leq t.
\]
Moreover, one can show that $\mu_{u_\dagger}=\mu_{u^\dagger}=\mu_u$ (see \autoref{rmk:rearrangement-ot},
where it is implicitly proven), therefore by the layer-cake representation we have
\[
	\int_\Omega f\bigl(u(x)\bigr) \dx
	= \int_0^{\abs\Omega} f\bigl(u_\dagger(s)\bigr) \d s
	= \int_0^{\abs\Omega} f\bigl(u^\dagger(s)\bigr) \d s
\]
for every $f:[0,\infty]\to[0,\infty]$ measurable.

The following proposition collects two contractivity properties of the monotone rearrangement, both
with respect to the $L^p$ norm and the (linear) total variation. The proof is based upon two
technical lemmas, which we postpone. See also \cite{Talenti2016} for a survey about rearrangements.

\begin{proposition}\label{prop:rearrangement-contraction}
	Let $p\in[1,\infty]$ and $u,v:(a,b)\to[0,\infty]$. Then
	\begin{align}
		\norm{u^\dagger-v^\dagger}_p     & \leq \norm{u-v}_p,
		\label{eq:monotone-rearrangement-contraction-Lp}                \\
		\TV\bigl(u^\dagger;(0,b-a)\bigr) & \leq \TV\bigl(u;(a,b)\bigr).
		\label{eq:monotone-rearrangement-contraction-TV}
	\end{align}
	\begin{subequations}
		For $p>1$, equality in \eqref{eq:monotone-rearrangement-contraction-Lp} holds if and only if
		\begin{equation}\label{eq:monotone-rearrangement-contraction-Lp-eq-cond}
			\bigl(u(x_2)-u(x_1)\bigr) \bigl(v(x_2)-v(x_1)\bigr) \geq 0, \qquad
			\text{for almost every } x_1,x_2\in(a,b);
		\end{equation}
		for $p=1$, equality in \eqref{eq:monotone-rearrangement-contraction-Lp} holds if and only if
		\begin{equation}\label{eq:monotone-rearrangement-contraction-L1-eq-cond}
			\max\{u(x_1),v(x_2)\} \geq \min\{u(x_2),v(x_1)\}, \qquad
			\text{for almost every } x_1,x_2\in(a,b).
		\end{equation}
	\end{subequations}
	Equality in \eqref{eq:monotone-rearrangement-contraction-TV} holds if and only if $u$ is monotone,
\end{proposition}

\begin{proof}
	For the convenience of the reader, we provide a sketch of the proof of the classical result
	\eqref{eq:monotone-rearrangement-contraction-Lp} for the case $(a,b)=(0,1)$.
	See also \cite[Theorem 2.8(c)]{Baernstein} with $\Psi(x,y)=\abs{x-y}^p$
	or \cite[Theorem 1.3.1, Corollary 1.3.2]{Rakotoson} with $\rho(x,y)=\abs{x-y}^p$ for independent
	and more complete proofs.
	By approximation with simple functions, let
	\begin{align*}
		u_n & = \sum_{i=1}^n u_{n,i} \bm1_{\left(\frac{i-1}n,\frac in\right)}, &
		v_n & = \sum_{i=1}^n v_{n,i} \bm1_{\left(\frac{i-1}n,\frac in\right)}.
	\end{align*}
	Let $\sigma:\{1,\dots,n\}\to\{1,\dots,n\}$ be a permutation that rearranges the values
	$(u_{n,i})_{i=1}^n$ in non-decreasing order, i.e.\ $u_{n,\sigma(i)} \leq u_{n,\sigma(i+1)}$ for
	every $i\in\{1,\dots,n-1\}$. Suppose that $i\mapsto v_{n,\sigma(i)}$ is not non-decreasing: then
	there must be $j\in\{1,\dots,n-1\}$ such that $v_{n,\sigma(j)}>v_{n,\sigma(j+1)}$.
	Let $\pi$ be the transposition that swaps $j$ and $j+1$.
	By applying \autoref{lem:two-point-inequality} below with
	\begin{align*}
		a & = u_{n,\sigma(j)},   &
		A & = u_{n,\sigma(j+1)}, &
		b & = v_{n,\sigma(j+1)}, &
		B & = v_{n,\sigma(j)},
	\end{align*}
	we get that
	\[
		\norm{u_n-v_n}_p^p
		= \sum_{i=1}^n \abs{u_{n,i}-v_{n,i}}^p
		= \sum_{i=1}^n \abs{u_{n,\sigma(i)}-v_{n,\sigma(i)}}^p
		\geq \sum_{i=1}^n \abs{u_{n,\sigma(i)}-v_{n,\pi\circ\sigma(i)}}^p.
	\]
	The number of inversions in the sequence $(v_{n,\pi\circ\sigma(i)})_{i=1}^n$ is one less than that
	in $(v_{n,\sigma(i)})_{i=1}^n$. By iterating this argument, we find a permutation $\tau$ such that
	$(v_{n,\tau(i)})_{i=1}^n$ is non-decreasing and
	\[
		\norm{u_n-v_n}_p^p
		\geq \sum_{i=1}^n \abs{u_{n,\sigma(i)}-v_{n,\tau(i)}}^p
		= \norm{u_n^\dagger-v_n^\dagger}_p^p.
	\]
	The proof is concluded letting $u_n\to u$ and $v_n\to v$, for which we have also
	$u_n^\dagger\to u^\dagger$ and $v_n^\dagger\to v^\dagger$.

	The proof of the equality condition for \eqref{eq:monotone-rearrangement-contraction-Lp} is
	reminiscent of the proof of $c$-cyclical monotonicity of optimal plans; see for instance
	\cite[Theorem 3.17]{AmbrosioBrueSemola}. Suppose $p>1$ and
	assume by contradiction that \eqref{eq:monotone-rearrangement-contraction-Lp-eq-cond} is violated.
	Without loss of generality, we may assume that $x_1<x_2$, $u(x_1)<u(x_2)$ and $v(x_1)>v(x_2)$,
	with $x_1,x_2$ Lebesgue points of both $u$ and $v$.
	For sufficiently small radius $r>0$, define the intervals $I_i^r=(x_i-r,x_i+r)$ and the rearranged function
	$v_r:(a,b)\to\setR$ given by
	\[
		v_r(x) = \begin{cases}
			v\bigl(x+(x_2-x_1)\bigr) & x\in I_1^r,       \\
			v\bigl(x-(x_2-x_1)\bigr) & x\in I_2^r,       \\
			v(x)                     & \text{otherwise}.
		\end{cases}
	\]
	By \autoref{lem:two-point-inequality} with $a=u(x_1)$, $A=u(x_2)$, $b=v(x_2)$, $B=v(x_1)$ and
	$f(t)=\abs{t}^p$ we have
	\[
		\begin{split}
			\abs{u(x_1)-v(x_1)}^p + \abs{u(x_2)-v(x_2)}^p
			&> \abs{u(x_1)-v(x_2)}^p + \abs{u(x_2)-v(x_1)}^p \\
			&= \abs{u(x_1)-v_r(x_1)}^p + \abs{u(x_2)-v_r(x_2)}^p.
		\end{split}
	\]
	By the Lebesgue property, this implies that for $r$ sufficiently small
	\[
		J_r \coloneqq
		\dashint_{-r}^r \sum_{i=1}^2 \bigl( \abs{u(x_i+y)-v(x_i+y)}^p - \abs{u(x_i+y)-v_r(x_i+y)}^p \bigr) \dy
		> 0.
	\]
	Then
	\[
		\norm{u-v}_p^p - \norm{u-v_r}_p^p
		= \int_{I_1^r\cup I_2^r} \bigl(\abs{u(x)-v(x)}^p-\abs{u(x)-v_r(x)}^p\bigr) \dx
		= 2r J_r > 0.
	\]
	This proves that
	\[
		\norm{u-v}_p > \norm{u-v_r}_p \geq \norm{u^\dagger-v_r^\dagger}_p = \norm{u^\dagger-v^\dagger}_p,
	\]
	hence \eqref{eq:monotone-rearrangement-contraction-Lp} is a strict inequality.
	The proof of \eqref{eq:monotone-rearrangement-contraction-L1-eq-cond} for the case $p=1$ is
	analogous.

	Let us now turn to the proof of \eqref{eq:monotone-rearrangement-contraction-TV}.
	Given $\eps>0$, let $x'_\eps,x''_\eps\in(a,b)$ be two different points such that
	\[
		u(x'_\eps) \leq \inf_{(a,b)} u + \eps
		\qquad\text{and}\qquad
		u(x''_\eps) \geq \sup_{(a,b)} u - \eps.
	\]
	By \autoref{def:TV},
	\[
		\TV\bigl(u;(a,b)\bigr)
		\geq \abs{u(x'_\eps)-u(x''_\eps)}
		\geq \sup_{(a,b)} u - \inf_{(a,b)} u - 2\eps.
	\]
	Letting $\eps\to0$ shows that
	\begin{equation}\label{eq:tv-sup-inf}
		\TV\bigl(u;(a,b)\bigr) \geq \sup_{(a,b)} u - \inf_{(a,b)} u.
	\end{equation}
	On the other hand,
	\[
		\TV\bigl(u^\dagger;(0,b-a)\bigr)
		= \sup_{(0,b-a)} u^\dagger - \inf_{(0,b-a)} u^\dagger
		= \esssup_{(a,b)} u - \essinf_{(a,b)} u
		\leq \sup_{(a,b)} u - \inf_{(a,b)} u,
	\]
	therefore we deduce the thesis.
	Let us now prove that if equality in \eqref{eq:tv-sup-inf} holds then $u$ is monotone.
	We prove the contrapositive. Assume that $u$ is not monotone: there exist $x_1<x_2$ and $x_3<x_4$
	such that $u(x_1)<u(x_2)$ and $u(x_3)>u(x_4)$. For every $\eps>0$, let $x'_\eps,x''_\eps$ as before.
	Let $S_\eps=\{x_1,x_2,x_3,x_4,x'_\eps,x''_\eps\}$.
	Rename the points as $(y_0^\eps,\dots,y_5^\eps)=(x_1,x_2,x_3,x_4,x'_\eps,x''_\eps)$.
	There exists a permutation $\sigma_\eps$ of $\{0,\dots,5\}$ such that $z_i^\eps = y_{\sigma_\eps(i)}^\eps$
	is increasing in $i$. Since the set of permutations is finite, there exists $\eps_n\to0^+$ such
	that $\sigma_{\eps_n}$ is always the same permutation $\sigma$.

	Let $a_i^{\eps_n}=u(y_i^{\eps_n})$, hence $u(z_i^{\eps_n})=a_{\sigma(i)}^{\eps_n}$.
	When $n\to\infty$, the finite sequence $(a^{\eps_n}_i)_{i=0}^5$ converges to some $(a_i)_{i=0}^5$
	with the property that $(a_{\sigma(i)})_{i=0}^5$ is not monotone, because of the presence of the
	points $x_1,x_2,x_3,x_4$ which cause a violation of monotonicity.
	From $\sup u-\eps\leq u(x''_\eps)\leq\max_i a_i^\eps\leq\sup u$ we deduce
	$\lim_{n\to\infty} \max_i a_i^{\eps_n}=\sup u$, and similarly for $\min$ and $\inf$.
	By \autoref{lem:c-monotonicity} below we therefore have
	\[
		\begin{split}
			\TV(u)
			&\geq \liminf_{n\to\infty} \TV(u;S_{\eps_n})
			= \liminf_{n\to\infty} \sum_{i=1}^5 \abs*{u(z_i^{\eps_n})-u(z_{i-1}^{\eps_n})}
			= \liminf_{n\to\infty} \sum_{i=1}^5 \abs*{u(y_{\sigma(i)}^{\eps_n})-u(y_{\sigma(i-1)}^{\eps_n})} \\
			&= \lim_{n\to\infty} \sum_{i=1}^5 \abs*{a_{\sigma(i)}^{\eps_n}-a_{\sigma(i-1)}^{\eps_n}}
			= \sum_{i=1}^5 \abs*{a_{\sigma(i)}-a_{\sigma(i-1)}}
			> \max_i a_i - \min_i a_i
			= \sup_{(a,b)} u - \inf_{(a,b)} u,
		\end{split}
	\]
	therefore the inequality in \eqref{eq:tv-sup-inf} is strict.
\end{proof}

\begin{remark}\label{rmk:rearrangement-ot}
	Inequality \eqref{eq:monotone-rearrangement-contraction-Lp} in the previous
	\autoref{prop:rearrangement-contraction} is closely related to the theory of optimal transport on
	the real line with a convex cost.
	Indeed, introduce the probabilities $\mu=u_\#\bigl(\leb^1\res(0,1)\bigr),
		\nu=v_\#\bigl(\leb^1\res(0,1)\bigr) \in \Prob\bigl([0,\infty)\bigr)$.
	We claim that $\mu=(u_\dagger)_\#\bigl(\leb^1\res(0,1)\bigr)$, and similarly for $\nu$.
	It is sufficient to prove that the measures coincide on the set $(t,\infty)$ for every $t>0$.
	We have
	\[
		u^{-1}\bigl((t,\infty)\bigr) = \set{x\in(0,1)}{u(x)>t}
	\]
	and
	\[
		\begin{split}
			(u_\dagger)^{-1}\bigl((t,\infty)\bigr)
			&= \set{s\geq0}{u_\dagger(s)>t}
			= \set{s\geq0}{(\mu_u)^{-1}(s)>t} \\
			&= \set{s\geq0}{\mu_u(t)>s}
			= \bigl[0,\mu_u(t)\bigr),
		\end{split}
	\]
	therefore
	\[
		\mu\bigl((t,\infty)\bigr)
		= \leb^1\bigl(u^{-1}(t,\infty)\bigr)
		= \leb^1\bigl(\set{x\in(0,1)}{u(x)>t}\bigr)
		= \mu_u(t)
	\]
	and
	\[
		(u_\dagger)_\#\bigl(\leb^1\res(0,1)\bigr)\bigl((t,\infty)\bigr)
		= \leb^1\bigl((u_\dagger)^{-1}\bigl((t,\infty)\bigr)\bigr)
		= \leb^1\bigl([0,\mu_u(t))\bigr)
		= \mu_u(t).
	\]

	Define now the transport plans $\pi,\pi_\dagger \in\Pi(\mu,\nu)$ given by
	\begin{align*}
		\pi         & = (u,v)_\#\bigl(\leb^1\res(0,1)\bigr),                 &
		\pi_\dagger & = (u_\dagger,v_\dagger)_\#\bigl(\leb^1\res(0,1)\bigr).
	\end{align*}
	The plan $\pi_\dagger$ is supported on a monotone graph%
	\footnote{A set $\Gamma\subset\setR^2$ such that $(x_2-x_1)(y_2-y_1)\geq0$ for every
		$(x_1,y_1),(x_1,y_2)\in\Gamma$.},
	hence it is $c$-cyclically monotone with respect to the cost $c(x,y)=\abs{x-y}^p$,
	therefore it is optimal
	(see \cite[Proposition 4.5]{AmbrosioBrueSemola} where the converse is proven, although it is
	actually an equivalence; observe that their inequality (4.5) there is related to our
	\autoref{lem:two-point-inequality}; see also \cite[Theorem 1.11]{AmbrosioBrueSemola} for the
	optimality of the monotone rearrangement). This implies that
	\[
		\begin{split}
			\norm{u_\dagger-v_\dagger}_p^p
			&= \int_0^1 \abs{u_\dagger(z)-v_\dagger(z)}^p \dz
			= \int_{[0,\infty)^2} \abs{x-y}^p \d\pi_\dagger(x,y) \\
			&\leq \int_{[0,\infty)^2} \abs{x-y}^p \d\pi(x,y)
			= \int_0^1 \abs{u(z)-v(z)}^p \dz
			= \norm{u-v}_p^p.
		\end{split}
	\]
	With regard to the equality condition,
	if the inequality \eqref{eq:monotone-rearrangement-contraction-Lp} is saturated and $p>1$,
	then the plan $\pi$ must be supported on a monotone graph too,
	which translates to \eqref{eq:monotone-rearrangement-contraction-Lp-eq-cond}.
\end{remark}

\begin{lemma}\label{lem:two-point-inequality}
	If $f:\setR\to\setR$ is a convex function and $a\leq A$ and $b\leq B$, then
	\[
		f(b-a) + f(B-A) \leq f(B-a) + f(b-A).
	\]
	Moreover, the inequality is strict if $f$ is strictly convex, $a<A$ and $b<B$.

	In particular, for every $p\geq1$ we have $\abs{b-a}^p + \abs{B-A}^p \leq \abs{B-a}^p + \abs{b-A}^p$.
	If $p>1$ the inequality is strict if and only if $a<A$, and $b<B$.
	For $p=1$, the inequality is strict if and only if $a<A$, $b<B$, $a<B$ and $b<A$, which is
	equivalent to $\max\{a,b\}<\min\{A,B\}$.
\end{lemma}

\begin{proof}
	Let $x\coloneqq b-A\leq b-a\eqqcolon X$ and $\Delta b=B-b\geq 0$.
	The convexity of $f$ implies the monotonicity of forward differences, therefore we have
	\[
		f(x+\Delta b)-f(x) \leq f(X+\Delta b)-f(X),
	\]
	that is
	\[
		f(B-A)-f(b-A) \leq f(B-a) - f(b-a),
	\]
	which is equivalent to the thesis.
	The inequality is strict if $f$ is strictly convex, $x<X$ and $\Delta b>0$, which are the stated conditions.

	The inequality for $p\geq1$ and the characterization of strict inequality for $p>1$ follow using
	$f(t)=\abs{t}^p$.

	Let us turn to the characterization of strict inequality for $p=1$, $f(t)=\abs{t}$.
	If $x=X$ or $\Delta b=0$, then equality holds. Let $x<X,\Delta b>0$.
	By the monotonicity of difference quotients,
	\[
		\frac{f(x+\Delta b)-f(x)}{\Delta b}
		\leq \frac{f(X+\Delta b)-f(x)}{X+\Delta b-x}
		\leq \frac{f(X+\Delta b)-f(X)}{\Delta b},
	\]
	where equality is verified if and only if $f$ is affine in $(x,X+\Delta b)$,
	Therefore, the strict inequality holds if and only if $0\in(x,X+\Delta)$, which amounts to the
	required property on $a,A,b,B$.
\end{proof}

The next lemma is a classical rearrangement inequality, whose proof can be found in any textbook
about optimal transport, in relation to the $c$-cyclical monotonicity with respect to the cost $\abs{x-y}$,
see e.g.\ \cite{AmbrosioBrueSemola}.

\begin{lemma}\label{lem:c-monotonicity}
	Let $n\in\setN_+$ and let $a=(a_i)_{i=0}^n$ be a finite sequence of real numbers.
	For every permutation $\sigma\in\mathcal{P}_n$ we have
	\[
		\sum_{i=1}^n \abs{a_{\sigma(i)}-a_{\sigma(i-1)}}
		\geq \max_i a_i - \min_i a_i,
	\]
	with equality if and only if $(a_{\sigma(i)})_{i=0}^n$ is monotone.
\end{lemma}

\begin{proposition}\label{prop:global-truncation}
	Let $g\in\BV\bigl((a,b)\bigr)$ and let $\lambda,\gamma>0$, $p\in(0,\infty]$.
	If $u\in\BV\bigl((a,b)\bigr)$ minimizes the functional
	\[
		F(u) = \TV(u) + \lambda\norm{u-g}_p^\gamma
	\]
	and $m\leq g\leq M$ then also $m\leq u\leq M$.
\end{proposition}

\begin{proof}
	Let $\tilde u = u\vee m\wedge M$.
	Then $\TV(\tilde u)\leq\TV(u)$ by \autoref{rmk:BV-truncation}, and $\norm{\tilde u-g}_p\leq\norm{u-g}_p$,
	with this latter inequality being strict unless $m\leq u\leq M$.
\end{proof}

\begin{proposition}\label{prop:increasing-minimizer}
	Let $g\in\BV\bigl((a,b)\bigr)$ be a non-decreasing function and $\lambda,\gamma>0$, $p\in[1,\infty]$.
	If $u\in\BV\bigl((a,b)\bigr)$ minimizes the functional
	\[
		F(u) = \TV(u) + \lambda\norm{u-g}_p^\gamma
	\]
	then it is non-decreasing.
\end{proposition}

\begin{proof}
	If $g$ is constant, then the only minimizer is $u=g$, which achieves $F(u)=0$.
	Indeed, if $u$ is any other constant function, then $\norm{u-g}_p>0$; otherwise, if $u$ is not
	constant, then $\TV(u)>0$.

	We may now assume that $g$ is not constant; therefore, there exist $x_1<x_2$ such that
	$g(x_1)<g(x_2)$.
	Let $\bar u(x)=u^\dagger(x-a)$ be the non-decreasing rearrangement of $u$ suitably translated to
	preserve the same domain of definition. Since $g$ is already non-decreasing, we have
	$g(x)=g^\dagger(x-a)$.
	By \eqref{eq:monotone-rearrangement-contraction-Lp} and \eqref{eq:monotone-rearrangement-contraction-TV}
	of \autoref{prop:rearrangement-contraction} we have
	\begin{gather*}
		\norm{\bar u-g}_{L^p((a,b))}
		= \norm{u^\dagger-g^\dagger}_{L^p((0,b-a))}
		\leq \norm{u-g}_{L^p((a,b))}, \\
		\TV\bigl(\bar u;(a,b)\bigr)
		= \TV\bigl(u^\dagger;(0,b-a)\bigr)
		\leq \TV\bigl(u;(a,b)\bigr),
	\end{gather*}
	and both inequalities must be saturated.
	By the equality conditions stated in \autoref{prop:rearrangement-contraction},
	the latter forces $u$ to be monotone. We wish to prove that $u$ is actually non-decreasing.
	By contradiction, suppose there exist $x_3<x_4$ such that $u(x_3)>u(x_4)$. Since $u$ is monotone,
	this forces it to be non-increasing. Let $x'=\min\{x_1,x_3\}$ and $x''=\max\{x_2,x_4\}$.
	Then for any points $y_1\leq x'$ and $y_2\geq x''$ we have
	\begin{gather*}
		g(y_1) \leq g(x') \leq g(x_1) < g(x_2) \leq g(x'') \leq g(y_2), \\
		u(y_1) \geq u(x') \geq u(x_1) > u(x_2) \geq u(x'') \geq u(y_2),
	\end{gather*}
	from which $\bigl(u(y_2)-u(y_1)\bigr)\bigl(g(y_2)-g(y_1)\bigr)<0$.
	In the case $p>1$,
	this fact contradicts the equality condition \eqref{eq:monotone-rearrangement-contraction-Lp-eq-cond}.

	We now treat separately the case $p=1$. By \autoref{prop:global-truncation}, we have
	\[
		m \coloneqq \inf_{(a,b)} g \leq u \leq \sup_{(a,b)} g \eqqcolon M.
	\]
	In the limit $y_1\to a^+$ and $y_2\to b^-$ we have $g(y_1)<u(y_1)$ and $u(y_2)<g(y_2)$ because
	\[
		\begin{split}
			\lim_{y_1\to a^+} \bigl(u(y_1)-g(y_1)\bigr)
			&= \lim_{y_1\to a^+} \bigl[ \bigl(u(y_2)-g(y_1)\bigr) + \bigl(u(y_1)-u(y_2)\bigr) \bigr] \\
			&\geq \bigl(u(x_1)-u(x_2)\bigr) + \lim_{y_1\to a^+} \bigl(u(y_2)-g(y_1)\bigr) \\
			&\geq \bigl(u(x_1)-u(x_2)\bigr) + \bigl(u(y_2)-m\bigr)
			\geq u(x_1)-u(x_2)
			> 0
		\end{split}
	\]
	and similarly
	\[
		\lim_{y_2\to b^-} \bigl(g(y_2)-u(y_2)\bigr)
		\geq u(x_1)-u(x_2)
		> 0.
	\]
	Therefore there is a non-negligible set of points $y_1,y_2$ which violate the inequality
	\eqref{eq:monotone-rearrangement-contraction-L1-eq-cond}.
\end{proof}

The next lemma contains the explicit calculation of an integral which will be useful to estimate the
fidelity term from below in \autoref{lem:fidelity lower bound}.

\begin{lemma}\label{lem:holder-integral}
	Let $a\in\setR$ and $h,C,\alpha\in\setR_+$. Then
	\[
		\int_a^{a+(h/C)^{1/\alpha}} \bigl(h-C(x-a)^\alpha\bigr)^p \dx
		= K(p,\alpha) C^{-1/\alpha} h^{p+1/\alpha},
	\]
	where
	\[
		K(p,\alpha) = \frac{\Beta(p+1,1/\alpha)}\alpha
		= \frac{\Gamma(p+1)\Gamma(1/\alpha)/\alpha}{\Gamma(p+1/\alpha+1)}
		= \frac{\Gamma(p+1)\Gamma(1/\alpha+1)}{\Gamma(p+1/\alpha+1)}
		= \binom{p+1/\alpha}{p}^{-1}
	\]
	and
	\begin{equation}\label{eq:beta}
		\Beta(x,y) = \int_0^1 t^{x-1}(1-t)^{1-y} \d t = \Gamma(x)\Gamma(y)/\Gamma(x+y)
	\end{equation}
	is the Beta function.
\end{lemma}

\begin{proof}
	After a translation by $a$ and with the substitution $y=\frac Ch x^\alpha$,
	$\dx=\frac{(h/C)^{1/\alpha}}\alpha y^{1/\alpha-1}\dy$, we have
	\[
		\begin{split}
			\int_a^{a+(h/C)^{1/\alpha}} \bigl(h-C(x-a)^\alpha\bigr)^p \dx
			&= \int_0^{(h/C)^{1/\alpha}} (h-Cx^\alpha)^p \dx
			= h^p \int_0^{(h/C)^{1/\alpha}} \left(1-\frac Ch x^\alpha\right)^p \dx \\
			&= \frac{h^p(h/C)^{1/\alpha}}\alpha \int_0^1 (1-y)^p y^{1/\alpha-1} \dy \\
			&= \frac{h^p(h/C)^{1/\alpha}}\alpha \Beta(p+1,1/\alpha). \qedhere
		\end{split}
	\]
\end{proof}

\begin{lemma}[Fidelity lower bound]\label{lem:fidelity lower bound}
	Let $g:(a,b)\to\setR$ be a measurable function
	for which there exists $C>0$ such that $g(x)-g^+(a)\leq C(x-a)^\alpha$ for every $x\in(a,b)$,
	and let $u:(a,b)\to\setR$ be a measurable function such that
	\[
		h \coloneqq \inf_{(a,b)}u - g^+(a) \geq 0.
	\]
	Then
	\[
		\norm{u-g}_{L^p\bigl((a,b)\bigr)}^p
		\geq K(p,\alpha) C^{-1/\alpha} h^{p+1/\alpha},
	\]
	where
	\[
		K(p,\alpha) = \frac{\Beta(p+1,1/\alpha)}\alpha
		= \frac{\Gamma(p+1)\Gamma(1/\alpha)/\alpha}{\Gamma(p+1/\alpha+1)}
		= \frac{\Gamma(p+1)\Gamma(1/\alpha+1)}{\Gamma(p+1/\alpha+1)}
		= \binom{p+1/\alpha}{p}^{-1}.
	\]
	If $g$ satisfies also $g(x)-g^-(b)\geq -C(b-x)^\alpha$ for every $x\in(a,b)$ and
	\[
		k \coloneqq g^-(b) - \sup_{(a,b)}u \geq 0,
	\]
	then
	\begin{equation}\label{eq:fidelity lower bound}
		\norm{u-g}_{L^p\bigl((a,b)\bigr)}^p
		\geq K(p,\alpha) C^{-1/\alpha} \bigl(h^{p+1/\alpha}+k^{p+1/\alpha}\bigr).
	\end{equation}
\end{lemma}

\begin{proof}
	In the interval $\bigl(a,a+(h/C)^{1/\alpha}\bigr)$ we have
	\[
		g(x) \leq g^+(a) + C(x-a)^\alpha \leq g^+(a) + h \leq u(x),
	\]
	therefore, by \autoref{lem:holder-integral},
	\[
		\begin{split}
			\norm{u-g}_{L^p((a,b))}^p
			&\geq \int_a^{a+(h/C)^{1/\alpha}} \bigl(u(x)-g(x)\bigr)^p \dx
			\geq \int_a^{a+(h/C)^{1/\alpha}} \bigl(h-C(x-a)^\alpha\bigr)^p \dx \\
			&\geq K(p,\alpha) C^{-1/\alpha} h^{p+1/\alpha}.
		\end{split}
	\]
	Suppose that the additional assumptions on $g$ and $u$ are satisfied.
	Letting $U(x)=g^+(a)+C(x-a)^\alpha$ and $L(x)=g^-(b)-C(b-x)^\alpha$, we have
	\[
		U\bigl(a+(h/C)^\alpha\bigr)
		= \inf u \leq \sup u
		= L(b-(k/C)^{1/\alpha})
		\leq g(b-(k/C)^{1/\alpha})
		\leq U(b-(k/C)^{1/\alpha}),
	\]
	hence $a+(h/C)^{1/\alpha} \leq b-(k/C)^{1/\alpha}$.
	The interval $\bigl(b-(k/C)^{1/\alpha},b\bigr)$ is disjoint from the previous one and in it we have
	\[
		u(x) \leq g^-(b) - k \leq g^-(b) - C(b-x)^\alpha \leq g(x),
	\]
	therefore
	\[
		\begin{split}
			\norm{u-g}_{L^p((a,b))}^p
			&\geq \int_a^{a+(h/C)^{1/\alpha}} \bigl(u(x)-g(x)\bigr)^p \dx
			+ \int_{b-(k/C)^{1/\alpha}}^b \bigl(g(x)-u(x)\bigr)^p \dx \\
			&\geq \int_a^{a+(h/C)^{1/\alpha}} \bigl(h-C(x-a)^\alpha\bigr)^p \dx
			+ \int_{b-(k/C)^{1/\alpha}}^b \bigl(k-C(b-x)^\alpha\bigr)^p \dx \\
			&\geq K(p,\alpha) C^{-1/\alpha} \bigl(h^{p+1/\alpha}+k^{p+1/\alpha}\bigr). \qedhere
		\end{split}
	\]
\end{proof}

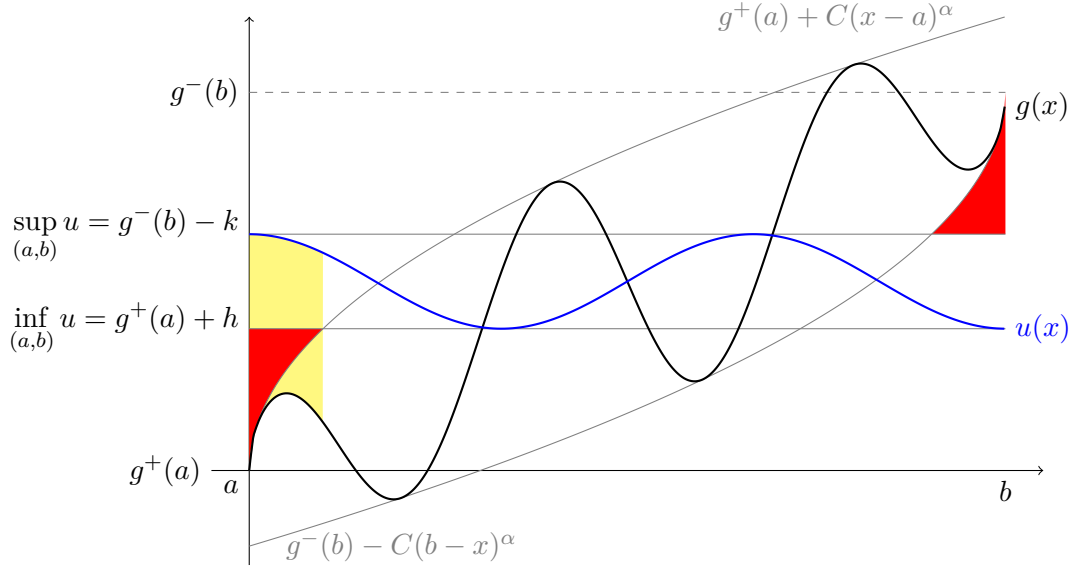
\begin{figure}
	\centering
	\begin{tikzpicture}[xscale=10,yscale=5]
		\def\ymax{1.2}
		\def\X{9/64/\ymax/\ymax}

		\draw[->]
		(-.05,0) node[left] {$g^+(a)$} -- (0,0) node[below left] {$a$} -- (1,0) node[below] {$b$} -- (1.05,0);
		\draw[->] (0,1-\ymax-.05) -- (0,\ymax);

		\begin{scope}

			\fill[yellow,opacity=0.5]
			plot[domain=0:\X,samples=60]
			(\x, {(1+\ymax*(sqrt(\x)-sqrt(1-\x)))/2 + cos(5*180*\x)*(-1+\ymax*(sqrt(\x)+sqrt(1-\x)))/2})
			-- plot[domain=0:\X,samples=60] (\X-\x, {1/2+cos(3*180*(\X-\x))/8})
			-- cycle;

			\fill[red,opacity=1.5]
			plot[domain=0:\X,samples=60] (\x, {\ymax*sqrt(\x)}) -- (0,3/8) -- cycle
			plot[domain=0:\X,samples=60] (1-\x,{1-\ymax*sqrt(\x)}) -- (1,5/8) -- cycle;

			\draw[gray,domain=0:1,samples=200]
			plot (\x, {\ymax*sqrt(\x)}) node[left=15] {$g^+(a)+C(x-a)^\alpha$}
			plot (\x, {1-\ymax*sqrt(1-\x)}) (0,1-\ymax) node[right=10] {$g^-(b)-C(b-x)^\alpha$};

			\draw[gray]
			(0,3/8) node[left,black] {$\displaystyle\inf_{(a,b)}u=g^+(a)+h$} -- (1,3/8)
			(0,5/8) node[left,black] {$\displaystyle\sup_{(a,b)}u=g^-(b)-k$} -- (1,5/8);
			\draw[gray,dashed]
			(0,1) node[left,black] {$g^-(b)$} -- (1,1);

			\draw[black,thick,domain=0:1,samples=200,smooth]
			plot (\x, {(1+\ymax*(sqrt(\x)-sqrt(1-\x)))/2 + cos(5*180*\x)*(-1+\ymax*(sqrt(\x)+sqrt(1-\x)))/2})
			node[right] {$g(x)$};

			\draw[blue,thick,domain=0:1,samples=100,smooth]
			plot (\x, {1/2+cos(3*180*\x)/8})
			node[right] {$u(x)$};

		\end{scope}
	\end{tikzpicture}
	\caption{depiction of the situation described in \autoref{lem:fidelity lower bound}.}
	\label{fig:fidelity lower bound}
\end{figure}

\begin{definition}\label{def:p-mean}
	Given $n\in\setN_+$, an exponent $p>0$ and non-negative real numbers $x_1,\dots,x_n\in[0,\infty)$, their
	$p$-mean is
	\[
		M_p(x_1,\dots,x_n) = \left(\frac1n \sum_{i=1}^n x_i^p\right)^{1/p}.
	\]
\end{definition}

If is well known that $M_p\leq M_q$ for $p<q$.
Moreover, clearly $M_p(x_1,\dots,x_n)=0$ if and only if $x_1=\dots=x_n=0$.

\subsection{Perspective of a convex function}\label{sec:perspective}

Let $\phi:[0,\infty)\to\setR$ be a convex function which is continuous in $0$.\footnote{
	One may extend it to $\phi:\setR\to\setR\cup\{\infty\}$ by setting $\phi\rvert_{(-\infty,0)}=\infty$,
	but of course this extension is no longer continuous in $0$.}

The \emph{perspective} of $\phi$ is $\phi:[0,\infty)\times(0,\infty)\to\setR$ given by $\phi(t,s)=s\phi(t/s)$.
It is clearly continuous, and it is also convex (see \cite[Section~3.2.6]{boyd_convex_2023}).

It is well known that the difference quotient function $t\mapsto\frac{\phi(t)-\phi(0)}t$ is non-decreasing.
As a consequence, we can define the \emph{asymptotic slope}
\[
	\phi'(\infty)
	= \lim_{t\to\infty} \frac{\phi(t)}t
	= \lim_{t\to\infty} \frac{\phi(t)-\phi(0)}t
	= \sup_{t>0} \frac{\phi(t)-\phi(0)}t
	\in \setR\cup\{\infty\}
\]
and the \emph{recession function} $\phi^\infty:[0,\infty)\to(-\infty,\infty]$
\[
	\phi^\infty(t)
	= \lim_{r\to\infty} \frac{\phi(rt)}r
	= \lim_{r\to\infty} \frac{\phi(rt)-\phi(0)}r
	= \sup_{r>0} \frac{\phi(rt)-\phi(0)}r
	\in \setR\cup\{\infty\}.
\]

It is immediate to verify that $\phi^\infty(t)=\phi'(\infty)t$ for every $t\in[0,\infty)$,
with the usual convention $\infty\cdot0=0$.
Indeed for $t=0$ we have $\phi^\infty(0)=0=\phi'(\infty)0$, and for $t>0$ we have
\[
	\phi^\infty(t)
	= \lim_{r\to\infty} \frac{\phi(rt)}r
	= t \lim_{r\to\infty} \frac{\phi(rt)}{rt}
	= t \lim_{r\to\infty} \frac{\phi(r)}r
	= t \phi'(\infty).
\]

Moreover, if $\phi(0)=0$, then $\phi(t,s)$ is non-increasing in $s$: indeed,
\[
	\frac1t \phi(t,s) = \frac st \phi(t/s) = \frac{\phi(t/s)-\phi(0)}{t/s}
\]
is non-decreasing in $t/s$, therefore non-increasing in $s$.

\begin{lemma}\label{lem:perspective-legendre}
	Using the notation $\phi_s=\phi(\plchldr,s)$ to denote the perspective, we have the following
	rules of commutation with the Legendre transform: for every $s>0$
	\begin{align*}
		(\phi_s)^* & = s\phi^*,   &
		(\phi^*)_s & = (s\phi)^*.
	\end{align*}
\end{lemma}

\begin{proof}
	Straightforward computations show that, for every $\tau\in\setR$,
	\[
		\begin{split}
			(\phi_s)^*(\tau)
			&= \sup_{t\geq0} \bigl(\tau t-\phi_s(t)\bigr)
			= \sup_{t\geq0} \bigl(\tau t-s\phi(t/s)\bigr) \\
			&= s \cdot \sup_{t\geq0} \bigl(\tau t/s-\phi(t/s)\bigr)
			= s \cdot \sup_{u\geq0} \bigl(\tau u-\phi(u)\bigr)
			= s  \phi^*(\tau)
		\end{split}
	\]
	and
	\[
		(\phi^*)_s(\tau)
		= s \phi^*(\tau/s)
		= s \cdot \sup_{t\geq0} \bigl((\tau/s)t-\phi(t)\bigr)
		= \sup_{t\geq0} \bigl(\tau t-s\phi(t)\bigr)
		= (s\phi)^*(\tau). \qedhere
	\]
\end{proof}

\begin{proposition}\label{prop:perspective}
	If $\phi'(\infty)\in\setR$ is finite, then the perspective $\phi:[0,\infty)\times(0,\infty)\to\setR$
	can be extended continuously to $\phi:[0,\infty)\times[0,\infty)\to\setR$ setting
	$\phi(t,0)=\phi^\infty(t)=\phi'(\infty)t$ and this extension is convex.
\end{proposition}

\begin{proof}[``Specific'' proof]\footnote{This proof uses more specific properties of the perspective.}
	Since for every $t\in[0,\infty)$ we have the directional limit
	\[
		\lim_{s\to0^+} \phi(t,s)
		= \lim_{s\to0^+} s\phi(t/s)
		= \lim_{r\to\infty} \phi(rt)/r
		= \phi^\infty(t),
	\]
	setting $\phi(t,0)=\phi^\infty(t)=\phi'(\infty)t$ is the only possible choice that may render the
	extension continuous. Let us verify that it is indeed the case.

	From $\frac{\phi(t)-\phi(0)}t \leq \phi'(\infty)$ we get the upper bound
	$\phi(t) \leq \phi'(\infty)t+\phi(0)$, valid for every $t\in[0,\infty)$.

	We claim that the Legendre transform $\phi^*(\sigma)\in\setR$ is finite for every
	$\sigma\in\bigl(-\infty,\phi'(\infty)\bigr)$.\footnote{
		Notice that in general it may be $\phi^*\bigl(\phi'(\infty)\bigr)=\infty$. This happens when
		$\phi'(\infty)t-\phi(t)$ does not go to zero as $t\to\infty$, for instance $\phi(t)=o(t)$ but
		$\phi(t)\not\to0$. An explicit example is $\phi(t)=-\sqrt t$.}
	Indeed, fix $T\in[0,\infty)$ such that
	$\frac{\phi(T)-\phi(0)}T \geq \sigma$. Then for every $t\in[T,\infty)$ we have
	$\frac{\phi(t)-\phi(0)}t \geq \frac{\phi(T)-\phi(0)}T \geq \sigma$, hence
	$\phi(t) \geq \sigma t+\phi(0)$. This implies that
	\[
		\begin{split}
			\phi^*(\sigma)
			&= \sup_{t\in[0,\infty)} \bigl(\sigma t-\phi(t)\bigr)
			= \max\left\{ \sup_{t\in[0,T]} \bigl(\sigma t-\phi(t)\bigr),
			\sup_{t\in[T,\infty)} \bigl(\sigma t-\phi(t)\bigr)\right\} \\
			&= \max\left\{ \max_{t\in[0,T]} \bigl(\sigma t-\phi(t)\bigr),
			\sup_{t\in[T,\infty)} \bigl(-\phi(0)\bigr)\right\}
			< \infty
		\end{split}
	\]
	is finite. This fact provides us with the lower bound $\phi(t) \geq \sigma t-\phi^*(\sigma)$,
	valid for every $t\in[0,\infty)$.

	The lower bound and upper bound can be combined, evaluated at $t/s$ and multiplied by $s$, yielding
	\[
		\sigma t - \phi^*(\sigma)s \leq s\phi(t/s) \leq \phi'(\infty)t + \phi(0)s,
		\qquad \forall (t,s)\in[0,\infty)\times(0,\infty).
	\]
	Notice that both the lower and upper bounds are continuous functions of $t$ and $s$.
	For a fixed $t_0\in[0,\infty)$, we get
	\[
		\limsup_{\substack{t\to t_0 \\ s\to0^+}} \phi(t,s)
		= \limsup_{\substack{t\to t_0 \\ s\to0^+}} s\phi(t/s)
		\leq \limsup_{\substack{t\to t_0 \\ s\to0^+}} \bigl(\phi'(\infty)t+\phi(0)s\bigr)
		= \phi'(\infty)t_0
	\]
	and
	\[
		\liminf_{\substack{t\to t_0 \\ s\to0^+}} \phi(t,s)
		= \liminf_{\substack{t\to t_0 \\ s\to0^+}} s\phi(t/s)
		\geq \liminf_{\substack{t\to t_0 \\ s\to0^+}} \bigl(\sigma t - \phi^*(\sigma)s\bigr)
		= \sigma t_0.
	\]
	By the arbitrariness of $\sigma<\phi'(\infty)$ we deduce that the $\liminf$ is actually greater
	than or equal to $\phi'(\infty)t_0$, which implies
	\[
		\lim_{\substack{t\to t_0 \\ s\to0^+}} \phi(t,s) = \phi'(\infty)t_0 = \phi(t_0,0),
	\]
	hence the extension is continuous.

	It is a general fact that a continuous extension is necessarily convex.
\end{proof}

\begin{proof}[``General'' proof]\footnote{This proof uses only the existence of the directional limit.}
	Since for every $t\in[0,\infty)$ we have the directional limit
	\[
		\lim_{s\to0^+} \phi(t,s)
		= \lim_{s\to0^+} s\phi(t/s)
		= \lim_{r\to\infty} \phi(rt)/r
		= \phi^\infty(t),
	\]
	setting $\phi(t,0)=\phi^\infty(t)=\phi'(\infty)t$ is the only possible choice that may render the
	extension continuous.

	First of all, observe that the extension is convex. To this end, let
	$(t_0,s_0),(t_1,s_1)\in[0,\infty)\times[0,\infty)$, $\tau\in[0,1]$ and
	$(t_\tau,s_\tau)=(1-\tau)(t_0,s_0)+\tau(t_1,s_1)$. For every $\eps>0$ we have
	$(t_0,s_0+\eps),(t_1,s_1+\eps)\in[0,\infty)\times(0,\infty)$ and
	$(t_\tau,s_\tau+\eps)=(1-\tau)(t_0,s_0+\eps)+\tau(t_1,s_1+\eps)$, therefore
	\[
		\phi(t_\tau,s_\tau+\eps) \leq (1-\tau)\phi(t_0,s_0+\eps) + \tau\phi(t_1,s_1+\eps).
	\]
	Sending $\eps\to0^+$ and using the directional continuity in $s$ yields
	\[
		\phi(t_\tau,s_\tau) \leq (1-\tau)\phi(t_0,s_0) + \tau\phi(t_1,s_1),
	\]
	hence the extension is convex.

	Let us verify that it is continuous.
	Notice that the extension $\phi(t,0)=\phi'(\infty)t$ is a continuous function of $t$.
	Fix $t_0\in(0,\infty)$ and $\eps>0$. There exists $\delta>0$ such that $t_0-\delta\geq0$ and
	\[
		\abs{\phi(t_0\pm\delta,0)-\phi(t_0,0)} \leq \eps/2.
	\]
	By the continuity in $s$, there exists $\eta>0$ such that for every $s\in[0,\eta]$ we have
	\begin{align*}
		\abs{\phi(t_0,s)-\phi(t_0,0)}                   & \leq \eps/2, &
		\abs{\phi(t_0\pm\delta,s)-\phi(t_0\pm\delta,0)} & \leq \eps/2.
	\end{align*}
	By the triangle inequality we deduce also $\abs{\phi(t_0\pm\delta,s)-\phi(t_0,0)}\leq\eps$.
	Let $(t,s)\in[t_0-\delta,t_0+\delta]\times[0,\eta]$ be an arbitrary point in this neighborhood of $(t_0,0)$.
	By convexity we have the upper bound
	\[
		\phi(t,s) \leq \max\{\phi(t_0\pm\delta,0),\phi(t_0\pm\delta,\eta)\} \leq \phi(t_0,0) + \eps.
	\]
	To find a lower bound, assume that $t_0\leq t\leq t_0+\delta$;
	the symmetric case $t_0-\delta\leq t\leq t_0$ is analogous.
	We can write $t_0=(1-\tau)(t_0-\delta)+\tau t$ with $\tau=\delta/(t-t_0+\delta)\in[1/2,1]$, therefore
	\[
		\phi(t_0,s) \leq (1-\tau)\phi(t_0-\delta,s) + \tau\phi(t,s),
	\]
	from which
	\[
		\begin{split}
			\phi(t,s)
			&\geq \frac1\tau \phi(t_0,s) - \frac{1-\tau}\tau \phi(t_0-\delta,s)
			\geq \frac1\tau \bigl(\phi(t_0,0)-\eps\bigr) - \frac{1-\tau}\tau \bigl(\phi(t_0,0)+\eps\bigr) \\
			&= \phi(t_0,0) - \left(1-\frac2\tau\right)\eps
			\geq \phi(t_0,0) - 3\eps.
		\end{split}
	\]
	Combined with the upper bound, this shows that $\abs{\phi(t,s)-\phi(t_0,0)}\leq 3\eps$ for every
	$(t,s)\in[t_0-\delta,t_0+\delta]\times[0,\eta]$, proving the continuity in $(t_0,0)$, for
	$t_0\in(0,\infty)$.

	The continuity in $(0,0)$ is proved similarly. Given $\eps>0$ there exists $\delta>0$ such that
	\begin{align*}
		\abs{\phi(\delta,0)-\phi(0,0)}  & \leq \eps/2, &
		\abs{\phi(2\delta,0)-\phi(0,0)} & \leq \eps/2.
	\end{align*}
	By the continuity in $s$, there exists $\eta>0$ such that for every $s\in[0,\eta]$ we have
	\begin{align*}
		\abs{\phi(0,s)-\phi(0,0)}             & \leq \eps/2, &
		\abs{\phi(\delta,s)-\phi(\delta,0)}   & \leq \eps/2, &
		\abs{\phi(2\delta,s)-\phi(2\delta,0)} & \leq \eps/2.
	\end{align*}
	By the triangle inequality we infer $\abs{\phi(i\delta,s)-\phi(0,0)}\leq\eps$ for $i=0,1,2$.
	If $(t,s)\in[0,2\delta]\times[0,\eta]$ is an arbitrary point in this neighborhood of the origin,
	with a similar argument as before we get $\abs{\phi(t,s)-\phi(0,0)}\leq3\eps$, hence proving the
	continuity.
\end{proof}

\section{Linear \texorpdfstring{$\TV$}{TV}}\label{sec:linear}

In this section we study the minimization of the functional
\[
	F(u) = \TV(u) + \lambda\norm{u-g}_p^\gamma
\]
with linear total variation and fidelity with respect to an input datum $g$.
We classify situations where the minimizer $u$ is not the input $g$ itself (\autoref{sec:linear-not-min})
and situations where the minimizer $u$ is indeed the input $g$ (\autoref{sec:linear-min}).

Our results are based on some conditions on $g$ ($\alpha$-growth in \autoref{def:alpha-growth},
monotonicity, piecewise monotonicity) and on specific ranges of the parameters $p,\gamma,\lambda$.
The complementarity of these ranges provides a full characterization and demonstrates the sharpness
of our results (see \autoref{rmk:sharpness}).

\subsection{The input is not a minimizer}\label{sec:linear-not-min}

We start with a model situation where the input $g$ is taken to be the Cantor--Lebesgue--Vitali function
$\chi\in C([0,1];[0,1])$.
Subsequently, this example will be generalized in \autoref{thm:not-min-alpha-growth} to functions
with at least $\alpha$-growth at one endpoint (see \autoref{def:alpha-growth}).

\begin{theorem}\label{thm:not-min-cantor-function}
	Let $\chi\in C([0,1];[0,1])$ be the Cantor--Lebesgue--Vitali function and $\lambda,\gamma>0$, $p\in[1,\infty]$.
	If
	\[
		\gamma > \left(\frac1p\cdot\frac{\log3}{\log2}+1\right)^{-1}
	\]
	then the function $\chi$ is not a minimizer of the functional
	$F(u) = \TV(u) + \lambda\norm{u-\chi}_p^\gamma.$
\end{theorem}

\begin{proof}
	Given $n\in\setN_+$ let $\chi_n=2^{-n}\vee\chi\wedge(1-2^{-n})\in C([0,1];[2^{-n},1-2^{-n}])$. We have $F(\chi)=\TV(\chi)=1$ and $\TV(\chi_n)=1-2^{-n+1}$. Moreover, for $p=\infty$ we have $\norm{\chi_n-\chi}_p = 2^{-n}$, whereas for $p\in[1,\infty)$ we have
	\[
		\begin{split}
			\norm{\chi_n-\chi}_p^p
			&= 2\int_0^{3^{-n}} \chi(x)^p \d x
			\leq 2\int_0^{3^{-n}} \chi(x)^{p-1} \chi(x) \d x
			\leq 2\norm{\chi}_{L^\infty([0,3^{-n}])}^{p-1} \int_0^{3^{-n}} \chi(x) \d x \\
			&= 2\cdot 2^{-n(p-1)} \frac12 2^{-n}3^{-n}
				= 2^{-np}3^{-n},
		\end{split}
	\]
	hence $\norm{\chi_n-\chi}_p\leq 2^{-n}3^{-n/p}$ in both cases (interpreting $-n/p=0$ for $p=\infty$). From this we get that
	\[
		F(\chi_n) \leq 1-2^{-n+1} + \lambda 2^{-n\gamma} 3^{-n\gamma/p},
	\]
	therefore $F(\chi_n)<F(\chi) \Leftrightarrow \lambda 2^{-n\gamma} 3^{-n\gamma/p}<2^{-n+1}$, i.e.\ if and only if
	\[
		\frac\lambda2
		< 2^{(\gamma-1)n} 3^{\frac\gamma{p}n}
		= \exp\oleft(\left((\gamma-1)\log2+\frac\gamma{p}\log3\right)n\right).
	\]
	For $n$ large enough this is satisfied if
	\[
		(\gamma-1)\log2+\frac\gamma{p}\log3 > 0
		\Longleftrightarrow
		(1-\gamma)\log2 < \frac\gamma{p}\log3
		\Longleftrightarrow
		\frac{1-\gamma}\gamma < \frac1p\cdot\frac{\log3}{\log2}
		\Longleftrightarrow
		\frac1\gamma < \frac1p\cdot\frac{\log3}{\log2} + 1. \qedhere
	\]
\end{proof}

\begin{remark}\label{rmk:not-min-cantor-function}
	The previous statement holds true for the functional
	\[
		F(u) = \TV\!\!_{\phi,\theta}(u) + \lambda\norm{u-\chi}_p^\gamma
	\]
	where
	\[
		\TV\!\!_{\phi,\theta}(u)
		= \int_0^1 \phi(\abs{u'}) + \abs{D^cu}((0,1)) + \sum_{J_u} \theta(\abs{u^+-u^-})
	\]
	is the non-linear total variation.%
	\footnote{The condition $\lim_{t\to\infty}\phi(t)/t=\lim_{t\to0^+}\theta(t)/t=1$ is not required.}
	This is because $\TV_{\phi,\theta}$ coincides with $\TV$ for both $\chi$ and the competitor $\chi_n$.
\end{remark}

The next lemma is used in the proof of \autoref{thm:not-min-alpha-growth}. We will need it only with
$\ell=\sup_{(a,c]}g$, but we prefer to give a more general statement, which might be of broader interest.

\begin{lemma}\label{lem:truncation}
	Let $g\in\BV\bigl((a,b)\bigr)$, $c\in(a,b)$ and
	\[
		\inf_{(a,c]} g \leq \ell \leq \sup_{(a,c]} g.
	\]
	Define $u:(a,b)\to\setR$ such that
	\[
		u(x) = \begin{cases} \ell & x\in(a,c), \\ g(x) & x\in(c,b), \end{cases}
	\]
	and $u$ is regular in $c$. Then $u\in\BV\bigl((a,b)\bigr)$ and
	\[
		\TV\bigl(g;(a,b)\bigr) \geq \TV\bigl(u;(a,b)\bigr) + \abs{\ell-g^+(a)}.
	\]
\end{lemma}

\begin{proof}
	Since $u$ is constant in $(a,c)$, is regular in $c$, and $u=g$ in $(c,b)$, we have
	\[
		\begin{split}
			\TV\bigl(u;(a,b)\bigr)
			= \TV\bigl(u;(a,c)\bigr) + \abs{u^-(c)-u^+(c)} + \TV\bigl(u;(c,b)\bigr)
			= \abs{\ell-g^+(c)} + \TV\bigl(g;(c,b)\bigr),
		\end{split}
	\]
	hence $u\in\BV\bigl((a,b)\bigr)$.
	Let $m=\inf_{(a,c]}g$ and $M=\sup_{(a,c]}g$.
	For every $\eps>0$ there exists $x'_\eps,x''_\eps\in(a,c]$ with $x'_\eps\leq x''_\eps$ such that
	\begin{gather*}
		\abs{g(x'_\eps)-m} \leq \eps \quad\wedge\quad \abs{g(x''_\eps)-M} \leq \eps
		\qquad\text{or}\qquad
		\abs{g(x'_\eps)-M} \leq \eps \quad\wedge\quad \abs{g(x''_\eps)-m} \leq \eps.
	\end{gather*}
	Let us treat the first alternative; the second is analogous.
	Exploiting $\abs{m-M}=\abs{m-\ell}+\abs{\ell-M}$, we can estimate
	\[
		\begin{split}
			&\TV\bigl(g;(a,b)\bigr) \\
			&= \TV\bigl(g;(a,x'_\eps]\bigr) + \TV\bigl(g;[x'_\eps,x''_\eps]\bigr)
			+ \TV\bigl(g;[x''_\eps,c]\bigr) + \TV\bigl(g;[c,b)\bigr) \\
			&\geq \abs{g^+(a)-g(x'_\eps)} + \abs{g(x'_\eps)-g(x''_\eps)} + \abs{g(x''_\eps)-g(c)}
			+ \abs{g(c)-g^+(c)} + \TV\bigl(g;(c,b)\bigr) \\
			&\geq \abs{g^+(a)-m} + \abs{m-M} + \abs{M-g(c)} + \abs{g(c)-g^+(c)} + \TV\bigl(g;(c,b)\bigr) - 4\eps \\
			&= \abs{g^+(a)-m} + \abs{m-\ell} + \abs{\ell-M} + \abs{M-g(c)} + \abs{g(c)-g^+(c)}
			+ \TV\bigl(g;(c,b)\bigr) - 4\eps \\
			&\geq \abs{g^+(a)-\ell} + \abs{\ell-g^+(c)} + \TV\bigl(g;(c,b)\bigr) - 4\eps \\
			&= \abs{g^+(a)-\ell} + \TV\bigl(u;(a,b)\bigr) - 4\eps.
		\end{split}
	\]
	The thesis follows letting $\eps\to0^+$.
\end{proof}

\begin{theorem}\label{thm:not-min-alpha-growth}
	Let $g\in\BV\bigl((a,b)\bigr)$ be a function with the property that there exist $\alpha,c>0$ and
	$0<\delta<b-a$ such that $g(x)-g^+(a)\geq c(x-a)^\alpha$ for all $x\in(a,a+\delta]$.
	If
	\begin{equation}\label{eq:gamma-p-alpha}
		\gamma > \left(1+\frac1{p\alpha}\right)^{-1} = \frac{p\alpha}{p\alpha+1}
	\end{equation}
	then $g$ is not a minimizer of the functional
	\[
		F(u) = \TV(u) + \lambda\norm{u-g}_p^\gamma
	\]
	for every $p\in[1,\infty]$ and $\lambda>0$.

	If instead equality holds in \eqref{eq:gamma-p-alpha}, the same conclusion is true provided that
	\begin{equation}\label{eq:lambda-small}
		\lambda
		< K(p,\alpha)^{-\gamma/p} c^{\gamma/(p\alpha)}
		= K(p,\alpha)^{-\alpha/(1+p\alpha)} c^{1/(1+p\alpha)}.
	\end{equation}
\end{theorem}

\begin{proof}
	Define the non-decreasing function $\omega:(a,a+\delta]\to\setR$
	\[
		\omega(x) = \sup_{(a,x]} g
	\]
	and observe that we have $\omega(x) \geq g(x) \geq g^+(a) + c(x-a)^\alpha$ and $\omega^+(a)=g^+(a)$.
	Given $y\in(a,a+\delta]$, introduce the auxiliary parameter $h=h(y)=\omega(y)-g^+(a)\geq c(y-a)^\alpha>0$
	and define the truncated function $u_y\in\BV\bigl((a,b)\bigr)$
	\[
		u_y(x) = \begin{cases} \omega(y) & x\in(a,y], \\ g(x) & x\in(y,b), \end{cases}
	\]
	which, thanks to \autoref{lem:truncation} applied with $c=y$ and $\ell=\omega(y)$, satisfies
	\[
		\TV\bigl(g,(a,b)\bigr)
		\geq \TV\bigl(u_y,(a,b)\bigr) + \abs{\omega(y)-g^+(a)}
		= \TV\bigl(u_y,(a,b)\bigr) + h.
	\]
	For every $x\in(a,y]$ we have $g^+(a)+c(x-a)^\alpha \leq g(x) \leq \omega(y) = g^+(a) + h = u_y(x)$,
	therefore $\abs{u_y(x)-g(x)} = u_y(x)-g(x) \leq h-c(x-a)^\alpha$.
	Moreover from $c(y-a)^\alpha\leq h$ we deduce $y\leq a+(h/c)^{1/\alpha}$, hence by
	\autoref{lem:holder-integral} we get
	\[
		\norm{u_y-g}_{L^p\bigl((a,b)\bigr)}^p
		\leq \int_a^y \bigl(h-c(x-a)^\alpha\bigr)^p \dx
		\leq \int_a^{a+(h/c)^{1/\alpha}} \bigl(h-c(x-a)^\alpha\bigr)^p \dx
		= K(p,\alpha) c^{-1/\alpha} h^{p+1/\alpha}.
	\]
	Combining the estimates for the total variation and fidelity term we obtain
	\[
		F(u_y)-F(g)
		\leq -h + \lambda K(p,\alpha)^{\gamma/p} c^{-\gamma/(p\alpha)} h^{\gamma\bigl(1+\frac{1}{p\alpha}\bigr)},
	\]
	which is negative for $y$ sufficiently small because $h(y)\to0^+$ for $y\to a^+$ and
	the exponent of $h$ in the last term is greater than $1$ by \eqref{eq:gamma-p-alpha},
	hence the negative linear term dominates.

	If instead $\gamma$ satisfies the equality in \eqref{eq:gamma-p-alpha}, the exponent of $h$ is $1$ and
	the negativity follows from \eqref{eq:lambda-small}.
\end{proof}

The following remark highlights a strong structural property of the minimizers of $F$ when the input
is monotone.

\begin{remark}
	Assume that the input $g$ is non-decreasing and let $u$ be a minimizer of
	$F(u) = \TV(u) + \lambda\norm{u-g}_p^\gamma$; $u$ must be non-decreasing. Define
	\begin{align*}
		x'  & = \sup\set{x\in(a,b)}{g\leq u^+(a) \text{ in } (a,x)}, \\
		x'' & = \inf\set{x\in(a,b)}{g\geq u^-(b) \text{ in } (x,b)},
	\end{align*}
	and the non-decreasing regular function
	\[
		\bar u(x) = \begin{cases}
			u^+(a) & x\in(a,x'],   \\
			g(x)   & x\in(x',x''), \\
			u^-(b) & x\in[x'',b).
		\end{cases}
	\]
	We have $\TV(\bar u)=u^-(b)-u^+(a)=\TV(u)$.
	In $(a,x')$ we have $g\leq\bar u\leq u$, whereas in $(x'',b)$ the reverse inequalities hold, therefore
	\begin{align*}
		\norm{\bar u-g}_{L^p\bigl((a,x')\bigr)}  & \leq \norm{u-g}_{L^p\bigl((a,x')\bigr)},  &
		\norm{\bar u-g}_{L^p\bigl((x'',b)\bigr)} & \leq \norm{u-g}_{L^p\bigl((x'',b)\bigr)}.
	\end{align*}
	In $(x',x'')$ we have $\bar u=g$, therefore
	\[
		\norm{\bar u-g}_{L^p\bigl((x',x'')\bigr)} = 0 \leq \norm{u-g}_{L^p\bigl((x',x'')\bigr)}.
	\]
	In all three cases, equality holds if and only if $u=\bar u$ in the respective intervals.
	This shows that $F(\bar u) \leq F(u)$, with equality if and only if $u=\bar u$ in $(a,b)$.

	This proves that all minimizers must be of the form $\bar u$, although we do not show uniqueness
	because the values at the endpoints $u^+(a)$ and $u^-(b)$ might not be uniquely determined.

	The main content of \autoref{thm:not-min-alpha-growth} is showing that $u^+(a)>g^+(a)$,
	hence a minimizer cannot coincide with $g$.
\end{remark}

\begin{theorem}\label{thm:every-min-constant}
	Let $g\in\BV\bigl((a,b)\bigr)$ be a non-decreasing function and let $p\in[1,\infty]$.
	If $0<\lambda<(b-a)^{-1/p}$ then every minimizer $u\in\BV\bigl((0,1)\bigr)$ of the functional
	\[
		F(u) = \TV(u) + \lambda \norm{u-g}_p
	\]
	is a constant function. Moreover, the minimizer is unique if $p>1$.
\end{theorem}

\begin{proof}
	We treat the case $p<\infty$; the case $p=\infty$ follows from a straightforward adaptation.
	Let $u\in\BV\bigl((a,b)\bigr)$ be a minimizer. By \autoref{prop:increasing-minimizer} $u$ is non-decreasing.
	Moreover, by \autoref{prop:global-truncation} we have $g^+(a)\leq u\leq g^-(b)$.

	Let $c\in[u^+(a),u^-(b)]$ and define $h=c-u^+(a)\geq0$ and $k=u^-(b)-c\geq0$.
	Using the fact that
	\[
		\norm{u-c}_p \leq (b-a)^{1/p}\norm{u-c}_\infty = (b-a)^{1/p}\max\{h,k\} \leq (b-a)^{1/p}(h+k)
	\]
	we can estimate
	\[
		\begin{split}
			F(u) - F(c)
			&= \TV(u) - \TV(c) + \lambda(\norm{u-g}_p-\norm{c-g}_p)
			\geq u^-(b)-u^+(a) - \lambda\norm{u-c}_p \\
			&\geq h+k - \lambda(b-a)^{1/p}(h+k)
			= \bigl(1-\lambda(b-a)^{1/p}\bigr)(h+k).
		\end{split}
	\]
	Under the assumption on $\lambda$, the inequality $F(u)\geq F(c)$ is strict unless $h+k=0$,
	i.e.\ $h=k=0$, that is $u\equiv c$. This proves that every minimizer is constant.

	If $p>1$, then $F(c)=\lambda\norm{c-g}_p$ is strictly convex, therefore the minimizer is unique.
\end{proof}

\subsection{The input is the minimizer}\label{sec:linear-min}

In this section we find conditions that ensure that the input $g$ is the minimizer.
This has the interesting consequence that,
under the assumptions of either \autoref{thm:min-flat}, \autoref{thm:min-monotone}
or \autoref{thm:min-piecewise-monotone}, if $D^cg\neq0$, then also $D^cu\neq0$ where $u$ is the
minimizer of $F(u) = \TV(u) + \lambda\norm{u-g}_p^\gamma$,
showing that in general minimizers can have the Cantor part of the derivative.

\begin{theorem}\label{thm:min-flat}
	Let $g\in\BV\bigl((0,1);[0,1]\bigr)$ be a non-decreasing function with $g\equiv0$ on $(0,\delta)$ and $g\equiv1$ on $(1-\delta,1)$ for some $\delta\in(0,1/2)$, and let $\lambda,\gamma>0$, $p\in[1,\infty]$.
	If $\gamma\leq1$ and $\lambda\delta^{\gamma/p}>2$ then $u=g$ is the unique minimizer of the functional $F:\BV\bigl((0,1)\bigr)\to\setR$ given by
	\[
		F(u) = \TV(u) + \lambda\norm{u-g}_p^\gamma.
	\]
\end{theorem}

\begin{proof}
	Let $u\in\BV\bigl((0,1)\bigr)$ be a minimizer. By \autoref{prop:increasing-minimizer} $u$ is non-decreasing. Moreover we have $0\leq u\leq1$ because $F(0\vee u\wedge1)\leq F(u)$ with strict inequality otherwise.

	We claim that $u\equiv u(0^+)$ on $(0,\delta)$ and $u\equiv u(1^-)$ on $(1-\delta,1)$. In fact if this is not the case then the function
	\[
		\tilde u(x)
		= \begin{cases}
			u(0^+) & x \in (0, \delta)        \\
			u(x)   & x \in (\delta, 1-\delta) \\
			u(1^-) & x \in (1-\delta, 1)
		\end{cases}
	\]
	has the property $\TV(\tilde u)\leq\TV(u)$ and $\norm{\tilde u-g}_p<\norm{u-g}_p$.

	We have the estimates $\norm{u-g}_\infty \geq \max\{u(0^+),1-u(1^-)\}$ and
	\[
		\begin{split}
			\norm{u-g}_p^p
			&\geq \int_{(0,\delta)\cup(1-\delta,1)} \abs{u(x)-g(x)}^p \d x \\
			&= \int_0^\delta u(0^+)^p \d x + \int_{1-\delta}^1 \bigl(1-u(1^-)\bigr)^p \d x
			= \delta \bigl[ u(0^+)^p + \bigl(1-u(1^-)\bigr)^p \bigr].
		\end{split}
	\]
	Letting $h=u(0^+)$ and $k=1-u(1^-)$, these can be summarized as $\norm{u-g}_p \geq \delta^{1/p}M_p(h,k)$
	where $M_p$ is the $p$-mean.\footnote{See \autoref{def:p-mean}.} Notice that $M_p(h,k)\leq1$, therefore
	$M_p(h,k)^\gamma\geq M_p(h,k)\geq M_1(h,k)$ for $\gamma\in(0,1]$.
	Using these estimates and $F(g)=1$ we can bound from below
	\[
		F(u)
		= u(1^-)-u(0^+) + \lambda\norm{u-g}_p^\gamma
		\geq 1 - (h+k) + \lambda\delta^{\gamma/p}M_p(h,k)^\gamma
		\geq F(g) + (\lambda\delta^{\gamma/p}-2) M_1(h,k)
	\]
	and this inequality is strict unless $M_1(h,k)=0$, i.e. $h=k=0$, that is $u(0^+)=0$ and $u(1^-)=1$.

	In this case $\TV(u)=u(1^-)-u(0^+)=1$, hence $F(u)=1+\lambda\norm{u-g}_p^\gamma=F(g)+\lambda\norm{u-g}_p^\gamma$ and for $u$ to be a minimizer we must have $u=g$.
\end{proof}

The next theorem treats the case of a monotone input. This will be later generalized to piecewise
monotone functions in \autoref{thm:min-piecewise-monotone}, with the additional difficulty of having
to deal with the jumps between the subintervals of monotonicity.

\begin{theorem}\label{thm:min-monotone}
	Let $g\in\BVreg\bigl((a,b)\bigr)$ be a non-decreasing function
	with at most $\alpha$-growth at the endpoints $a$ and $b$,
	i.e.\ there exist $\alpha>0$ and $C>0$ such that $g(x)-g^+(a)\leq C(x-a)^\alpha$ and
	$g^-(b)-g(x)\leq C(b-x)^\alpha$ for all $x\in(a,b)$.
	If the conditions
	\begin{align*}
		\gamma  & \leq \left(1+\frac1{p\alpha}\right)^{-1} = \frac{p\alpha}{p\alpha+1} &
		        & \text{and}                                                           &
		\lambda & > 2^{\frac{(p-1)\alpha+1}{p\alpha+1}}
		K(p,\alpha)^{-\gamma/p} C^{\gamma/(p\alpha)} H^{1-\gamma\bigl(1+1/(p\alpha)\bigr)}
	\end{align*}
	are satisfied, then $g$ is essentially the unique minimizer of the functional
	\begin{gather*}
		F     : \BV\bigl((x_0,x_n)\bigr)\to[0,\infty)  \\
		F(u)  = \TV(u) + \lambda\norm{u-g}_p^\gamma,
	\end{gather*}
	in the sense that any other minimizer coincides with $g$ outside of the jump set of $g$.
\end{theorem}

\begin{figure}
	\centering
	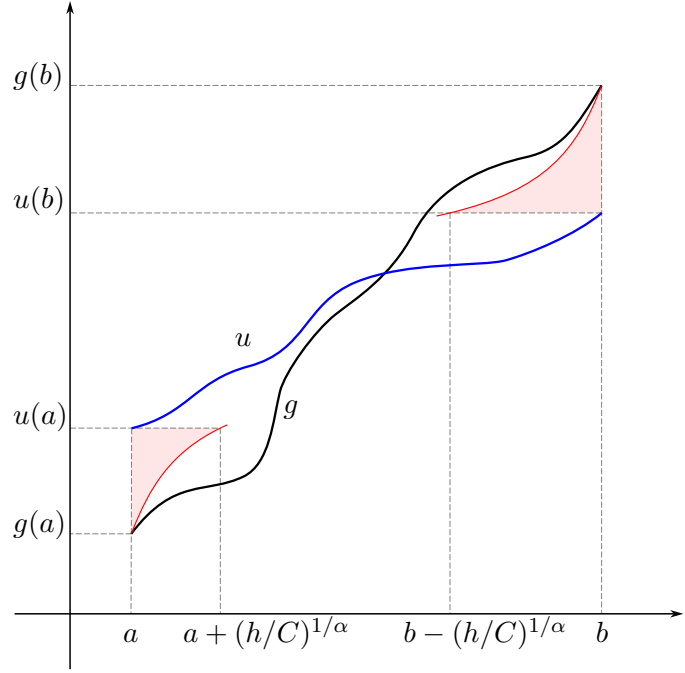
	\caption{Visual aid for the proof of \autoref{thm:min-monotone}.
		The two red arcs are $g(a)+C(x-a)^\alpha$ and $g(b)-C(b-x)^\alpha$,
		which bound $g$ from above and below respectively.
		The shaded regions are the two contributions considered in the estimate of $\norm{u-g}_p^p$.}
	\label{fig:holder}
\end{figure}

\begin{proof}
	Let $u$ be a minimizer and let $H=g^-(b)-g^+(a)$.
	From \autoref{prop:global-truncation} and \autoref{prop:increasing-minimizer} we know that
	$u$ is non-decreasing and $g^+(a)\leq u(x)\leq g^-(b)$ for every $x\in(a,b)$.
	Let $u^+(a)=g^+(a)+h$ and $u^+(b)=g^+(b)-k$ with $0\leq h,k\leq H$.
	Our first goal is to prove that $h=k=0$.
	With regard to the total variation, we have $\TV(u)=u^-(b)-u^+(a)=g^-(b)-g^+(a)-(h+k)=\TV(g)-(h+k)$.
	With regard to the fidelity term, from \autoref{lem:fidelity lower bound} we have that the
	estimate \eqref{eq:fidelity lower bound} holds, that is
	\[
		\norm{u-g}_p^p
		\geq K(p,\alpha) C^{-1/\alpha} \bigl(h^{p+1/\alpha}+k^{p+1/\alpha}\bigr).
	\]
	For convenience, introduce the parameter $\beta=p+1/\alpha>1$ appearing as exponent in the
	previous estimate.
	Recalling that $h,k\leq H$, if $\gamma\leq p/\beta=\frac{p\alpha}{p\alpha+1}$ by the inequalities
	between means we have
	\begin{align*}
		h+k                                    & \leq 2^{1-1/\beta}\bigl(h^\beta+k^\beta\bigr)^{1/\beta},       &
		\bigl(h^\beta+k^\beta\bigr)^{\gamma/p} & \geq H^{\beta\gamma/p-1}\bigl(h^\beta+k^\beta\bigr)^{1/\beta}.
	\end{align*}
	Combining these estimates leads to
	\[
		\begin{split}
			0
			&\geq F(u)-F(g)
			\geq -(h+k)
			+\lambda K(p,\alpha)^{\gamma/p} C^{-\gamma/(p\alpha)} \bigl(h^\beta+k^\beta\bigr)^{\gamma/p} \\
			&\geq \left[\lambda K(p,\alpha)^{\gamma/p} C^{-\gamma/(p\alpha)} H^{\beta\gamma/p-1}
				- 2^{1-1/\beta}\right]
			\bigl(h^\beta+k^\beta\bigr)^{1/\beta}.
		\end{split}
	\]
	The assumption on $\gamma$ implies that the coefficient in the square brackets is positive,
	therefore we must have $h=k=0$.

	Now that we have $u^+(a)=g^+(a)$ and $u^-(b)=g^-(b)$, we conclude by noticing that $\TV(u)=\TV(g)$, hence
	$F(u)-F(g)=\lambda\norm{u-g}^\gamma\geq0$, with equality if and only if $u=g$.
\end{proof}

\begin{proposition}\label{prop:u_contained_in_g}
	Let $x_0<\dots<x_n$ and let $g\in\BV\bigl((x_0,x_n)\bigr)$ be a regular function which is monotone
	in every interval $I_i=(x_{i-1},x_i)$.
	If the conditions
	\begin{align*}
		p       & \geq 1,                                     &
		\gamma  & \leq 1,                                     &
		\lambda & > 2/G(I_i) \quad \forall i\in\{1,\dots,n\},
	\end{align*}
	are satisfied, where for every $J\subseteq I\coloneqq(x_0,x_n)$
	\[
		G(J) = \frac{\gamma}{\bigl(H\leb^1(I)^{1/p}\bigr)^{1-\gamma}}
		\begin{cases}
			\leb^1(J)^{1/p}
			 & \text{$p=1$ or $g\rvert_{J}$ constant}, \\
			p \left(\frac{\norm{g-\inf_{J}g}_{L^{p-1}(J)}}{\norm{g-\inf_{J}g}_{L^p(I)}}
			\wedge
			\frac{\norm{g-\sup_{J}g}_{L^{p-1}(J)}}{\norm{g-\sup_{J}g}_{L^p(I)}}\right)^{p-1}
			 & \text{otherwise},
		\end{cases}
	\]
	then any minimizer $u\in\BV\bigl((x_0,x_n)\bigr)$ of the functional
	\begin{gather*}
		F    : \BV\bigl((x_0,x_n)\bigr)\to[0,\infty)  \\
		F(u) = \TV(u) + \lambda\norm{u-g}_p^\gamma.
	\end{gather*}
	satisfies
	\begin{align*}
		\inf_{I_i} g \leq \inf_{I_i} u & \leq \sup_{I_i} u \leq \sup_{I_i} g &
		                               & \forall i\in\{1,\dots,n\}.
	\end{align*}
\end{proposition}

\begin{proof}
	Let $H=\sup_I g-\inf_I g$.

	We may assume that $u$ is regular, otherwise regularizing it decreases strictly the total variation without
	affecting the fidelity term (see \autoref{prop:basic-properties}\ref{it:regularizations} and
	\autoref{prop:further-properties}\ref{it:regular-better}).
	We may also assume that $\inf_I g\leq\inf_I u\leq\sup_I u\leq\sup_I g$,
	thanks to \autoref{rmk:BV-truncation}.

	Exploiting the symmetry, assume by contradiction that the inequality $\inf_{I_i}g\leq\inf_{I_i}u$
	is violated for some index $i\in\{1,\dots,n\}$ and that $g$ is increasing in $I_i$.
	Letting $m=\inf_{I_i}g$, our goal is to show that the function $\tilde u$ such that $\tilde u=u$
	in $(x_0,x_n)\setminus I_i$ and $\tilde u=u\vee m$ in $I_i$ performs strictly better than $u$,
	i.e.\ $F(\tilde u)<F(u)$, therefore contradicting the absurd assumption.
	To this end we let
	\[
		\delta = \inf_{I_i}g - \sup_{I_i}u = m - \sup_{I_i}u
	\]
	and we distinguish two cases: either $\delta>0$, or $\delta\leq0$.

	\textit{Case $\delta>0$.}
	We start by showing that
	\begin{equation}\label{eq:tv-u-tildeu-delta}
		\TV\bigl(\tilde u;(x_0,x_n)\bigr) \leq \TV\bigl(u;(x_0,x_n)\bigr) + 2\delta
	\end{equation}
	We split
	\[
		\TV\bigl(u;(x_0,x_n)\bigr)
		=
		\TV\bigl(u;(x_0,x_{i-1}]\bigr)
		+ \TV\bigl(u;[x_{i-1},x_i]\bigr)
		+ \TV\bigl(u;[x_i,x_n)\bigr)
	\]
	and similarly for $\tilde u$.
	Since $\tilde u$ coincides with $u$ in $(x_0,x_{i-1}]\cup[x_i,x_n)$, we have
	\begin{align*}
		\TV\bigl(u;(x_0,x_{i-1}]\bigr) & = \TV\bigl(\tilde u;(x_0,x_{i-1}]\bigr), &
		\TV\bigl(u;[x_i,x_n)\bigr)     & = \TV\bigl(\tilde u;[x_i,x_n)\bigr).
	\end{align*}
	Let us focus on the total variations of $u$ and $\tilde u$ over the interval $[x_{i-1},x_i]$.
	Since $\sup_{I_i}u=m-\delta$, for every $\eps>0$ there exists $x^\star_\eps\in I_i$ such that
	$m-\delta-\eps\leq u(x^\star_\eps)\leq m-\delta$; moreover $\tilde u=m$ in $(x_{i-1},x_i)$.
	\begin{align*}
		\TV\bigl(u,[x_{i-1},x_i]\bigr)
		 & \geq	\abs{u(x^\star_\eps)-u(x_{i-1})} + \abs{u(x_i)-u(x^\star_\eps)}, \\
		\TV\bigl(\tilde u,[x_{i-1},x_i]\bigr)
		 & =	\abs{m-\tilde u(x_{i-1})} + \abs{\tilde u(x_i)-m}.
	\end{align*}
	From these and $\abs{u(x^\star_\eps)-m}\leq\delta+\eps$ we deduce
	$\TV\bigl(\tilde u;[x_{i-1},x_i]\bigr) \leq \TV\bigl(u;[x_{i-1},x_i]\bigr) + 2\delta + 2\eps$,
	which leads to \eqref{eq:tv-u-tildeu-delta} by letting $\eps\to0^+$.

	Let us now study the fidelity terms associated to $u$ and $\tilde u$.
	In $I\setminus I_i$ the functions $u$ and $\tilde u$ coincide,
	whereas in the interval $I_i$ we have $u\leq m-\delta<\tilde u=m\leq g$,
	hence $\abs{u-g}=g-u\geq g-m+\delta$ and $\abs{\tilde u-g}=g-\tilde u=g-m$.
	Letting $f=\abs{\tilde u-g}$ we have
	\[
		\norm{u-g}_p - \norm{\tilde u-g}_p
		\geq \norm*{f+\delta\bm1_{I_i}}_p - \norm{f}_p
		\eqqcolon \Phi(\delta).
	\]
	The function $\Phi:[0,\infty)\to\setR$ is non-negative, increasing, convex, and $\Phi(0)=0$.
	If $p=1$ then $\Phi(\delta)=\leb^1(I_i)\delta$;
	for $p>1$ instead, if $f\equiv0$ then $\Phi(\delta)=\leb^1(I_i)^{1/p}\delta$, otherwise
	\[
		\begin{split}
			\Phi'(\delta)
			&= \frac\d{\d\delta}\left(\int_I (f+\delta\bm1_{I_i})^p\right)^{\frac1p}
			= \frac1p\left(\int_I (f+\delta\bm1_{I_i})^p\right)^{\frac1p-1}
			\frac\d{\d\delta}\int_I (f+\delta\bm1_{I_i})^p \\
			&= \left(\int_I (f+\delta\bm1_{I_i})^p\right)^\frac{1-p}p
			p \int_{I_i} (f+\delta)^{p-1}
			= p\left(\frac{\norm{f+\delta}_{L^{p-1}(I_i)}}{\norm{f+\delta\bm1_{I_i}}_{L^p(I)}}\right)^{p-1},
		\end{split}
	\]
	hence
	\[
		\Phi(\delta)
		\geq \Phi'(0) \delta
		= p \left(\frac{\norm{f}_{L^{p-1}(I_i)}}{\norm{f}_{L^p(I)}}\right)^{p-1} \delta.
	\]
	In conclusion,
	\[
		\norm{u-g}_p - \norm{\tilde u-g}_p
		\geq C_i \delta
	\]
	where
	\[
		C_i = \begin{cases}
			\leb^1(I_i)^{1/p}
			 & p=1 \vee g\equiv m\ \text{in}\ I_i, \\
			\displaystyle p \left(\frac{\norm{f}_{L^{p-1}(I_i)}}{\norm{f}_{L^p(I)}}\right)^{p-1}
			 & \text{otherwise}.
		\end{cases}
	\]

	Since $0\leq\norm{\tilde u-g}_p\leq\norm{u-g}_p\leq H\leb^1(I)^{1/p}$, from the previous inequality
	and Lagrange's Mean Value Theorem we deduce
	\begin{equation}\label{eq:u-tildeu-fidelity}
		\norm{u-g}_p^\gamma - \norm{\tilde u-g}_p^\gamma
		\geq \frac{\gamma C_i \delta}{\bigl(H\leb^1(I)^{1/p}\bigr)^{1-\gamma}}
		\geq G(I_i)\delta.
	\end{equation}

	Combining the estimates \eqref{eq:tv-u-tildeu-delta} for the total variation and
	\eqref{eq:u-tildeu-fidelity} for the fidelity term, we get
	\[
		0 \geq F(u)-F(\tilde u)
		= \bigl(\TV(u)-\TV(\tilde u)\bigr)
		+ \lambda\left( \norm{u-g}_{L^p(I)}^\gamma - \norm{\tilde u-g}_{L^p(I)}^\gamma \right)
		\geq (\lambda G(I_i)-2)\delta,
	\]
	which is a contradiction because both factors in the right hand side are strictly positive.

	\textit{Case $\delta\leq0$.}
	We treat the total variation and the fidelity term separately.
	With regard to the total variation, our goal is to show that
	\begin{equation}\label{eq:tv-u-tildeu}
		\TV\bigl(\tilde u;(x_0,x_n)\bigr)\leq\TV\bigl(u;(x_0,x_n)\bigr).
	\end{equation}
	By hypothesis $\sup_{I_i}u\geq\inf_{I_i}g$, so for every $\eps>0$ there is $x^\star_\eps\in I_i$ such that
	$u(x^\star_\eps)>\inf_{I_i}g-\eps\eqqcolon m_\eps$.
	We split
	\[
		\TV\bigl(u;(x_0,x_n)\bigr)
		=
		\TV\bigl(u;(x_0,x_{i-1}]\bigr)
		+ \TV\bigl(u;[x_{i-1},x^\star_\eps]\bigr)
		+ \TV\bigl(u;[x^\star_\eps,x_i]\bigr)
		+ \TV\bigl(u;[x_i,x_n)\bigr)
	\]
	and similarly for $\tilde u$.
	Since $\tilde u$ coincides with $u$ in $(x_0,x_{i-1}]\cup[x_i,x_n)$, we have
	\begin{align*}
		\TV\bigl(u;(x_0,x_{i-1}]\bigr) & = \TV\bigl(\tilde u;(x_0,x_{i-1}]\bigr), &
		\TV\bigl(u;[x_i,x_n)\bigr)     & = \TV\bigl(\tilde u;[x_i,x_n)\bigr).
	\end{align*}
	Let us focus on the total variations of $u$ and $\tilde u$ over the interval $[x_{i-1},x^\star_\eps]$,
	since those over $[x^\star_\eps,x_i]$ are treated analogously.
	In all following cases we will be proving that
	\begin{equation}\label{eq:tv-u-tildeu-eps}
		\TV\bigl(\tilde u;[x_{i-1},x^\star_\eps]\bigr) \leq \TV\bigl(u;[x_{i-1},x^\star_\eps]\bigr)+\eps,
	\end{equation}
	so that passing to the limit $\eps\to0^+$ we finally get \eqref{eq:tv-u-tildeu}.
	\begin{enumerate}
		\item If $u(x_{i-1})\geq m$, then $\tilde u = u\vee m$ on the whole interval $[x_{i-1},x^\star_\eps]$,
		      hence we have immediately
		      $\TV\bigl(\tilde u;[x_{i-1},x^\star_\eps]\bigr) \leq \TV\bigl(u;[x_{i-1},x^\star_\eps]\bigr)$.

		\item Consider now $u(x_{i-1})<m$.
		      \begin{enumerate}
			      \item If $u^+(x_{i-1})\geq m_\eps$ then
			            $u^+(x_{i-1})\leq\tilde u^+(x_{i-1})\leq u^+(x_{i-1})+\eps$, hence
			            \eqref{eq:tv-u-tildeu-eps} follows as
			            \[
				            \begin{split}
					            \TV\bigl(u;[x_{i-1},x^\star_\eps]\bigr)
					            &= \abs{u^+(x_{i-1})-u(x_{i-1})} + \TV\bigl(u;(x_{i-1},x^\star_\eps]\bigr) \\
					            &\geq \abs{\tilde u^+(x_{i-1})-\tilde u(x_{i-1})} - \eps
					            + \TV\bigl(\tilde u;(x_{i-1},x^\star_\eps]\bigr) \\
					            &= \TV\bigl(\tilde u;[x_{i-1},x^\star_\eps]\bigr) - \eps.
				            \end{split}
			            \]
			            \item\label{it:2b} Otherwise, if $u^+(x_{i-1})<m_\eps$, define
			            \[
				            x'_\eps = \max\set*{x\in(x_{i-1},x^\star_\eps]}{u\leq m_\eps \text{ in } (x_{i-1},x)}.
			            \]
			            We have that $u^-(x'_\eps)\leq m_\eps$.
			            \begin{enumerate}
				            \item If $x'_\eps<x^\star_\eps$ then $u^+(x'_\eps)\geq m_\eps$,
				                  so we may assume without loss of generality that $u(x'_\eps)=m_\eps$;
				                  moreover $\tilde u=m$ in $(x_{i-1},x'_\eps]$.
				                  \[
					                  \begin{split}
						                  \TV\bigl(u;[x_{i-1},x^\star_\eps]\bigr)
						                  &= \TV\bigl(u;[x_{i-1},x'_\eps]\bigr)
						                  + \TV\bigl(u;[x'_\eps,x^\star_\eps]\bigr) \\
						                  &\geq \abs{u(x'_\eps)-u(x_{i-1})}
						                  + \TV\bigl(u;[x'_\eps,x^\star_\eps]\bigr)
					                  \end{split}
				                  \]
				                  and
				                  \[
					                  \begin{split}
						                  \TV\bigl(\tilde u;[x_{i-1},x^\star_\eps]\bigr)
						                  &= \TV\bigl(\tilde u;[x_{i-1},x'_\eps]\bigr)
						                  + \TV\bigl(\tilde u;[x'_\eps,x^\star_\eps]\bigr) \\
						                  &= \abs{m-u(x_{i-1})}
						                  + \TV\bigl(\tilde u;[x'_\eps,x^\star_\eps]\bigr)
					                  \end{split}
				                  \]
				                  Observing that $\abs{u(x'_\eps)-m}=\abs{m_\eps-m}=\eps$ and
				                  $\TV\bigl(u;[x'_\eps,x^\star_\eps]\bigr)
					                  \geq \TV\bigl(\tilde u;[x'_\eps,x^\star_\eps]\bigr)$,
				                  we deduce \eqref{eq:tv-u-tildeu-eps}.
				            \item If instead $x'_\eps=x^\star_\eps$,
				                  then $u\leq m_\eps$ and $\tilde u=m$ in $(x_{i-1},x^\star_\eps)$, hence
				                  \[
					                  \begin{split}
						                  \TV\bigl(u;[x_{i-1},x^\star_\eps]\bigr)
						                  \geq \abs{u(x^\star_\eps)-u(x_{i-1})}
					                  \end{split}
				                  \]
				                  and, since $u(x_{i-1})<m\leq\tilde u(x^\star_\eps)$,
				                  \[
					                  \TV\bigl(\tilde u;[x_{i-1},x^\star_\eps]\bigr)
					                  = \abs{m-u(x_{i-1})} + \abs{\tilde u(x^\star_\eps)-m}
					                  = \abs{\tilde u(x^\star_\eps)-u(x_{i-1})};
				                  \]
				                  noticing that $u(x^\star_\eps)\leq\tilde u(x^\star_\eps)\leq u(x^\star_\eps)+\eps$
				                  we conclude again that \eqref{eq:tv-u-tildeu-eps} holds.
			            \end{enumerate}
		      \end{enumerate}
	\end{enumerate}
	This concludes the proof of \eqref{eq:tv-u-tildeu-eps}, and hence that of \eqref{eq:tv-u-tildeu}.

	On the other hand, since $\inf_{I_i}u<\inf_{I_i}g$, there is a point $x^\star\in I_i$ such that
	$u(x^\star)<\inf_{I_i}g$. By regularity of $u$, there is an open interval (one endpoint of which
	is $x^\star$) where $u<\inf_{I_i}g$. This implies that $\norm{\tilde u-g}_p<\norm{u-g}_p$,
	which together with \eqref{eq:tv-u-tildeu} leads to $F(\tilde u)<F(u)$, which contradicts the
	minimality of $u$.
\end{proof}

\begin{theorem}\label{thm:min-piecewise-monotone}
	Let $x_0<\dots<x_n$ and let $g\in\BVreg\bigl((x_0,x_n)\bigr)$ be a function which is monotone
	in every interval $I_i=(x_{i-1},x_i)$ and with at most $\alpha$-growth at their extremes,
	i.e.\ there exist $\alpha,C>0$ such that
	\begin{align*}
		\abs{g(x)-g^+(x_{i-1})} & \leq C(x-x_{i-1})^\alpha    &
		\abs{g(x)-g^-(x_i)}     & \leq C(x_i-x)^\alpha        &
		                        & \forall x\in (x_{i-1},x_i).
	\end{align*}
	Given some points $\xi_i\in I_i$, let $I_i^-=(x_{i-1},\xi_i)$ and $I_i^+=(\xi_i,x_i)$.
	If the conditions of \autoref{prop:u_contained_in_g} hold for the intervals $I_i$, $I_i^-$ and $I_i^+$,
	and in addition the inequalities
	\begin{align*}
		\gamma  & \leq \left(1+\frac1{p\alpha}\right)^{-1} = \frac{p\alpha}{p\alpha+1}                            &
		        & \text{and}                                                                                      &
		\lambda & > 2^{2+1/\beta} K(p,\alpha)^{-\gamma/p} C^{\gamma/(p\alpha)} n^{1-\gamma/p} H^{1-\beta\gamma/p}
	\end{align*}
	are satisfied, then $g$ is essentially the unique minimizer of the functional
	\begin{gather*}
		F     : \BV\bigl((x_0,x_n)\bigr)\to[0,\infty)  \\
		F(u)  = \TV(u) + \lambda\norm{u-g}_p^\gamma,
	\end{gather*}
	in the sense that any other minimizer coincides with $g$ outside of the jump set of $g$.
\end{theorem}

In the course of the proof it would appear natural to rearrange monotonically $u\rvert_{I_i}$
via \autoref{prop:increasing-minimizer}. Unfortunately, this transformation does not keep track of
jumps at the endpoints of these intervals, which may cause an increase in energy.
This obstruction is overcome by the additional requests on the intervals $I_i^\pm$
and a more involved ad hoc proof strategy.

\begin{proof}
	Let $I=(x_0,x_n)$, $H=\sup_Ig-\inf_Ig$ and $H_i=\sup_{I_i}g-\inf_{I_i}g$.
	By \autoref{prop:u_contained_in_g}, for every $i\in\{1,\dots,n\}$ we have
	\begin{align*}
		\inf_{I_i} g \leq \inf_{I_i} u         & \leq \sup_{I_i} u \leq \sup_{I_i} g,         &
		\inf_{I_i^\pm} g \leq \inf_{I_i^\pm} u & \leq \sup_{I_i^\pm} u \leq \sup_{I_i^\pm} g.
	\end{align*}

	Fix $i\in\{1,\dots,n\}$. Assume that $g$ is non-decreasing in $I_i$. Define
	\begin{align*}
		h_i & = \inf_{I_i}u-g^+(x_{i-1}) = \inf_{I_i}u-\inf_{I_i}g \geq 0, &
		k_i & = g^-(x_i)-\sup_{I_i}u = \sup_{I_i}g-\sup_{I_i}u \geq 0.
	\end{align*}
	Notice that $h_i+k_i\leq H_i\leq H$.
	Moreover
	\[
		\inf_{I_i^-} u \leq \sup_{I_i^-} g = g^-(\xi_i) \leq g^+(\xi_i) = \inf_{I_i^+} g \leq \inf_{I_i^+} u
	\]
	and
	\[
		u(\xi_i) \geq \min\{u^-(\xi_i),u^+(\xi_i)\} \geq \min\left\{\inf_{I_i^-}u, \inf_{I_i^+}u\right\}
		\geq \inf_{I_i^-}u
	\]
	therefore
	\[
		\inf_{I_i^-} u = \inf_{I_i} u
		\qquad\text{and}\qquad
		\sup_{I_i^+} u = \sup_{I_i} u.
	\]

	Given $\eps>0$, let $x^\pm_{i,\eps}\in I_i^\pm$ such that
	\begin{align*}
		u(x^-_{i,\eps}) & \leq \inf_{I_i}u + \eps  &
		                & \text{and}               &
		u(x^+_{i,\eps}) & \geq \sup_{I_i}u - \eps.
	\end{align*}

	For every $i\in\{1,\dots,n\}$ there exits $x_{i-1}<x'_{i,\eps}<x''_{i,\eps}<x_i$ such that
	\begin{align}\label{eq:x' x'' ineq}
		\abs{u(x'_{i,\eps})-g^+(x_{i-1})} & \leq h_i + \eps, &
		\abs{u(x''_{i,\eps})-g^-(x_i)}    & \leq k_i + \eps.
	\end{align}

	By the monotonicity of $g\rvert_{I_i}$ we have
	\begin{equation}\label{eq:TV g}
		\begin{split}
			\TV(g;[x_0,x_n])
			&= \sum_{i=1}^n \TV(g;(x_{i-1},x_i)) + \sum_{i=1}^{n-1} \abs{[g](x_i)} \\
			&= \sum_{i=1}^n \abs{g^-(x_i)-g^+(x_{i-1})} + \sum_{i=1}^{n-1} \abs{g^+(x_i)-x^-(x_i)}.
		\end{split}
	\end{equation}
	On the other hand $\TV(u;[x_0,x_n])$ can be estimated from below along the sorted list of points
	\[
		x'_{1,\eps}<x''_{1,\eps} < \dots < x'_{i,\eps}<x''_{i,\eps} < \dots < x'_{n,\eps}<x''_{n,\eps}
	\]
	and using \eqref{eq:x' x'' ineq} and \eqref{eq:TV g}, obtaining
	\[
		\begin{split}
			&\TV(u;[x_0,x_n]) \\
			&\geq \sum_{i=1}^n \abs{u(x''_{i,\eps})-u(x'_{i,\eps})}
			+ \sum_{i=1}^{n-1} \abs{u(x'_{i+1,\eps})-u(x''_{i,\eps})} \\
			&\overset{\eqref{eq:x' x'' ineq}}\geq \sum_{i=1}^n \bigl(\abs{g^-(x_i)-g^+(x_{i-1})}-h_i-k_i-2\eps\bigr)
			+ \sum_{i=1}^{n-1} \bigl(\abs{g^+(x_i)-x^-(x_i)}-h_{i+1}-k_i-2\eps\bigr) \\
			&\overset{\eqref{eq:TV g}}\geq \TV(g;[x_0,x_n]) - 2\sum_{i=1}^n (h_i+k_i) - (4n-2)\eps.
		\end{split}
	\]
	Letting $\eps\to0$ we deduce $\TV(u;[x_0,x_n])\geq\TV(g;[x_0,x_n])-2\sum_{i=1}^n(h_i+k_i)$.

	Applying \autoref{lem:fidelity lower bound} to every interval $I_i$ and summing from $1$ to $n$ we get
	\[
		\norm{u-g}_{L^p(I)}^p
		\geq \sum_{i=1}^n \norm{u-g}_{L^p(I_i)}^p
		\geq K(p,\alpha) C^{-1/\alpha} \sum_{i=1}^n (h_i^{p+1/\alpha}+k_i^{p+1/\alpha}).
	\]
	Let us call $\beta=p+1/\alpha>1$ the exponent appearing in the fidelity term estimate. Using for
	convenience the shorter notation $M_q(h;k)=M_q(h_1,k_1,\dots,h_n,k_n)$ to denote the $q$-mean of
	the $2n$ parameters (see \autoref{def:p-mean}), we have that $M_1(h;k) \leq M_\beta(h;k)$.
	Since $t\mapsto t^\beta$ is super-additive, we have also
	$h_i^\beta+k_i^\beta\leq(h_i+k_i)^\beta\leq H_i^\beta\leq H^\beta$, hence
	\[
		M_\beta(h;k)
		= \left(\frac1{2n} \sum_{i=1}^n (h_i^\beta+k_i^\beta)\right)^{1/\beta}
		\leq \left(\frac1{2n} \sum_{i=1}^n H_i^\beta\right)^{1/\beta}
		\leq \left(\frac1{2n} nH^\beta\right)^{1/\beta}
		= 2^{-1/\beta}H.
	\]
	The assumption $\gamma\leq p/\beta$ ensures that $\beta\gamma/p\leq1$, therefore
	\[
		\left(\frac{M_\beta(h;k)}{2^{-1/\beta}H}\right)^{\beta\gamma/p}
		\geq \frac{M_\beta(h;k)}{2^{-1/\beta}H}.
	\]
	This implies
	\[
		\begin{split}
			\norm{u-g}_{L^p(I)}^\gamma
			&\geq K(p,\alpha)^{\gamma/p} C^{-\gamma/(p\alpha)}
			\left(\sum_{i=1}^n (h_i^{p+1/\alpha}+k_i^{p+1/\alpha})\right)^{\gamma/p} \\
			&= K(p,\alpha)^{\gamma/p} C^{-\gamma/(p\alpha)}
			\left(2n M_\beta(h;k)^\beta\right)^{\gamma/p} \\
			&= K(p,\alpha)^{\gamma/p} C^{-\gamma/(p\alpha)}
			(2n)^{\gamma/p} M_\beta(h;k)^{\beta\gamma/p} \\
			&\geq K(p,\alpha)^{\gamma/p} C^{-\gamma/(p\alpha)} (2n)^{\gamma/p}
			2^{-\gamma/p+1/\beta} H^{\beta\gamma/p-1} M_\beta(h;k) \\
			&= 2^{-1/\beta} K(p,\alpha)^{\gamma/p} C^{-\gamma/(p\alpha)} n^{\gamma/p}
			H^{\beta\gamma/p-1} M_\beta(h;k).
		\end{split}
	\]

	Combining the resulting lower bounds for the total variation and the fidelity terms leads to
	\[
		\begin{split}
			0
			&\geq F(u) - F(g)
			= \TV(u;[x_0,x_n]) - \TV(g;[x_0,x_n]) + \lambda\norm{u-g}_{L^p([x_0,x_n])}^\gamma \\
			&\geq - 4n M_1(h;k)
			+ \lambda 2^{-1/\beta} K(p,\alpha)^{\gamma/p} C^{-\gamma/(p\alpha)} n^{\gamma/p}
			H^{\beta\gamma/p-1} M_\beta(h;k) \\
			&\geq \bigl[\lambda 2^{-1/\beta} K(p,\alpha)^{\gamma/p} C^{-\gamma/(p\alpha)}
				n^{\gamma/p} H^{\beta\gamma/p-1} - 4n\bigr] M_\beta(h;k).
		\end{split}
	\]
	If $\lambda$ is sufficiently large (precisely as stated in the assumptions of the theorem),
	the square bracket is positive, hence we must have
	\[
		h_1=k_1=\dots=h_n=k_n=0.
	\]

	This implies that $\TV(u;[x_0,x_n])\geq\TV(g;[x_0,x_n])$, therefore $F(u)\geq F(g)$ and $g$ is a
	minimizer.

	Let us show that $u=g$ on $I\setminus J_g$. If it is not the case, there is $x\in I\setminus J_g$
	such that $u(x)\neq g(x)$. Since $u$ is regular, we have either $u^-(x)\neq g(x)$ or $u^+(x)\neq g(x)$.
	Moreover, $x$ is a continuity point for $g$, hence there is a small interval of the form
	$(x-\eps,x)$ or $(x,x+\eps)$ where $u$ and $g$ differ, leading to $\norm{u-g}_{L^p(I)}>0$ and
	contradicting the optimality of $u$.
\end{proof}

\begin{remark}\label{rmk:sharpness}
	Notice that the assumption on $\gamma$ in \autoref{thm:not-min-alpha-growth}%
	\footnote{As a particular case, \autoref{thm:not-min-cantor-function} falls in this setting too, with
		$\alpha=\log2/\log3$.}
	is complementary to that in \autoref{thm:min-monotone} and \autoref{thm:min-piecewise-monotone},
	demonstrating the sharpness of our result.
\end{remark}

\section{Nonlinear \texorpdfstring{$\TVphitheta$}{TV\_phi\_theta}}\label{sec:nonlinear}

In this section, we tackle the problem with the nonlinear total variation $\TVphitheta$. As we will
see in the ensuing sections, this case is remarkably more difficult than the linear case. In
particular, the functional does not behave well with respect to local truncations. To overcome this
difficulty, we study a suitable boundary value problem in \autoref{sec:bvp}, characterize its
minimizers, and use them to construct competitors for the problem with the nonlinear total
variation. \autoref{sec:nonlinear-not-min} is the nonlinear counterpart to
\autoref{sec:linear-not-min}, whereas the results in \autoref{sec:linear-min} cannot be reproduced
in this setting: in \autoref{sec:nonlinear-min} we prove \autoref{prop:nonlinear-competitors} where
we show that different behaviors of $\phi$ at infinity can disprove both \autoref{thm:min-monotone}
and \autoref{thm:min-piecewise-monotone}. As a consequence, we propose to study the penalized
functional in \eqref{eq:penalized-functional}, for which numerical results illustrate the behavior
of minimizers.

\subsection{Boundary value problem}\label{sec:bvp}

We want to minimize the nonlinear total variation on an interval of a function with prescribed limits
at the endpoints, which can be formulated as the optimization problem
\begin{equation}\label{eq:E-def}
	E(L,H) = \inf\set*{\TVphitheta\bigl(u;(0,L)\bigr)}{u\in\BV\bigl((0,L)\bigr),\ u^+(0)=0,\ u^-(L)=H},
\end{equation}
where $L>0$ and $H\geq0$ are two prescribed constants.

\begin{theorem}\label{thm:bvp}
	The problem \eqref{eq:E-def} admits at least a minimizer $u\in\BVreg\bigl((0,L)\bigr)$ of the form%
	\footnote{$\Theta(\plchldr)$ is the Heaviside step function.}
	\[
		u(x) = ax + b\Theta(x-\bar x),
	\]
	for some $a,b\geq0$ and $\bar x\in(0,L)$, hence
	\begin{equation}\label{eq:E-min-ab}
		\begin{split}
			E(L,H)
			&= \min\set*{L\phi(a)+\theta(b)}{a,b\geq0,\ La+b=H} \\
			&= \min\set*{L\phi(a)+\theta(H-La)}{0\leq a\leq H/L} \\
			&= \min\set*{L\phi\bigl((H-b)/L\bigr)+\theta(b)}{0\leq b\leq H} \\
			&= \min\set*{\phi_L(H\alpha) + \theta\bigl(H(1-\alpha)\bigr)}{0\leq\alpha\leq1}.
		\end{split}
	\end{equation}
	In particular, there is at least one minimizer which has no Cantor part.
	If $\phi$ is strictly convex and $\theta$ is strictly concave then any minimizer is of this form.
\end{theorem}

\begin{proof}
	Let $I=(0,L)$, let $u$ be a minimizer, let $Du=u'\leb^1+D^cu+[u]\haus^0\rvert_{J_u}$ be the
	decomposition of its derivative, and let
	\begin{align*}
		A & = \abs{\D^au}(I), &
		C & = \abs{\D^cu}(I), &
		J & = \abs{\D^ju}(I).
	\end{align*}

	Observe that $H = \D u(I) = \D^au(I)+\D^cu(I)+\D^ju(I) \leq A+C+J$, with equality if and only if
	all three measures are non-negative, i.e.\ $u$ is non-decreasing.

	\begin{itemize}
		\item \textit{Monotonicity}. The function $\tilde u$ such that
		      $\D\tilde u=\frac{H}{A+C+J}\left(\abs{\D^au}+\abs{\D^cu}+\abs{\D^ju}\right)$
		      is non-decreasing and satisfies
		      \[
			      \begin{split}
				      \TVphitheta(\tilde u)
				      &= \int_0^L \phi\oleft(\frac H{A+C+J}\abs{u'(x)}\right)\dx
				      + \frac H{A+C+J}C + \sum_{x\in J_u} \theta\oleft(\frac H{A+C+J}\abs{[u](x)}\right) \\
				      &\leq \int_0^L \phi\bigl(\abs{u'(x)}\bigr)\dx
				      + C + \sum_{x\in J_u} \theta\bigl(\abs{[u](x)}\bigr)
				      = \TVphitheta(u)
			      \end{split}
		      \]
		      because $\frac H{A+C+J}\leq1$ and $\phi$ and $\theta$ are non-decreasing.
		      This implies that there is always a non-decreasing minimizer.

		      Let us investigate the equality conditions.
		      If $C>0$ we must have $H=A+C+J$, i.e.\ $u$ non-decreasing.
		      If $A>0$ we must have $\phi\oleft(\frac H{A+C+J}\abs{u'(x)}\right) = \phi(\abs{u'})$,
		      which implies $H=A+C+J$, i.e.\ $u$ non-decreasing, if $\phi$ is strictly increasing.%
		      \footnote{Notice that $\phi$ strictly convex implies $\phi$ strictly increasing.}
		      If $J>0$ we must have $\theta\oleft(\frac H{A+C+J}\abs{[u](x)}\right) = \theta(\abs{[u](x)})$,
		      which implies $H=A+C+J$, i.e.\ $u$ non-decreasing, if $\theta$ is strictly increasing.%
		      \footnote{Notice that $\theta$ strictly concave implies $\theta$ strictly increasing.}

		      This means that a minimizer must be non-decreasing if $\phi$ and $\theta$ are strictly increasing.
		\item \textit{Transfer Cantor part to absolutely continuous part}. Assume that $u$ is non-decreasing.
		      Observe that $\phi(\beta)-\phi(\alpha)\leq\beta-\alpha$ for every $0\leq\alpha<\beta$;
		      moreover, if $\phi$ is strictly convex, then this inequality is strict for every $\alpha<\beta$.

		      \begin{itemize}
			      \item Assume $A>0$. Then the function $\tilde u$ such that
			            $\D\tilde u = \left(1+\frac CA\right)\D^au+\D^ju$ satisfies
			            \[
				            \begin{split}
					            \TVphitheta(\tilde u)
					            &= \int_0^L \phi\bigl(\tilde u'(x)\bigr) \dx
					            + \sum_{x\in J_u} \theta\bigl(\abs{[\tilde u](x)}\bigr) \\
					            &= \int_0^L \phi\oleft(\left(1+\frac CA\right)u'(x)\right) \dx
					            + \sum_{x\in J_u} \theta\bigl(\abs{[u](x)}\bigr) \\
					            &\leq \int_0^L \phi\bigl(u'(x)\bigr) \dx
					            + \int_0^L \frac CA u'(x) \dx
					            + \sum_{x\in J_u} \theta\bigl(\abs{[u](x)}\bigr) \\
					            &= \int_0^L \phi\bigl(u'(x)\bigr) \dx + C
					            + \sum_{x\in J_u} \theta\bigl(\abs{[u](x)}\bigr)
					            = \TVphitheta(u).
				            \end{split}
			            \]
			            In this case, if $\phi$ is strictly convex, equality holds if and only if $\frac CAu'(x)=0$ for
			            almost every $x$, hence $C=0$.

			      \item If $A=0$, then the function $\tilde u$ such that $\D\tilde u = \frac CL + \D^ju$ satisfies
			            \[
				            \begin{split}
					            \TVphitheta(\tilde u)
					            &= \int_0^L \phi\oleft(\frac CL\right) \dx
					            + \sum_{x\in J_u} \theta\bigl(\abs{[\tilde u](x)}\bigr) \\
					            &\leq \int_0^L \frac CL \dx
					            + \sum_{x\in J_u} \theta\bigl(\abs{[u](x)}\bigr) \\
					            &= C + \sum_{x\in J_u} \theta\bigl(\abs{[u](x)}\bigr)
					            = \TVphitheta(u).
				            \end{split}
			            \]
			            If $\phi\neq\Id$, then $\phi(t)<t$ for every $t>0$, therefore equality holds
			            if and only if $C/L=0$, hence $C=0$.
		      \end{itemize}
		\item \textit{Transfer Cantor part to jump part.} Assume that $u$ is non-decreasing.
		      Observe that the function $\theta$ is subadditive, i.e.\
		      $\theta(\alpha+\beta)\leq\theta(\alpha)+\theta(\beta)$ for every $\alpha,\beta\geq0$;
		      moreover, if $\theta$ is strictly concave, then this inequality is strict for every
		      $\alpha,\beta>0$.
		      \begin{itemize}
			      \item Assume $J>0$. Then the function $\tilde u$ such that
			            $\D\tilde u = \D^au+\left(1+\frac CJ\right)\D^ju$ satisfies
			            \[
				            \begin{split}
					            \TVphitheta(\tilde u)
					            &= \int_0^L \phi\bigl(\abs{u'(x)}\bigr)\dx
					            + \sum_{x\in J_u} \theta\oleft(\left(1+\frac{C}{J}\right)\abs{[u](x)}\right) \\
					            &\leq \int_0^L \phi\bigl(\abs{u'(x)}\bigr)\dx
					            + \sum_{x\in J_u} \theta\bigl(\abs{[u](x)}\bigr)
					            + \sum_{x\in J_u} \theta\oleft(\frac CJ\abs{[u](x)}\right) \\
					            &\leq \int_0^L \phi\bigl(\abs{u'(x)}\bigr)\dx
					            + \sum_{x\in J_u} \theta\bigl(\abs{[u](x)}\bigr)
					            + \sum_{x\in J_u} \frac CJ\abs{[u](x)} \\
					            &= \int_0^L \phi\bigl(\abs{u'(x)}\bigr)\dx
					            + \sum_{x\in J_u} \theta\bigl(\abs{[u](x)}\bigr)
					            + C
					            = \TVphitheta{(u)}.
				            \end{split}
			            \]
			            If $\theta(t)<t$ for every $t>0$, then equality holds if and only if
			            $\frac CJ\abs{[u](x)}=0$ for every $x\in J_u$, hence $C=0$.
			      \item If $J=0$, then the function $\tilde u$ such that $\D\tilde u = \D^au + C\delta_{\bar x}$, with $\bar x\in I$, satisfies
			            \[
				            \begin{split}
					            \TVphitheta(\tilde u)
					            &= \int_0^L \phi\bigl(\abs{u'(x)}\bigr)\dx + \theta(C)
					            \leq \int_0^L \phi\bigl(\abs{u'(x)}\bigr)\dx + C
					            = \TVphitheta(u).
				            \end{split}
			            \]
			            If $\theta(t)<t$ for every $t>0$, then equality holds if and only if $C=0$.
		      \end{itemize}

		\item \textit{Transfer Cantor part to jump part, additively (optional)}.
		      This step is presented as an alternative way to transfer the Cantor part to the jump part.
		      Let $\tilde u$ be a function such that $\D\tilde u=\D^a u+\D^ju+\rho+\sigma$
		      where $\rho,\sigma$ are atomic measures satisfying
		      $\rho(I)+\sigma(I)=\D^cu(I)$, $\abs{\rho}(I)+\abs{\sigma}(I)=\abs{\D^cu}(I)$,
		      $\supp(\rho)\subseteq J_u$, $\sigma\perp D^ju$, $S=\supp(\sigma)$.
		      Then
		      \[
			      \begin{split}
				      \TVphitheta(\tilde u)
				      &= \int_0^L \phi(\abs{u'(x)})\dx
				      + \sum_{x\in J_u} \theta\bigl(\abs{[u](x)+\rho(x)}\bigr)
				      + \sum_{x\in S} \theta\bigl(\abs{\sigma(x)}\bigr) \\
				      &\leq \int_0^L \phi(\abs{u'(x)})\dx
				      + \sum_{x\in J_u} \theta\bigl(\abs{[u](x)}\bigr)
				      + \sum_{x\in J_u} \theta\bigl(\abs{\rho(x)}\bigr)
				      + \sum_{x\in S} \theta\bigl(\abs{\sigma(x)}\bigr) \\
				      &\leq \int_0^L \phi(\abs{u'(x)})\dx
				      + \sum_{x\in J_u} \theta\bigl(\abs{[u](x)}\bigr)
				      + \sum_{x\in J_u} \abs{\rho(x)}
				      + \sum_{x\in S} \abs{\sigma(x)} \\
				      &\leq \int_0^L \phi(\abs{u'(x)})\dx
				      + \sum_{x\in J_u} \theta\bigl(\abs{[u](x)}\bigr)
				      + \abs{\D^cu}(I)
				      \leq \TVphitheta(u).
			      \end{split}
		      \]

		      This works in particular by taking either $\rho$ or $\sigma$ to be $\D^cu(I)\delta_{\bar x}$
		      and the other equal to zero.

		      Since this step is optional, we do not discuss here the equality conditions.
		\item \textit{Averaging the absolutely continuous part}. By Jensen inequality we have
		      \[
			      \phi\oleft(\abs*{\frac{\D^au(I)}L}\right)
			      = \phi\oleft(\abs*{\dashint_0^L u'(x)\dx}\right)
			      \leq \dashint_0^L \phi\bigl(\abs{u'(x)}\bigr)\dx
			      = \frac1L \int_0^L \phi\bigl(\abs{u'(x)}\bigr)\dx;
		      \]
		      moreover, if $\phi$ is strictly convex, then equality holds if and only if $u'$
		      is constantly equal to $\D^au(I)/L$.%
		      \footnote{Notice that the function $\phi\circ\abs{\plchldr}:\setR\to\setR$ is strictly
			      convex if and only if $\phi$ is strictly convex.}

		      As a consequence, the function $\tilde u$ such that
		      $\D\tilde u = \frac{\D^au(I)}L\leb^1 + \D^cu + \D^ju$
		      satisfies
		      \[
			      \begin{split}
				      \TVphitheta(\tilde u)
				      &= \int_0^L \phi\oleft(\abs*{\frac{\D^au(I)}L}\right) \dx
				      + C + \sum_{x\in J_u} \theta\bigl(\abs{[u](x)}\bigr) \\
				      &\leq \int_0^L \phi\bigl(\abs{u'(x)}\bigr)\dx
				      + C + \sum_{x\in J_u} \theta\bigl(\abs{[u](x)}\bigr)
				      = \TVphitheta(u),
			      \end{split}
		      \]
		      with equality if and only if $u'$ is constantly equal to $\D^au(I)/L$,
		      under the strict convexity assumption of $\phi$.
		\item \textit{Coalescing jumps into a single one}.
		      The function $\tilde u$ such that $\D\tilde u = \D^au + \D^cu + \D^ju(I)\delta_{\bar x}$ for some
		      $\bar x\in I$ satisfies
		      \[
			      \begin{split}
				      \TVphitheta(\tilde u)
				      &= \int_0^L \phi\bigl(\abs{u'(x)}\bigr)\dx + \abs{\D^cu}(I)
				      + \theta\bigl(\abs{D^ju(I)}\bigr) \\
				      &= \int_0^L \phi\bigl(\abs{u'(x)}\bigr)\dx + \abs{\D^cu}(I)
				      + \theta\left(\abs*{\sum_{x\in J_u}[u](x)}\right) \\
				      &\leq \int_0^L \phi\bigl(\abs{u'(x)}\bigr)\dx + \abs{\D^cu}(I)
				      + \theta\left(\sum_{x\in J_u}\abs{[u](x)}\right) \\
				      &\leq \int_0^L \phi\bigl(\abs{u'(x)}\bigr)\dx + \abs{\D^cu}(I)
				      + \sum_{x\in J_u} \theta\bigl(\abs{[u](x)}\bigr)
				      = \TVphitheta(u).
			      \end{split}
		      \]
		      Let us investigate the equality conditions.
		      In the first inequality, equality holds if and only if $\abs{\D^ju(I)}=\abs{\D^ju}(I)$
		      if $\theta$ is strictly increasing;
		      in the second inequality, equality holds if and only if $u$ has at most one jump
		      if $\theta$ is strictly concave.
	\end{itemize}

	If $\phi$ is strictly convex and $\theta$ is strictly concave, then a minimizer $u$ must necessarily be
	non-decreasing, without any Cantor part ($\D^cu=0$), with $u'$ constant and at most one jump.
	Such functions can be written as $u(x)=ax+b\Theta(x-\bar x)$, where $\Theta$ is the Heaviside step function,
	with $\bar x\in(0,L)$ and $a,b\in[0,\infty)$, in which case the functional to minimize is
	$\TVphitheta(u)=L\phi(a)+\theta(b)$.

	The admissibility condition for the boundary value problem is $aL+b=H$, so we need to find
	$a\in[0,H/L]$ that minimizes $L\phi(a)+\theta(H-La)$, or equivalently $b\in[0,H]$ that minimizes
	$L\phi\bigl((H-b)/L\bigr)+\theta(b)$.
	The last formula in \eqref{eq:E-min-ab} follows letting $\alpha=\frac LHa$: then $a\in[0,H/L]$
	if and only if $\alpha\in[0,1]$ and we have
	\[
		L\phi(a) + \theta(H-La)
		= L\phi\left(\frac HL\alpha\right) + \theta\bigl(H(1-\alpha)\bigr)
		= \phi_L(H\alpha) + \theta\bigl(H(1-\alpha)\bigr). \qedhere
	\]
\end{proof}

\begin{remark}
	If $\phi$ and $\theta$ are differentiable, then a pair $(a,b)$ optimal for \eqref{eq:E-min-ab}
	and such that $a,b>0$ must necessarily satisfy $\phi'(a)=\theta'(b)$.
	Notice that in general there can be a finite, countable or uncountable number of minimizers of
	\eqref{eq:E-min-ab} and pairs satisfying the previous condition on the derivative.
\end{remark}

\begin{remark}
	If we extend $\phi,\theta$ to $\bar\phi,\bar\theta:\setR\to[0,\infty]$ by setting
	$\bar\phi(t)=\bar\theta(t)=+\infty$ for $t<0$,%
	\footnote{Notice that $\bar\phi$ is still convex, whereas $\bar\theta$ is neither convex
		nor concave globally.}
	then \eqref{eq:E-min-ab} can be rewritten as
	\[
		\begin{split}
			E(L,H)
			&= \min_{0\leq b\leq H} \left(L\phi\oleft(\frac{H-b}L\right)+\theta(b)\right)
			= \min_{b\geq0} \left(L\bar\phi\oleft(\frac{H-b}L\right)+\theta(b)\right) \\
			&= \min_{b\in\setR} \left(L\bar\phi\oleft(\frac{H-b}L\right)+\bar\theta(b)\right)
			= \min_{b\in\setR} \bigl(\bar\phi_L(H-b)+\bar\theta(b)\bigr)
			= (\bar\phi_L\square\bar\theta)(H),
		\end{split}
	\]
	where $\bar\phi_L(t)=L\bar\phi(t/L)$ is the perspective of $\bar\phi$ and $\square$ denotes the $\inf$-convolution.
\end{remark}

In the following propositions we study several properties of $E(L,H)$:
monotonicity and subadditivity (\autoref{prop:E-monotone}), upper bounds (\autoref{prop:E-upper-bounds}),
lower bounds (\autoref{prop:E-lower-bounds}), and continuous extension (\autoref{prop:E-continuous}).

\begin{proposition}\label{prop:E-monotone}
	The following properties are true.
	\begin{enumerate}
		\item\label{it:E-monotone} $E(L,H)$ is non-increasing in $L$ and non-decreasing in $H$.
		\item $E$ is jointly subadditive: $E(L_1+L_2,H_1+H_2) \leq E(L_1,H_1)+E(L_2,H_2)$.
	\end{enumerate}
\end{proposition}

\begin{proof}
	We prove the two properties separately.
	\begin{enumerate}
		\item In the last representation formula in \eqref{eq:E-min-ab}, observe that for every
		      $\alpha\in[0,1]$ the function $\phi_L(H\alpha)+\theta\bigl(H(1-\alpha)\bigr)$ to be minimized
		      enjoys the claimed monotonicity (the monotonicity in $L$ follows from the monotonicity of
		      the perspective $\phi_L$ mentioned in \autoref{sec:perspective}).

		      In \autoref{sec:E-monotone-alt-proofs} we provide two alternative proofs of this property.

		\item Let $u_1\in\BV((0,L_1))$ and $u_2\in\BV((0,L_2))$ be optimal for $E(L_1,H_1)$ and $E(L_2,H_2)$ respectively. Define $u\in\BV\bigl((0,L_1+L_2)\bigr)$ by
		      \[
			      u(x) = \begin{cases}
				      u_1(x)           & x\in(0,L_1),       \\
				      H_1              & x=L_1,             \\
				      H_1 + u_2(x-L_1) & x\in(L_1,L_1+L_2).
			      \end{cases}
		      \]
		      Observe that $u$ is continuous in $L_1$ and has the correct limits at $0$ and $L_1+L_2$
		      to be admissible for $E(L_1+L_2,H_1+H_2)$, therefore
		      \[
			      \begin{split}
				      E(L_1+L_2,H_1+H_2)
				      &\leq \TVphitheta\bigl(u,(0,L_1+L_2)\bigr) \\
				      &= \TVphitheta\bigl(u,(0,L_1)\bigr)
				      + \TVphitheta\bigl(u,(L_1,L_1+L_2)\bigr) \\
				      &= \TVphitheta\bigl(u_1,(0,L_1)\bigr)
				      + \TVphitheta\bigl(u_2,(0,L_2)\bigr) \\
				      &= E(L_1,H_1) + E(L_2,H_2). \qedhere
			      \end{split}
		      \]
	\end{enumerate}
\end{proof}

\begin{proposition}\label{prop:E-upper-bounds}
	We have the upper bounds
	\begin{align}\label{eq:E-upper-bounds}
		E(L,H) & \leq \phi_L(H) = L\phi(H/L), &
		E(L,H) & \leq \theta(H).
	\end{align}
\end{proposition}

\begin{proof}
	These upper bounds follow directly from \eqref{eq:E-min-ab}, taking $a=H/L,b=0$ for the former and
	$a=0,b=H$ for the latter.
\end{proof}

\begin{proposition}\label{prop:E-lower-bounds}
	We have the lower bounds
	\begin{equation}\label{eq:E-lower-bound-simple}
		E(L,H) \geq \phi_L(H) + \theta(H) - H,
	\end{equation}
	and
	\begin{equation}\label{eq:E-lower-bound-complex}
		\begin{split}
			E(L,H)
			&\geq \phi_L(H) - \bigl(\phi_L'^-(H)H-\theta(H)\bigr)_+ \\
			&= \begin{cases}
				\phi_L(H)                         & \phi_L'^-(H) \leq \theta(H)/H, \\
				\phi_L(H)-\phi_L'^-(H)H+\theta(H) & \phi_L'^-(H) \geq \theta(H)/H,
			\end{cases}
		\end{split}
	\end{equation}
	where $\phi_L'^-(H)=\min\partial\phi_L(H)$ is the smallest element of the subdifferential.\footnote{
		The estimate is available for every element of the subdifferential, but taking the smallest
		leads to the best one.}
\end{proposition}

\begin{proof}
	Let us prove the first lower bound.
	Define the complementary functions $\psi,\rho:[0,\infty)\to[0,\infty)$ by $\psi(t)=t-\phi(t)$ and
	$\rho(t)=t-\theta(t)$.
	Observe that $\psi(0)=\rho(0)=0$, $\psi$ is concave, increasing and sublinear at $+\infty$,%
	\footnote{In the sense that $\lim_{t\to\infty} \psi(t)/t=0$.}
	whereas $\rho$ is convex and $\rho'(0)=0$, hence increasing too.
	If we define also $\psi_L(t)=L\psi(t/L)$, then we have $\psi_L(t)=t-\phi_L(t)$.
	Exploiting the monotonicity of $\psi_L$ and $\rho$, for every $\alpha\in[0,1]$ we have
	\[
		\begin{split}
			\phi_L(H\alpha) + \theta\bigl(H(1-\alpha)\bigr)
			&= H\alpha - \psi_L(H\alpha) + H(1-\alpha) - \rho\bigl(H(1-\alpha)\bigr) \\
			&= H - \psi_L(H\alpha) - \rho\bigl(H(1-\alpha)\bigr) \\
			&\geq H - \psi_L(H) - \rho(H)
			= \phi_L(H) + \theta(H) - H.
		\end{split}
	\]

	Let us now prove the second lower bound.
	Since $\phi_L$ is convex and $\theta$ is concave, for every $t\in[0,H]$ we have
	$\phi_L(t) \geq \phi_L(H) + \phi_L'^-(H)(t-H)$ and $\theta(t) \geq \theta(H)/H t$, therefore
	\[
		\begin{split}
			\phi_L(H\alpha) + \theta\bigl(H(1-\alpha)\bigr)
			&\geq \phi_L(H) + \phi_L'^-(H)(H\alpha-H) + \theta(H)(1-\alpha) \\
			&\geq \phi_L(H) - \bigl(\phi_L'^-(H)H - \theta(H)\bigr)(1-\alpha).
		\end{split}
	\]
	If $\phi_L'^-(H)H\leq\theta(H)$, the minimum is achieved for $1-\alpha=0$ and the lower bound is
	just $\phi_L(H)$; otherwise it is achieved for $1-\alpha=1$, and the lower bound is
	$\phi_L(H) - \phi_L'^-(H)H + \theta(H)(1-\alpha)$.
\end{proof}

\begin{remark}\label{rmk:E-lower-bound}
	If $\phi'(0)=1$, then necessarily $\phi(t)=t$ for every $t\in[0,\infty)$.
	In \eqref{eq:E-lower-bound-complex} we always have $\theta(H)/H\leq1=\phi_L'(H)$, therefore
	$E(L,H) \geq \theta(H)$. Combined with \eqref{eq:E-upper-bounds} we get $E(L,H)=\theta(H)$.

	If on the other hand $\phi'(0)<1$, since $\lim_{H\to0^+} \theta(H)/H=1$ we have that for $H$
	sufficiently small $\theta(H)/H\geq\phi_L'^-(H)$, hence from \eqref{eq:E-lower-bound-complex} we
	get $E(L,H)\geq\phi_L(H)$. Combined with \eqref{eq:E-upper-bounds} we get $E(L,H)=\phi_L(H)$ for
	$H$ sufficiently small.

	More in general, one can rewrite the right hand side of the lower bound \eqref{eq:E-lower-bound-simple} as
	$\phi_L(H)-\rho(H)=\theta(H)-\psi_L(H)$. Combined with the upper bounds \eqref{eq:E-upper-bounds},
	this leads to
	\begin{align*}
		\phi_L(H) - \rho(H)   & \leq E(L,H) \leq \phi_L(H), &
		\theta(H) - \psi_L(H) & \leq E(L,H) \leq \theta(H).
	\end{align*}
	Since $\rho(H)=o(H)$ for $H\to0$, the former estimate is more useful for $H$ small, whereas the
	latter is more useful for large values of $H$, especially when $\theta$ is sublinear.
\end{remark}

\begin{proposition}\label{prop:E-continuous}
	$E(L,H)$ is continuous in $(0,\infty)\times[0,\infty)$ and it admits a continuous extension to
	$[0,\infty)\times[0,\infty)$, with $E(0,H)=\theta(H)$.
\end{proposition}

\begin{proof}
	Thanks to \autoref{prop:perspective}, the function
	\[
		g(L,H,\alpha) = \phi_L(H\alpha) + \theta\bigl(H(1-\alpha)\bigr)
		= \phi(H\alpha,L) + \theta\bigl(H(1-\alpha)\bigr)
	\]
	is continuous in $[0,\infty)\times[0,\infty)\times[0,1]$, hence it is uniformly continuous in
	$[0,\bar L]\times[0,\bar H]\times[0,1]$ for every $\bar L,\bar H>0$. From this we deduce that
	in $(0,\bar L)\times[0,\bar H)$ the function
	\[
		E(L,H) = \min_{\alpha\in[0,1]} g(L,H,\alpha)
	\]
	is continuous.
	Recalling that $\phi'(\infty)=1$ and therefore $\phi(H,0)=H$,
	the continuous extension to $L=0$ is provided by the minimization problem in the right hand side,
	i.e.
	\[
		\begin{split}
			\min_{\alpha\in[0,1]} g(0,H,\alpha)
			&= \min_{\alpha\in[0,1]} \bigl[\phi(H\alpha,0) + \theta\bigl(H(1-\alpha)\bigr)\bigr]
			= \min_{\alpha\in[0,1]} \bigl[H\alpha + \theta\bigl(H(1-\alpha)\bigr)\bigr] \\
			&= \min_{\alpha\in[0,1]} \bigl[H - \rho\bigl(H(1-\alpha)\bigr)\bigr]
			= H - \rho(H) = \theta(H). \qedhere
		\end{split}
	\]
\end{proof}

\subsection{The input is not a minimizer}\label{sec:nonlinear-not-min}

Recall from \autoref{rmk:not-min-cantor-function} that the Cantor function provides a simple example
where the input is not a minimizer even in the nonlinear setting.
This is generalized in the following theorem to an input datum $g$ with at least $\alpha$-growth
at one endpoint, providing the nonlinear analogue of \autoref{thm:not-min-alpha-growth}.

\begin{theorem}\label{thm:not-min-alpha-growth-nonlinear}
	Let $g\in\BVreg\bigl((a,b)\bigr)$ be a function with the property that there exist $\alpha,c>0$ and
	$0<\delta<b-a$ such that $g(x)-g^+(a)\geq c(x-a)^\alpha$ for all $x\in(a,a+\delta]$.
	Assume that there exist $C,\eta>0$ such that $E(\delta,H)\geq C H^\eta$ for every $H\in[0,\bar H]$,
	where $\bar H=\esssup_{(a,a+\delta]}g-g^+(a)$.
	If
	\[
		\gamma > \eta \left(1+\frac1{p\alpha}\right)^{-1} = \frac{p\alpha\eta}{p\alpha+1}
	\]
	then $g$ is not a minimizer of the functional
	\[
		F(u) = \TVphitheta(u) + \lambda\norm{u-g}_p^\gamma
	\]
	for any $p\in[1,\infty]$ and $\lambda>0$.
\end{theorem}

\begin{proof}
	Define the non-decreasing function $\omega:(a,a+\delta]\to\setR$
	\[
		\omega(x) = \esssup_{(a,x]} g
	\]
	and observe that we have $\omega(x) \geq g^+(a) + c(x-a)^\alpha$ and $\omega^+(a)=g^+(a)$.
	Given $y\in(a,a+\delta]$, introduce the auxiliary parameter $h=h(y)=\omega(y)-g^+(a)\geq c(y-a)^\alpha>0$.

	For every $\eps>0$, the set $\set{x\in(a,y]}{g(x)\geq\omega(y)-\eps}$ has positive $\leb^1$-measure,
	and $\leb^1$-a.e.\ $x\in(a,y]$ is a continuity point of $g$, therefore there exists
	$x_{y,\eps}\in(a,y]$ such that $g(x_{y,\eps})\geq\omega(y)-\eps$ and $g$ is continuous in $x_{y,\eps}$.
	From this continuity, we deduce $g(x_{y,\eps})\leq\omega(y)$, hence $\abs{\omega(y)-g(x_{y,\eps})}\leq\eps$.
	Define the truncated function $u_{y,\eps}\in\BV\bigl((a,b)\bigr)$
	\[
		u_{y,\eps}(x) = \begin{cases} \omega(y) & x\in(a,x_{y,\eps}), \\ g(x) & x\in[x_{y,\eps},b). \end{cases}
	\]
	We have
	\[
		\begin{split}
			\TVphitheta\bigl(u_{y,\eps};(a,b)\bigr)
			&= \TVphitheta\bigl(u_{y,\eps};(a,x_{y,\eps})\bigr)
			+ \theta(\abs{u_{y,\eps}^-(x_{y,\eps})-u_{y,\eps}^+(x_{y,\eps})})
			+ \TVphitheta\bigl(u_{y,\eps};(x_{y,\eps},b)\bigr) \\
			&= \theta\bigl(\abs{\omega(y)-g(x_{y,\eps})}\bigr) + \TVphitheta\bigl(g;(x_{y,\eps},b)\bigr) \\
			&\leq \theta(\eps) + \TVphitheta\bigl(g;(x_{y,\eps},b)\bigr).
		\end{split}
	\]
	On the other hand, recalling the continuity of $g$ in $x_{y,\eps}$, $x_{y,\eps}-a\leq y-a\leq\delta$,
	the monotonicity stated in \autoref{prop:E-monotone}, and the lower bound for $E(\delta,\plchldr)$,
	we have
	\[
		\begin{split}
			\TVphitheta\bigl(g;(a,b)\bigr)
			&= \TVphitheta\bigl(g;(a,x_{y,\eps})\bigr) + \TVphitheta\bigl(g;(x_{y,\eps},b)\bigr) \\
			&\geq E\bigl(x_{y,\eps}-a,\abs{g(x_{y,\eps})-g^+(a)}\bigr) + \TVphitheta\bigl(g;(x_{y,\eps},b)\bigr) \\
			&\geq E\bigl(\delta,\abs{\omega(y)-g^+(a)}-\eps\bigr) + \TVphitheta\bigl(g;(x_{y,\eps},b)\bigr) \\
			&= E\bigl(\delta,h-\eps\bigr) + \TVphitheta\bigl(g;(x_{y,\eps},b)\bigr) \\
			&\geq C(h-\eps)^\eta + \TVphitheta\bigl(g;(x_{y,\eps},b)\bigr).
		\end{split}
	\]
	From these two total variation estimates we get
	\[
		\TVphitheta\bigl(g;(a,b)\bigr)
		\geq \TVphitheta\bigl(u_{y,\eps};(a,b)\bigr)
		+ C(h-\eps)^\eta
		- \theta(\eps).
	\]

	Since for $\leb^1$-a.e.\ $x\in(a,a+\delta]$ we have $u_{y,\eps}(x)\geq g(x)$,
	with a similar computation as in \autoref{thm:not-min-alpha-growth}, we have
	\[
		\norm{u_{y,\eps}-g}_{L^p((a,b))}^p
		\leq K(p,\alpha) C^{-1/\alpha} h^{p+1/\alpha}.
	\]

	Combining the estimates for the total variation and the fidelity term we arrive at
	\[
		F(u_{y,\eps})-F(g)
		\leq \theta(\eps) - C (h-\eps)^\eta + \lambda K(p,\alpha)^{\gamma/p} c^{-\gamma/(p\alpha)} h^{\gamma\bigl(1+\frac{1}{p\alpha}\bigr)},
	\]
	which can be made strictly negative under the assumption on the exponent $\gamma$
	sending $\eps\to0$ and $y\to0$, which implies $h(y)\to0$.
\end{proof}

Observe that \autoref{prop:E-lower-bounds} provides a convenient way to verify the assumption on
$E(\delta,\plchldr)$.
For instance, if $\phi(t)=c t^\eta+o(t^\eta)$ for $t\to0$ with $\eta>1$, then by
\autoref{rmk:E-lower-bound} we have $E(L,H)=\phi_L(H)=cL^{1-\eta}H^\eta+o(H^\eta)$ for $H\to0$, hence
the assumption is satisfied.

\subsection{Is the input a minimizer?}\label{sec:nonlinear-min}

The intent of this section is to try and replicate the results contained in \autoref{sec:linear-min}.
Unfortunately, as it turns out, the expected statements are not true. Indeed, the culprit lies in
the failure of the nonlinear analogue of \autoref{prop:u_contained_in_g}: more specifically, while
it would be possible to replicate the majority of its proof, the sub-case \ref{it:2b} proves to be an insurmountable obstruction, due to the fact that not all monotone functions are minimizers of the boundary value problem for $\TVphitheta$, a simple fact that was exploited in \ref{it:2b} to perform the local truncation without degrading the linear total variation.

Case \ref{it:nonlinear-competitors-g-not-min} of \autoref{prop:nonlinear-competitors} demonstrates a
simple situation where the nonlinear analogues of \autoref{prop:u_contained_in_g} and
\autoref{thm:min-piecewise-monotone} both fail, exploiting the behavior at infinity of the
nonlinearity $\phi$, see \eqref{eq:phi-asymptote-no}. The sharpness of the assumption is highlighted
by the opposite behavior under the complementary conditions in case
\ref{it:nonlinear-competitors-g-maybe-min} of the same proposition. In the particular case
$p=\gamma$, conditions \eqref{eq:phi-asymptote} amount to the existence or non-existence of an
oblique asymptote for $\phi$; in the case $\gamma>p$, they are related to a sufficiently fast
convergence rate to the oblique asymptote.

A more detailed analysis of \ref{it:2b} leads us to consider the penalized problem in
\autoref{prop:not-min-zero-nonlinear}, which reveals the underlying reason for the failure of the
desired nonlinear statements: constant functions are not optimal profiles for transitioning between
two levels, not even in the presence of a fidelity term with respect to a constant datum.

\begin{proposition}\label{prop:nonlinear-competitors}
	Let $L,H>0$. Let $g\in\BV\bigl((0,L)\bigr)$ given by
	\[
		g(x) = H\Theta(x-x_0) = \begin{cases} 0 & x < x_0, \\ H & x > x_0, \end{cases}
	\]
	with $x_0\in(0,L)$. For $\alpha\geq\alpha^*\coloneqq\frac H{2\min\{x_0,L-x_0\}}$, define the competitor
	$u_\alpha\in\BV\bigl((0,L)\bigr)$ as
	\[
		u_\alpha(x)
		= \left(\frac H2+\alpha(x-x_0)\right)\vee0\wedge H
		= \begin{cases}
			0                        & x<x_0-\frac H{2\alpha},                      \\
			\frac H2 + \alpha(x-x_0) & x_0-\frac H{2\alpha}<x<x_0+\frac H{2\alpha}, \\
			H                        & x<x_0-\frac H{2\alpha}.
		\end{cases}
	\]
	\begin{subequations}\label{eq:phi-asymptote}
		\begin{enumerate}
			\item \label{it:nonlinear-competitors-g-not-min} If $\theta(H)=H$ and
			      \begin{equation}\label{eq:phi-asymptote-no}
				      \inf_{\alpha\geq\alpha^*} \frac{\phi(\alpha)-\alpha}{\alpha^{1-\gamma/p}} = -\infty,
			      \end{equation}
			      then for every $\lambda\geq0$ there exists $\alpha\geq\alpha^*$ such that $F(u_\alpha)<F(g)$;
			      hence the input $g$ is not optimal.
			\item \label{it:nonlinear-competitors-g-maybe-min} If instead either $\theta(H)<H$, or $\theta(H)=H$ and
			      \begin{equation}\label{eq:phi-asymptote-yes}
				      \inf_{\alpha\geq\alpha^*} \frac{\phi(\alpha)-\alpha}{\alpha^{1-\gamma/p}} > -\infty,
			      \end{equation}
			      then for every $\lambda\geq0$ sufficiently large we have $F(u_\alpha)\geq F(g)$ for all
			      $\alpha\geq\alpha^*$; hence the input $g$ performs better than these particular competitors.
		\end{enumerate}
	\end{subequations}
\end{proposition}

\begin{proof}
	We have $F(g)=\theta(H)$ and
	\[
		F(u_\alpha)
		= H \frac{\phi(\alpha)}\alpha
		+ \frac{\lambda H^{\gamma(1+1/p)}}{2^\gamma(p+1)^{\gamma/p}\alpha^{\gamma/p}}.
	\]
	Assume $\theta(H)<H$. The difference in energy can be written as
	\[
		F(u_\alpha)-F(g)
		= \alpha^{-\gamma/p} \left[ \alpha^{\gamma/p} \left(H\frac{\phi(\alpha)}\alpha-\theta(H)\right)
			+ \frac{\lambda H^{\gamma(1+1/p)}}{2^\gamma(p+1)^{\gamma/p}} \right].
	\]
	Since $H\frac{\phi(\alpha)}\alpha-\theta(H)$ is a bounded function of $\alpha$ and $\gamma/p>0$,
	we have
	\[
		\lim_{\alpha\to0^+}
		\left[\alpha^{\gamma/p} \left(H\frac{\phi(\alpha)}\alpha-\theta(H)\right)\right] = 0.
	\]
	On the other hand, $\lim_{\alpha\to\infty} \left(H\frac{\phi(\alpha)}\alpha-\theta(H)\right) =
		H-\theta(H) > 0$, therefore
	\[
		\lim_{\alpha\to\infty}
		\left[\alpha^{\gamma/p} \left(H\frac{\phi(\alpha)}\alpha-\theta(H)\right)\right] = \infty.
	\]
	By continuity in $\alpha$, it follows that
	\[
		m \coloneqq
		\inf_{\alpha\geq\alpha^*} \left[\alpha^{\gamma/p} \left(H\frac{\phi(\alpha)}\alpha-\theta(H)\right)\right]
		\geq \inf_{\alpha>0} \left[\alpha^{\gamma/p} \left(H\frac{\phi(\alpha)}\alpha-\theta(H)\right)\right]
		> -\infty.
	\]
	This implies that for every $\lambda$ sufficiently large, namely $\lambda >
		-\frac{2^{\gamma}(p+1)^{\gamma/p}}{H^{\gamma(1+1/p)}}m$, we have $F(u_\alpha)-F(g)>0$
	for every $\alpha$.

	Assume $\theta(H)=H$. In this case the difference in energy can be written as
	\[
		F(u_\alpha)-F(g)
		= \alpha^{-\gamma/p} \left[ H \frac{\phi(\alpha)-\alpha}{\alpha^{1-\gamma/p}}
			+ \frac{\lambda H^{\gamma(1+1/p)}}{2^\gamma(p+1)^{\gamma/p}} \right].
	\]
	In case (i) we have that for every $\lambda\geq0$ the infimum of the square bracket is $-\infty$,
	therefore there exists $\alpha\geq\alpha^*$ for which it is negative, hence $F(u_\alpha)-F(g)<0$.
	In case (ii) we have that for $\lambda$ sufficiently large, namely
	\[
		\lambda \geq -2^\gamma(p+1)^{\gamma/p}H^{1-\gamma(1+1/p)}
		\inf_{\alpha\geq\alpha^*} \frac{\phi(\alpha)-\alpha}{\alpha^{1-\gamma/p}},
	\]
	the infimum of the square bracket is non-negative, hence $F(u_\alpha)-F(g)\geq0$ for all
	$\alpha\geq\alpha^*$.
\end{proof}

To replicate the truncation argument used in \ref{it:2b}, we would need the zero function to be a
global minimizer of the functional
\begin{equation}\label{eq:penalized-functional}
	F(u) = \TVphitheta\bigl(u;(0,L)\bigr) + \theta\bigl(\abs{H-u^-(L)}\bigr) + \lambda\norm{u}_p^\gamma
\end{equation}
among all functions $u\in BV\bigl((0,L)\bigr)$ with $u^+(0)=0$.
In \autoref{prop:not-min-zero-nonlinear} below we show that this is not the case.

By a line of reasoning similar to that in the proof of \autoref{thm:bvp}, we can show that a
minimizer $u$ must be increasing (see the proof of monotonicity at the beginning of the proof) and
with values in $[0,H]$, its Cantor and jump part can be moved forward to $x=L$ lowering both the
total variation and the fidelity term, hence we may assume $u\in W^{1,1}\bigl((0,L)\bigr)$, and
finally $u$ must be convex since the increasing rearrangement of $u'$ does not change the total
variation but lowers the fidelity.

\begin{proposition}\label{prop:not-min-zero-nonlinear}
	Let $L,H>0$, let $p,\gamma\geq1$, and assume $\theta'_-(H)>\phi'_+(0)$.
	Then the function $g=0$ is not a local minimizer in $\BV\bigl((0,L)\bigr)$ of the functional
	\[
		F(u) = \TVphitheta\bigl(u;(0,L)\bigr) + \theta\bigl(\abs{H-u^-(L)}\bigr) + \lambda\norm{u}_p^\gamma
	\]
	with boundary condition $u^+(0)=0$ for any $\lambda>0$.
\end{proposition}

\begin{proof}
	Considering non-negative variations $v\in W^{1,1}\bigl((0,L);[0,\infty)\bigr)$ with $v^+(0)=0$,
	the one-sided directional derivative of the functional $F$ at $u=g$ is (see
	\autoref{sec:first-variation})
	\[
		\begin{split}
			\lim_{\eps\to0^+} \frac{F(g+\eps v)-F(g)}\eps
			&= \int_0^L \phi'_+(0)v'(x)\dx - \theta'_-(H)v^-(L) + \delta_{\gamma,1} \lambda \norm{v}_p \\
			&= \bigl(\phi'_+(0)-\theta'_-(H)\bigr) v^-(L) + \delta_{\gamma,1} \lambda \norm{v}_p.
		\end{split}
	\]
	For any $\lambda$, this can be made strictly negative by taking $v(x)=(x/L)^q$ with $q>0$ large enough.
	This disproves the local minimality of $g\in W^{1,1}\bigl((0,L)\bigr)$, hence in $BV\bigl((0,L)\bigr)$.
\end{proof}

\subsubsection{Numerical examples}\label{sec:numerical-examples}

We conclude this section by investigating numerically the nature of the global minimizer in a simple
situation. Let us take $p=\gamma=1$ and $\phi,\theta$ differentiable.
Recall that the minimizer is a convex function $u\in W^{1,1}\bigl((0,L);[0,H]\bigr)$ with $u^+(0)=0$.
Let $x_0=\max\set{x\in[0,L]}{u\rvert_{(0,x]}=0}$.
From \autoref{sec:first-variation-simple} we know that in the interval $(x_0,L)$ the optimal profile
$u$ is given by \eqref{eq:optimal-profile-explicit}:
\[
	u(x) = \begin{cases}
		0                                                       & x\in(0,x_0], \\
		\frac1\lambda \phi^*\bigl(\phi'(0)+\lambda(x-x_0)\bigr) & x\in(x_0,L).
	\end{cases}
\]

In this sub section, we consider two explicit examples:
\begin{align*}
	\phi_1(t)   & = t-\log(1+t),    &
	\phi_2(t)   & = \sqrt{1+t^2}-1,   \\
	\theta_1(t) & = \log(1+t),      &
	\theta_2(t) & = \sqrt{1+2t}-1.
\end{align*}
For the non-linearity $\phi_1$ we have $\phi_1'(0)=0$ and $\phi_1^*(\tau) = -\tau-\log(1-\tau)$,
therefore in the interval $(x_0,L)$
\[
	u(x) = -(x-x_0) - \frac{\log\bigl(1-\lambda(x-x_0)\bigr)}\lambda.
\]
On the other hand, for $\phi_2$ we have $\phi_2'(0)=0$ and $\phi_2^*(\tau) = 1-\sqrt{1-\tau^2}$,
hence in the interval $(x_0,L)$
\[
	u(x) = \frac{1-\sqrt{1-\lambda^2(x-x_0)^2}}\lambda.
\]
These two families of profiles are depicted in \autoref{fig:optimal-profiles}.
\begin{figure}[ht]
	\includegraphics[width=0.5\textwidth]{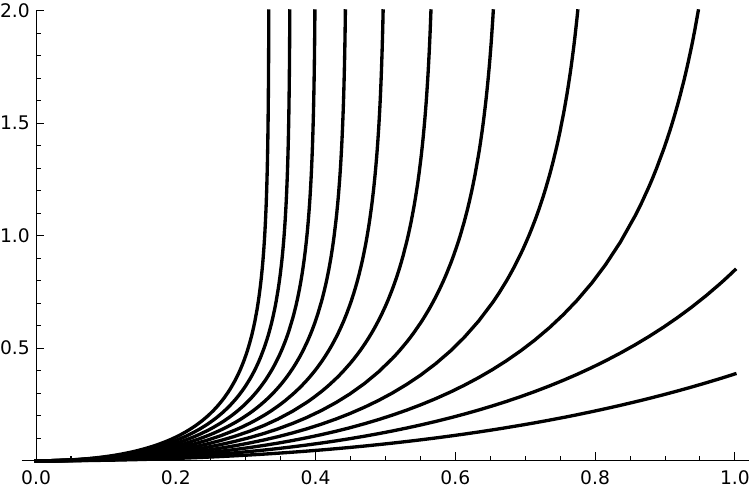}%
	\includegraphics[width=0.5\textwidth]{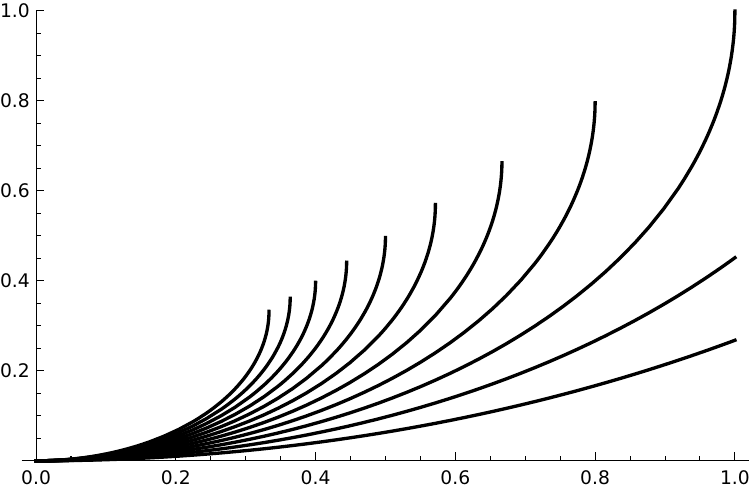}
	\caption{profiles $u$ shifted at $x_0=0$ for the non-linearities $\phi_1$ (left) and $\phi_2$ (right),
	with the parameter $\lambda$ ranging from $1/2$ (lowest curves) to $3$ (highest curves) in steps of $1/4$.
	Notice that in both cases the interval of definition of the profiles is $[0,1/\lambda)$.}
	\label{fig:optimal-profiles}
\end{figure}

In \autoref{fig:minimizers} we illustrate some minimizers $u$ of the functional $F(u)$ for the two
sets of non-linearities and various choices of the parameter $\lambda$. These are computed
numerically through a search of local minimizers and it is possible to verify that they follow the
shape of the analytical profiles introduced above.
\begin{figure}[ht]
	\includegraphics[width=0.5\textwidth]{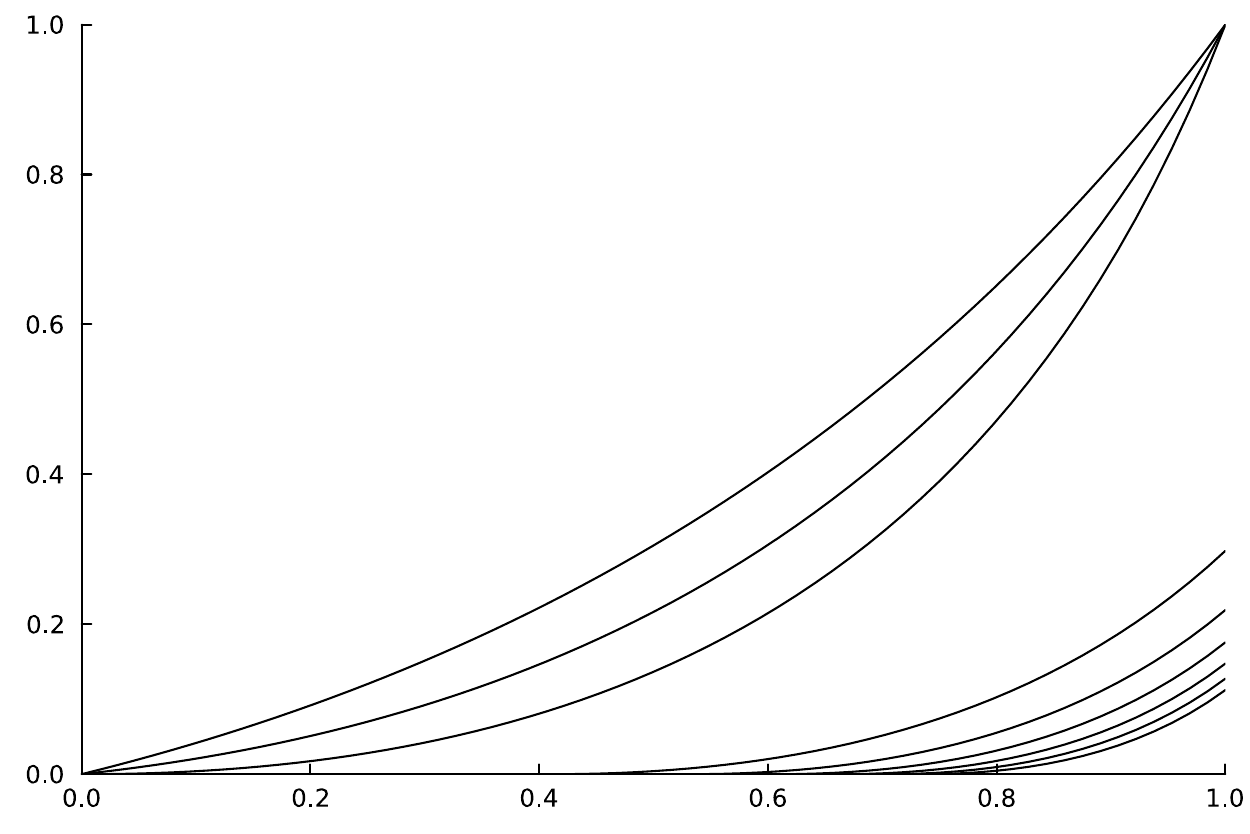}%
	\includegraphics[width=0.5\textwidth]{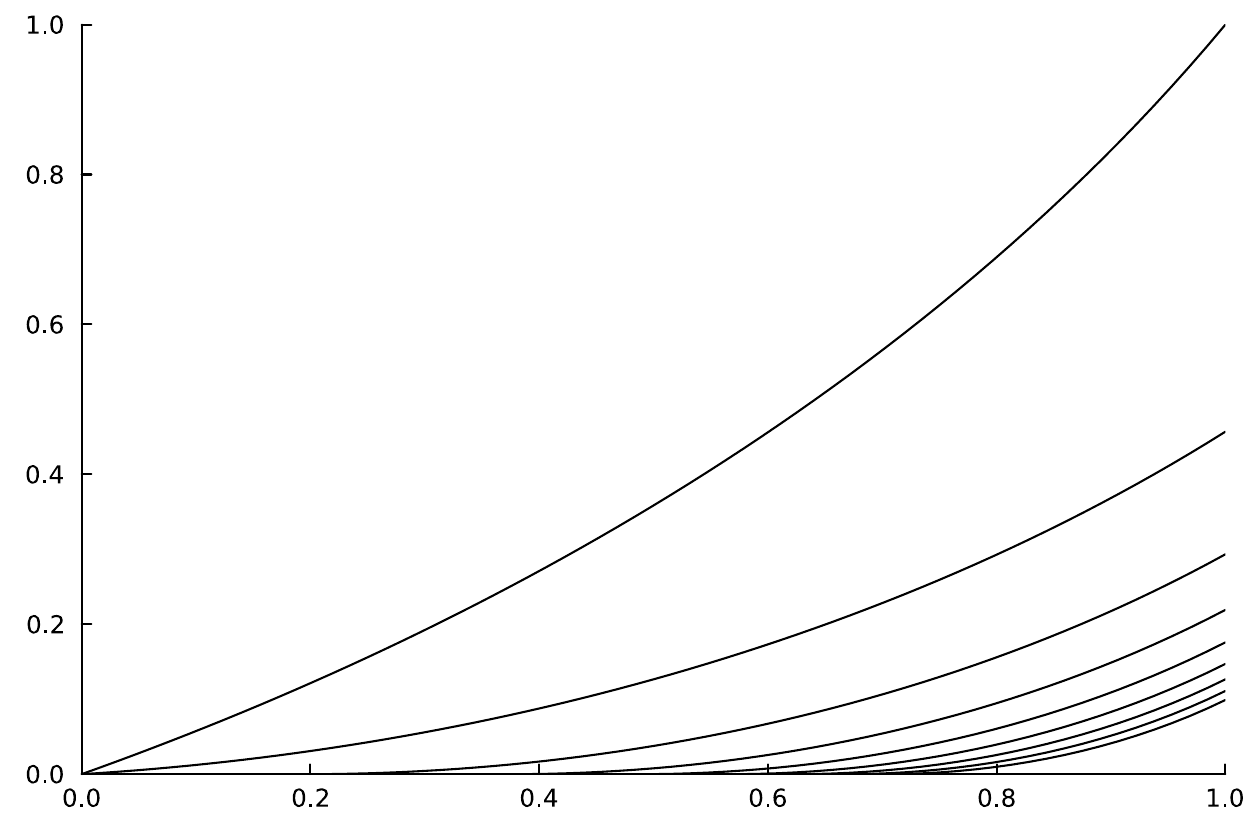}
	\caption{minimizers of $F$ with non-linearities $\phi_1,\theta_1$ (left) and $\phi_2,\theta_2$ (right),
		$L=H=1$, and with the parameter $\lambda$ ranging
		from $0.4$ (highest curves) to $2$ (lowest curves) in steps of $0.2$.}
	\label{fig:minimizers}
\end{figure}

In \autoref{fig:minimizer-slope} we plot the functions $u$, $u'$ and $\phi'(u')$, where $u$ is the
minimizer of the problem with non-linearities $\phi_1,\theta_1$ and parameters $L=H=1$, $\lambda=1.1$.
As predicted by the ODE \eqref{eq:optimal-profile-ode-simplified}, the function $\phi'(u')$ appears to be
affine in the interval where $u>0$.
\begin{figure}[ht]
	\centering
	\includegraphics[width=0.5\textwidth]{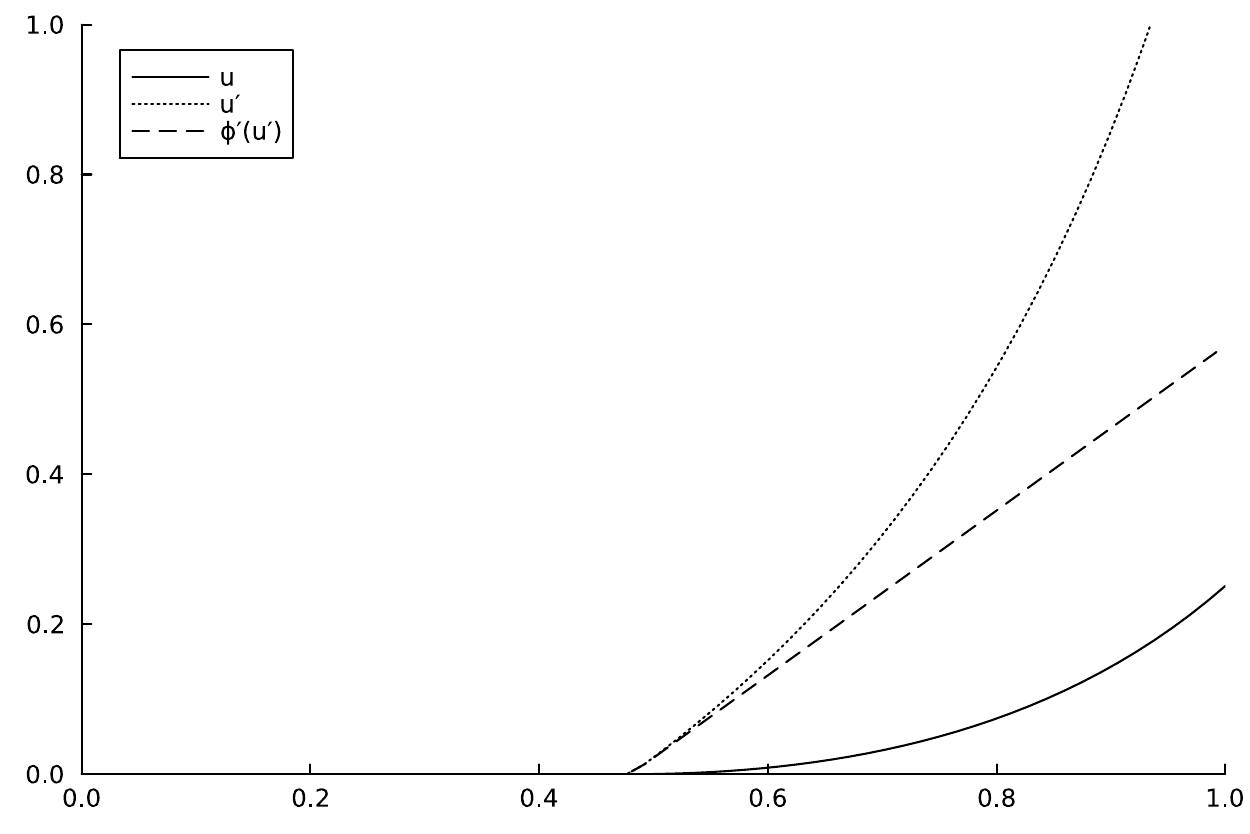}
	\caption{plot of $u$, $u'$ and $\phi'(u')$, where $u$ is the minimizer of the problem with
		non-linearities $\phi_1,\theta_1$ and parameters $L=H=1$, $\lambda=1.1$; notice that in the interval
		where $u>0$ the function $\phi'(u')$ is affine, in accordance with
		\eqref{eq:optimal-profile-ode-simplified}. The value of $x_0$, computed numerically with both
		Julia and Mathematica, is approximately $x_0\approx0.4801305$.}
	\label{fig:minimizer-slope}
\end{figure}

\paragraph{Numerical implementation}
The numerical approximations of the optimal profiles are obtained as follows.
Knowing that the minimizer of the functional $F$ should be a non-negative continuous function%
\footnote{Additionally, $u\in W^{1,1}\bigl((0,L);[0,H]\bigr)$ non-decreasing and convex.},
for a fixed $n\in\setN_+$ we introduce the uniform mesh $x_i=i\Delta x$, where $\Delta x=L/n$,
and approximate the unknown function $u$ with a continuous piecewise affine function
\[
	\hat u(x) = \sum_{i=1}^n \left(\frac{x_i-x}{x_i-x_{i-1}} u_{i-1}
	+ \frac{x-x_{i-1}}{x_i-x_{i-1}} u_i\right) \bm1_{(x_{i-1},x_i]}(x)
\]
parametrized by the vector $(u_1,\dots,u_n)\in[0,\infty)^n$, with the convention $u_0=0$.
The values $u_i=\hat u(x_i)$ correspond to a sampling of $u$ over the uniform mesh.
This leads to introducing the function $F_n:[0,\infty)^n\to[0,\infty)$ given by
\[
	\begin{split}
		F_n(u_1,\dots,u_n)
		&= F(\hat u)
		= \sum_{i=1}^n \phi\left(\abs*{\frac{u_i-u_{i-1}}{\Delta x}}\right) \Delta x
		+ \theta(\abs{H-u_n})
		+ \lambda \left( \sum_{i=1}^{n-1} u_i \Delta x + \frac{u_n}2 \Delta x \right) \\
		&= \left( \sum_{i=1}^n \phi\left(\abs*{\frac{u_i-u_{i-1}}{\Delta x}}\right)
		+ \lambda \left( \sum_{i=1}^{n-1} u_i + \frac{u_n}2 \right) \right) \Delta x
		+ \theta(\abs{H-u_n}),
	\end{split}
\]
which we minimize among the vectors $(u_1,\dots,u_n)\in[0,H]^n$.
The function $F_n$ amounts to approximating the $L^1$ norm in the fidelity with the trapezoidal rule,
the derivative $u'$ with finite differences, and the integral of $\phi(\abs{u'})$ with the rectangle rule.

Alternatively, the problem can be reformulated in terms of the discrete slopes
$(a_1,\dots,a_n)\in\setR^n$, which are related to the values $(u_1,\dots,u_n)\in\setR^n$ by the
transformation
\[
	a_i = \frac{u_i-u_{i-1}}{\Delta x}, \qquad
	u_i = \sum_{j=1}^i a_i \Delta x.
\]
Notice that the mapping $\setR^n\to\setR^n$ between $(u_i)_{i=1}^n$ and $(a_i)_{i=1}^n$ is a bijection.
Recalling the monotonicity of the optimal profile $u$, the slopes $(a_i)_{i=1}^n$ can be assumed to
be non-negative.

The function $G_n:[0,\infty)^n\to[0,\infty)$ which corresponds to $F_n$ in this set of coordinates,
i.e.\ $G_n(a_1,\dots,a_n)=F_n(u_1,\dots,u_n)$, is given by
\[
	\begin{split}
		& G_n(a_1,\dots,a_n) \\
		&= \Delta x \left( \sum_{i=1}^n \phi(a_i)
		+ \lambda \Delta x \left(\sum_{i=1}^{n-1} \sum_{j=1}^i a_j
			+ \frac12 \sum_{j=1}^n a_j \right) \right)
		+ \theta\oleft(\abs*{H - \Delta x \sum_{j=1}^n a_j}\right) \\
		&= \Delta x \left( \sum_{i=1}^n \phi(a_i)
		+ \lambda \Delta x \left(\sum_{i=1}^n a_i (n-i+1/2) \right) \right)
		+ \theta\oleft(\abs*{H - \Delta x \sum_{j=1}^n a_j}\right).
	\end{split}
\]

Not only the formulation in terms of the slopes avoids the bad conditioning arising from the finite
difference operator, but it also allows to easily restrict the search to monotone functions instead
of merely non-negative ones. These benefits manifest empirically in a faster convergence of the
optimization algorithm.

The actual numerical minimization has been implemented in the Julia programming language \cite{Julia}
using the box-constrained limited-memory Broyden--Fletcher--Goldfarb--Shanno (\mbox{L-BFGS-B}) algorithm
provided by the package \verb!Optim.jl! \cite{Optim.jl}. In addition, a multistart strategy from
some hand-picked and some random initial points has been employed in order to reduce the risk of
converging to a local minimum.

The approximate profiles obtained via minimization of $G_n$ have been validated by comparison with
the semi-analytical approach described in \autoref{sec:first-variation-simple}, which consists
in minimizing the scalar function defined in \eqref{eq:function-x_0} subjected to the constraint
\eqref{eq:constraint-x_0}.

As an example, in \autoref{fig:one_dim_min_phi_1} we plot the function $f$ defined in \eqref{eq:function-x_0}
associated with non-linearities $\phi_1,\theta_1$ and parameters $L=H=1$, $\lambda=1.1$; the minimum
is found at $x_0\approx0.4801305$, as already mentioned in \autoref{fig:minimizer-slope}.
In \autoref{fig:one_dim_min_phi_2} we illustrate the similar situation for $\phi_2,\theta_2$, $L=1$,
$H=0.55$ and $\lambda=1.5$.
\begin{figure}[ht]
	\includegraphics[width=0.49\textwidth]{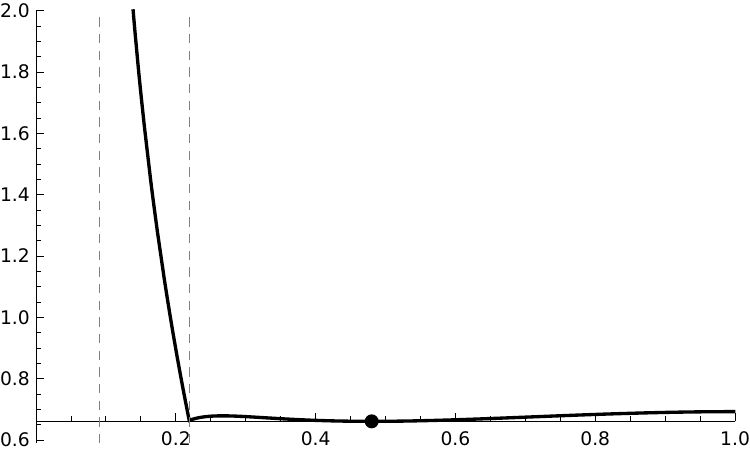}\hfill
	\includegraphics[width=0.49\textwidth]{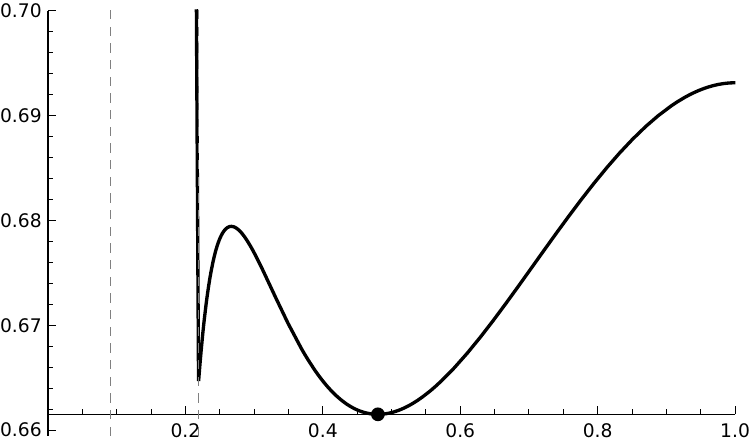}
	\caption{plot of the function $f$ defined in \eqref{eq:function-x_0} associated with
		non-linearities $\phi_1,\theta_1$ and parameters $L=H=1$, $\lambda=1.1$. The same plot is
		presented at two different vertical zoom levels. The vertical dashed lines, from left to right,
		are positioned at $L-\frac{1-\phi'(0)}\lambda\approx0.0909091$ and
		$L-\frac{(\phi^*)^{-1}(\lambda H)-\phi'(0)}\lambda\approx0.21909$.
		The minimum, highlighted with a black dot, is found at $x_0\approx0.4801305$, as already
		mentioned in \autoref{fig:minimizer-slope}.}
	\label{fig:one_dim_min_phi_1}
\end{figure}
\begin{figure}[ht]
	\includegraphics[width=0.49\textwidth]{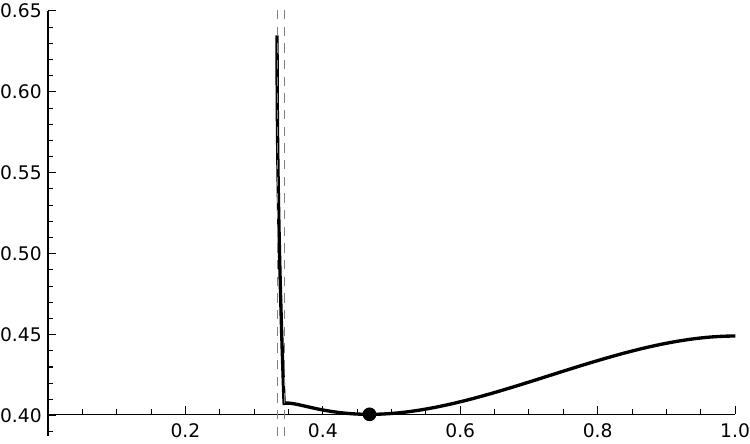}\hfill
	\includegraphics[width=0.49\textwidth]{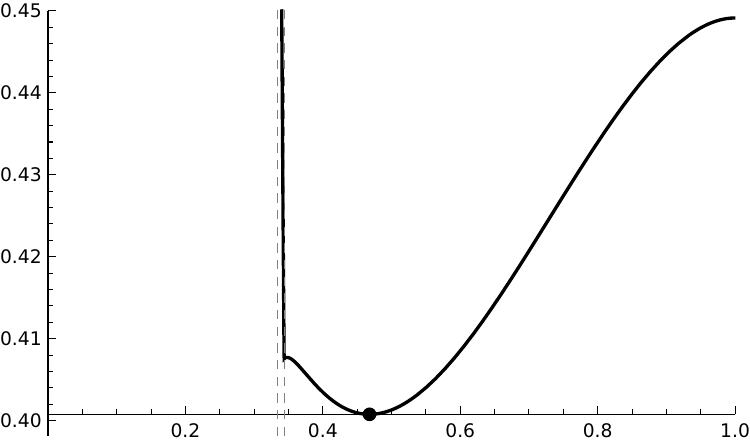}
	\caption{plot of the function $f$ defined in \eqref{eq:function-x_0} associated with
		non-linearities $\phi_2,\theta_2$ and parameters $L=1$, $H=0.55$, $\lambda=1.5$. The same plot is
		presented at two different vertical zoom levels. The vertical dashed lines, from left to right,
		are positioned at $L-\frac{1-\phi'(0)}\lambda\approx0.333333$ and
		$L-\frac{(\phi^*)^{-1}(\lambda H)-\phi'(0)}\lambda\approx0.343621$.
		The minimum, highlighted with a black dot, is found at $x_0\approx0.467959$.}
	\label{fig:one_dim_min_phi_2}
\end{figure}

\appendix

\section{Miscellanea}\label{sec:miscellanea}

\subsection{Alternative proofs of \autoref{prop:E-monotone}\ref{it:E-monotone}}\label{sec:E-monotone-alt-proofs}

\begin{proof}[Alternative proof of \autoref{prop:E-monotone}\ref{it:E-monotone}, from definition \eqref{eq:E-def}]
	\emph{Monotonicity in $L$.}
	Let $L_1\leq L_2$ and $u\in\BV\bigl((0,L_1)\bigr)$ be optimal for $E(L_1,H)$ and increasing.
	Define $v\in\BV\bigl((0,L_2)\bigr)$ by $v(x)=u(L_1/L_2x)$.
	The limits at the endpoints are preserved, so $v$ is admissible for $E(L_2,H)$.
	Setting $T(x)=L_2/L_1x$, we have $\D v=T_\#\D u$ and also $\D^{a,c,j}v=T_\#\D^{a,c,j}u$.
	Owing to the bijectivity of $T$, the contributions of $\D^c$ and $\D^j$ to the functional
	$\TVphitheta$ are preserved, because
	\[
		\begin{split}
			\abs{\D^cv}\bigl((0,L_2)\bigr)
			&= \abs{T_\#\D^cu}\bigl((0,L_2)\bigr)
			= T_\#\abs{\D^cu}\bigl((0,L_2)\bigr)
			= \abs{\D^cu}\bigl(T^{-1}\bigl((0,L_2)\bigr)\bigr)
			= \abs{\D^cu}\bigl((0,L_1))
		\end{split}
	\]
	and $\abs{\D^jv}$ and $\abs{\D^ju}$ have the same atom values.
	On the other hand, since $L_1/L_2\leq1$, we have $\phi\bigl((L_1/L_2)t\bigr)\leq (L_1/L_2)\phi(t)$
	for every $t\geq0$, hence
	\[
		\begin{split}
			\int_0^{L_2} \phi\bigl(v'(x)\bigr) \dx
			&= \int_0^{L_2} \phi\oleft(\frac{L_1}{L_2}u'\oleft(\frac{L_1}{L_2}x\right)\right) \dx \\
			&\leq \int_0^{L_2} \frac{L_1}{L_2} \phi\oleft(u'\oleft(\frac{L_1}{L_2}x\right)\right) \dx
			= \int_0^{L_1} \phi\bigl(u'(x)\bigr) \dx
		\end{split}
	\]
	Therefore
	$E(L_2,H)\leq\TVphitheta\bigl(v;(0,L_2)\bigr)=\TVphitheta\bigl(u;(0,L_1)\bigr)\leq E(L_1,H)$.

	\emph{Monotonicity in $H$.}
	Let $H_1\leq H_2$ and $u\in\BV\bigl((0,L)\bigr)$ optimal for $E(L,H_2)$.
	Define $v\in\BV\bigl((0,L)\bigr)$ by $v(x)=(H_1/H_2)u(x)$.
	We have $v^+(0)=0$ and $v^-(L)=H_1$, so $v$ is admissible for $E(L,H_1)$.
	Since $\phi$ and $\theta$ are monotone and $\abs{\D v}=(H_1/H_2)\abs{\D u}\leq \abs{\D u}$, we get
	$E(L,H_1)\leq\TVphitheta\bigl(v;(0,L)\bigr)\leq\TVphitheta\bigl(u;(0,L)\bigr)=E(L,H_2)$.
\end{proof}

\begin{proof}[Alternative proof of \autoref{prop:E-monotone}\ref{it:E-monotone}, from \eqref{eq:E-min-ab}]
	\emph{Monotonicity in $L$.}
	Recall the formula
	\[
		E(L,H)
		= \min\set*{L\phi\oleft(\frac{H-b}L\right)+\theta(b)}{0\leq b\leq H}
	\]
	in \eqref{eq:E-min-ab}.
	Observe that for every $t\geq0$ the function $L\mapsto L\phi(t/L)$ is decreasing.
	Indeed, if $L_1\leq L_2$, then $\phi(t/L_2)=\phi\bigl((L_1/L_2)(t/L_1)\bigr)\leq (L_1/L_2)\phi(t/L_1)$.
	This means that the objective function to minimize is decreasing in $L$, and therefore so
	is $E(L,H)$.

	\emph{Monotonicity in $H$.}
	Let $L_1\leq L_2$ and $a,b\geq0$ with $La+b=H_2$ be optimal in \eqref{eq:E-min-ab} for $E(L,H_2)$.
	Setting $\tilde a=(H_1/H_2)a\leq a$ and $\tilde b=(H_1/H_2)b\leq b$, we have $L\tilde a+\tilde b=H_1$.
	Since $\phi$ and $\theta$ are increasing, we deduce
	$E(L,H_1)\leq L\phi(\tilde a)+\theta(\tilde b)\leq L\phi(a)+\theta(b)=E(L,H_2)$.
\end{proof}

\section{First variation of the functional in \autoref{sec:nonlinear-min}}\label{sec:first-variation}

Let $\sigma:[0,\infty\bigr)\to\setR$ be a function which admits left and right derivatives at every point.
We introduce the following notation for the one-sided directional derivative of the composition
$\sigma\circ\abs\plchldr$:
\[
	\mathcal{D}_\sigma(a)[b]
	\coloneqq (\sigma\circ\abs\plchldr)'_{\sign b}(a)b
	= \lim_{\eps\to0^+} \frac{\sigma(\abs{a+\eps b})-\sigma(\abs{a})}\eps
	= \begin{cases}
		\sigma'_+(0)\abs{b}                   & a=0,    \\
		\sigma'_{\sign(ab)}(\abs{a})\sign(a)b & a\neq0,
	\end{cases}
\]
where $\sigma'_-,\sigma'_+$ denote the left and right derivatives respectively.
Notice that, for every fixed $a\in\setR$, the function $\mathcal{D}_\sigma(a)[b]$ is positively
1-homogeneous in $b$.
Moreover, if $\sigma$ is differentiable everywhere, then
$\mathcal{D}_\sigma(a)[b]=\sigma'(\abs{a})\sign(a)b$ is linear in $b$ for $a\neq0$, and it is linear for
$a=0$ if and only if $\sigma'_+(0)=0$.

Let us compute the one-sided G\^ateaux derivative of the functional $F$.
Starting with the total variation, for $u,v\in W^{1,1}\bigl((0,L)\bigr)$ we have
\[
	\begin{split}
		\lim_{\eps\to0^+} \frac{\TVphitheta(u+\eps v)-\TVphitheta(u)}\eps
		&= \int_0^L \lim_{\eps\to0^+} \frac{\phi\bigl(\abs{u'(x)+\eps v'(x)}\bigr)-\phi\bigl(\abs{u'(x)}\bigr)}\eps \dx \\
		&= \int_0^L \mathcal{D}_\phi\bigl(u'(x)\bigr)[v'(x)] \dx;
	\end{split}
\]
on the other hand, the derivative of the penalization at the right endpoint is
\[
	\lim_{\eps\to0^+} \frac{\theta\bigl(\abs{u^-(L)+\eps v^-(L)-H}\bigr)-\theta\bigl(\abs{u^-(L)-H}\bigr)}\eps
	= \mathcal{D}_\theta\bigl(u^-(L)-H\bigr)[v^-(L)].
\]
Let us now compute the one-sided G\^ateaux derivative of $\norm{u}_p^\gamma$.
When $u=0$ is the identically zero function, we have
\[
	\lim_{\eps\to0^+} \frac{\norm{u+\eps v}_p^\gamma-\norm{u}_p^\gamma}\eps
	= \lim_{\eps\to0^+} \frac{\eps^\gamma}\eps \norm{v}_p^\gamma
	= \begin{cases}
		0          & \gamma>1,  \\
		\norm{v}_p & \gamma=1,  \\
		\infty     & \gamma <1.
	\end{cases}
\]
Assume now $u\not\equiv0$, hence $\norm{u}_p>0$.
For $p=1$ we have
\[
	\begin{split}
		\lim_{\eps\to0^+} \frac{\norm{u+\eps v}_p^p-\norm{u}_p^p}\eps
		&= \lim_{\eps\to0^+} \int_0^L \frac{\abs{u(x)+\eps v(x)}-\abs{u(x)}}\eps \dx \\
		&= \int_{\{u\neq0\}} \sign u(x) v(x) \dx + \int_{\{u=0\}} \abs{v(x)} \dx,
	\end{split}
\]
whereas for $p>1$ we have
\[
	\begin{split}
		\lim_{\eps\to0^+} \frac{\norm{u+\eps v}_p^p-\norm{u}_p^p}\eps
		&= \lim_{\eps\to0^+} \int_0^L \frac{\abs{u(x)+\eps v(x)}^p-\abs{u(x)}^p}\eps \dx \\
		&= \int_0^L p \abs{u(x)}^{p-1} \sign u(x) v(x) \dx \\
		&= \int_{\{u\neq0\}} p \abs{u(x)}^{p-1} \sign u(x) v(x) \dx.
	\end{split}
\]
These last two formulas can be combined succinctly as
\begin{equation}\label{eq:norm-first-variation}
	\lim_{\eps\to0^+} \frac{\norm{u+\eps v}_p^p-\norm{u}_p^p}\eps
	= \int_{\{u\neq0\}} \abs{u(x)}^{p-1} \sign u(x) v(x) \dx
	+ \delta_{p,1} \int_{\{u=0\}} \abs{v(x)} \dx,
\end{equation}
valid for all $p\geq1$, where $\delta_{p,1}$ denotes the Kronecker delta.
Writing $\norm{u}_p^\gamma=(\norm{u}_p^p)^{\gamma/p}$, by the chain rule we finally get
\[
	\lim_{\eps\to0^+} \frac{\norm{u+\eps v}_p^\gamma-\norm{u}_p^\gamma}\eps
	= \frac\gamma p \norm{u}_p^{\gamma-p}
	\lim_{\eps\to0^+} \frac{\norm{u+\eps v}_p^p-\norm{u}_p^p}\eps,
\]
where this latter right derivative is given by \eqref{eq:norm-first-variation}.

In conclusion, the first variation of the functional $F$ at the identically zero function $u=0$ is
\[
	\lim_{\eps\to0^+} \frac{F(0+\eps v)-F(0)}\eps
	= \int_0^L \phi'_+(0) \abs{v'(x)} \dx
	+ \theta'_{\sign v^-(L)}(H) v^-(L)
	+ \begin{cases}
		0                  & \gamma>1,  \\
		\lambda \norm{v}_p & \gamma=1,  \\
		\infty             & \gamma <1;
	\end{cases}
\]
whereas for $u\not\equiv0$ instead we have
\[
	\begin{split}
		& \lim_{\eps\to0^+} \frac{F(u+\eps v)-F(u)}\eps \\
		&= \int_0^L \mathcal{D}_\phi\bigl(u'(x)\bigr)[v'(x)] \dx
		+ \mathcal{D}_\theta\bigl(u^-(L)-H\bigr)[v^-(L)]
		\spliteq + \lambda \gamma \norm{u}_p^{\gamma-p} \left(
		\int_{\{u\neq0\}} \abs{u(x)}^{p-1} \sign u(x) v(x) \dx
		+ \delta_{p,1} \int_{\{u=0\}} \abs{v(x)} \dx \right).
	\end{split}
\]

\subsection{Differentiable setting}

If $\phi$ and $\theta$ are differentiable everywhere, we can expand out the definition of
$\mathcal{D}_\phi,\mathcal{D}_\theta$ and obtain the simplified expression
\[
	\begin{split}
		& \lim_{\eps\to0^+} \frac{F(u+\eps v)-F(u)}\eps \\
		&= \int_{\{u'\neq0\}} \phi'\bigl(u'(x)\bigr)\sign\bigl(u'(x)\bigr) v'(x) \dx
		+ \int_{\{u'=0\}} \phi'(0) \abs{v'(x)} \dx
		\spliteq + \begin{cases}
			\abs{v^-(L)}                                             & u^-(L)=H     \\
			\theta'\bigl(\abs{u^-(L)-H}\bigr) \sign(u^-(L)-H) v^-(L) & u^-(L)\neq H
		\end{cases}
		\spliteq + \lambda \gamma \norm{u}_p^{\gamma-p} \left(
		\int_{\{u\neq0\}} \abs{u(x)}^{p-1} \sign u(x) v(x) \dx
		+ \delta_{p,1} \int_{\{u=0\}} \abs{v(x)} \dx \right) .
	\end{split}
\]

\subsubsection{Special case \texorpdfstring{$p=\gamma=1$}{p=γ=1}}\label{sec:first-variation-simple}

Since $u(x)=\int_0^x u'(y)\dy$, for non-decreasing functions the functional can be rewritten as
\[
	\begin{split}
		F(u)
		&= \int_0^L \phi\bigl(u'(x)\bigr)\dx + \theta\bigl(H-u^-(L)\bigr)
		+ \lambda \int_0^L\int_0^x u'(y)\dy\dx \\
		&= \int_0^L \bigl[\phi\bigl(u'(x)\bigr)+\lambda(L-x)u'(x)\bigr]\dx
		+ \theta\bigl(H-u^-(L)\bigr).
	\end{split}
\]
The first integral is convex in $u'$, whereas the second addend is concave.

\paragraph{Optimality conditions}
Let us derive simplified optimality conditions obtained from vertical variations.
\begin{itemize}
	\item If $u^-(L)=H$, the first variation with respect to $v\in C^1((0,L))$ with $v\rvert_{(0,x_0]}=0$ is
	      \[
		      \lim_{\eps\to0^+} \frac{F(u+\eps v)-F(u)}\eps
		      = \int_{x_0}^L \bigl[\phi'\bigl(u'(x)\bigr) v'(x) + \lambda(L-x) v'(x)\bigr] \dx + \abs{v^-(L)}.
	      \]
	      When $v^-(L)\geq0$ the one-sided variation is
	      \[
		      \int_{x_0}^L \bigl[\phi'\bigl(u'(x)\bigr) + \lambda(L-x) + 1\bigr] v'(x) \dx.
	      \]
	      If $u$ is a local minimizer, this must be positive for every $v$ such that
	      $\int_{x_0}^L v'(x) \dx \geq0$, hence $\phi'\bigl(u'(x)\bigr) + \lambda(L-x) + 1$ must be a
	      non-negative constant.

	      When $v^-(L)\leq0$ the one-sided variation is
	      \[
		      \int_{x_0}^L \bigl[\phi'\bigl(u'(x)\bigr) + \lambda(L-x) - 1\bigr] v'(x) \dx.
	      \]
	      If $u$ is a local minimizer, this must be positive for every $v$ such that
	      $\int_{x_0}^L v'(x) \dx \leq0$, hence $\phi'\bigl(u'(x)\bigr) + \lambda(L-x) - 1$ must be a
	      non-positive constant.

	      Putting these two conditions together we get that
	      $\phi'\bigl(u'(x)\bigr) + \lambda(L-x)$ must be a constant in $[0,1]$.

	\item If instead $u^-(L)<H$, then for every variation $v\in C^1\bigl((0,L)\bigr)$ with $v\rvert_{(0,x_0]}=0$
	      we have
	      \[
		      \begin{split}
			      \lim_{\eps\to0} \frac{F(u+\eps v)-F(u)}\eps
			      &= \int_{x_0}^L \bigl[\phi'\bigl(u'(x)\bigr) v'(x) + \lambda(L-x)v'(x)\bigr] \dx
			      - \theta'\bigl(H-u^-(L)\bigr)v^-(L) \\
			      &= \int_{x_0}^L \bigl[\phi'\bigl(u'(x)\bigr) + \lambda(L-x) - \theta'\bigl(H-u^-(L)\bigr)\bigr] v'(x) \dx.
		      \end{split}
	      \]
	      If $u$ is a local minimizer, then this variation must be zero, hence by the arbitrariness of $v$ we deduce
	      \begin{equation}\label{eq:nonlinear-minimizer-ode}
		      \phi'\bigl(u'(x)\bigr) + \lambda(L-x) = \theta'\bigl(H-u^-(L)\bigr)
		      \qquad \text{for a.e.\ $x\in(x_0,L)$}.
	      \end{equation}
	      Under continuity assumptions, evaluating at $x=L$ we get
	      \[
		      \phi'\bigl(u'(L)\bigr) = \theta'\bigl(H-u^-(L)\bigr).
	      \]
\end{itemize}

We may also consider horizontal variations arising from translations of $u$.
Given $\alpha\in\setR$, define the extended function
\[
	u_\alpha(x)
	= \begin{cases}
		0                   & x\leq0,   \\
		u(x)                & x\in(0,L) \\
		u^-(L)+\alpha(x-L), & x\geq L,
	\end{cases}
\]
and observe that $F(u)=F(u_\alpha)$.
\begin{itemize}
	\item \textit{Right translation.} Since $u'(L)\geq0$, we have
	      \[
		      \lim_{\tau\to0^+} \frac{F(u_\alpha(\plchldr-\tau))-F(u_\alpha)}\tau
		      = -\phi\bigl(u'(L)\bigr) + \theta'\bigl(H-u^-(L)\bigr)u'(L) - \lambda u^-(L) \geq 0.
	      \]
	\item \textit{Left translation.} The left translation is admissible if and only if $u$ is
	      identically zero in a right neighborhood of $0$.
	      If $u^-(L)<H$, we have
	      \[
		      \lim_{\tau\to0^+} \frac{F(u_\alpha(\plchldr+\tau))-F(u_\alpha)}\tau
		      = \phi(\alpha) - \theta'\bigl(H-u^-(L)\bigr)\alpha + \lambda u^-(L) \geq 0,
	      \]
	      from which we deduce the necessary condition
	      \[
		      \lambda u^-(L)
		      \geq \sup_{\alpha\in\setR} \bigl(\theta'\bigl(H-u^-(L)\bigr)\alpha-\phi(\alpha)\bigr)
		      = \phi^*\bigl(\theta'\bigl(H-u^-(L)\bigr)\bigr)
		      = \phi^*\bigl(\phi'\bigl(u'(L)\bigr)\bigr).
	      \]
	      If instead $u^-(L)=H$, then
	      \[
		      \lim_{\tau\to0^+} \frac{F(u_\alpha(\plchldr+\tau))-F(u_\alpha)}\tau
		      = \phi(\alpha) + \theta'\bigl(H-u^-(L)\bigr)\abs{\alpha} + \lambda u^-(L) \geq 0,
	      \]
	      which is a condition always satisfied since the three summands are individually non-negative.
\end{itemize}

Recall that the minimizer is a convex function $u\in W^{1,1}\bigl((0,L);[0,H]\bigr)$ with $u^+(0)=0$.
Let $x_0=\max\set{x\in[0,L]}{u\rvert_{(0,x]}=0}$.
From \eqref{eq:nonlinear-minimizer-ode} we deduce that the optimal profile $u$ must solve the ODE
\begin{equation}\label{eq:optimal-profile-ode}
	\phi'\bigl(u'(x)\bigr) + \lambda(L-x) = \theta'\bigl(H-u^-(L)\bigr) \eqqcolon C,
\end{equation}
where $C$ is a constant.
From the inequalities $\phi'(0)\leq\phi'\leq\phi'(\infty)=1$ and
$\theta'(\infty)\leq\theta'\leq\theta'(0)=1$,
evaluating \eqref{eq:optimal-profile-ode} at $x=L$ we deduce $\max\{\phi'(0),\theta'(\infty)\}\leq C\leq1$;
evaluating \eqref{eq:optimal-profile-ode} at $x=x_0$ we deduce $\lambda(L-x_0)\leq1-\phi'(0)$,
that is
\[
	x_0\geq L-\frac{1-\phi'(0)}\lambda.
\]
If $\lambda L>1-\phi'(0)$, this ensures $x_0>0$.%
\footnote{Incidentally, observe that in then linear case $\phi(t)=t$ we have
	$\lambda(L-x_0)\leq1-1=0$, therefore $x_0=L$ and the optimal profile cannot detach from $0$.
	This is consistent with the results in \autoref{sec:linear-min}.}

We now wish to find a condition relating $x_0$ and $C$. To do so, we perform yet another variation in
the following manner. Assuming $x_0\in(0,L)$, fix an arbitrary point $\bar x\in(x_0,L)$ and,
extending $u=0$ before $0$ as usual, define the function
\[
	u_\eps(x) = \begin{cases}
		u\bigl((1+\eps)(x-\bar x)+\bar x\bigr) & x\in(0,\bar x), \\
		u(x)                                   & x\in[\bar x,L).
	\end{cases}
\]
A straightforward computation shows that
\[
	\begin{split}
		0 &= \frac{\d}{\d\eps}\biggr\rvert_{\eps=0} F(u_\eps)
		= \int_{x_0}^{\bar x} \left\{ \phi'\bigl(u'(x)\bigr) \bigl[u'(x)+u''(x)(x-\bar x)\bigr]
		+ \lambda u(x) \right\} \dx,
	\end{split}
\]
from which, dividing by $\bar x-x_0$ and sending $\bar x\to x_0^+$, we infer
\[
	0 = \phi'\bigl(u'(x_0)\bigr) u'(x_0) + \lambda u(x_0) = \phi'\bigl(u'(x_0)\bigr) u'(x_0).
\]
If $\phi'>0$, then necessarily $u'(x_0)=0$, hence $\phi'\bigl(u'(x_0)\bigr)=\phi'(0)$.
Otherwise, let $[0,\delta]=\set{t\geq0}{\phi'(t)=0}$.
From the above condition we deduce either $u'(x_0)=0$ or $u'(x_0)\in[0,\delta]$, therefore again
$\phi'\bigl(u'(x_0)\bigr)=\phi'(0)$.
Evaluating \eqref{eq:optimal-profile-ode} at $x=x_0$ yields $C=\phi'(0)+\lambda(L-x_0)$,
and consequently \eqref{eq:optimal-profile-ode} becomes
\begin{equation}\label{eq:optimal-profile-ode-simplified}
	\phi'\bigl(u'(x)\bigr) = \phi'(0) + \lambda(x-x_0).
\end{equation}

\paragraph{Explicit formula for the optimal profile}
Recall the Fenchel--Young inequality $\phi(t)+\phi^*(\tau)\geq \tau t$ with its equality conditions
\begin{equation}\label{eq:Fenchel-Young-equality}
	\phi(t) + \phi^*(\tau) = t\tau
	\qquad\Longleftrightarrow\qquad
	\tau \in \partial\phi(t)
	\qquad\Longleftrightarrow\qquad
	t = \partial\phi^*(\tau)
\end{equation}
expressed in terms of the sub-differential. Even when $\phi$ is not strictly convex, the ODE
\eqref{eq:optimal-profile-ode-simplified} implies $u'(x)\in\partial\phi^*\bigl(\phi'(0)+\lambda(x-x_0)\bigr)$.
Since $\phi^*$ is locally Lipschitz in its domain, by Rademacher's theorem it is differentiable
almost everywhere,%
\footnote{Actually, being one-dimensional, it is differentiable outside of a countable set.}
therefore
\begin{equation}\label{eq:optimal-profile-ode-ae}
	u'(x)=(\phi^*)'\bigl(\phi'(0)+\lambda(x-x_0)\bigr) \qquad
	\text{for almost every $x\in(x_0,L)$}.
\end{equation}
This is sufficient to infer that
\begin{equation}\label{eq:optimal-profile-explicit}
	\begin{split}
		u(x)
		&= \int_{x_0}^x (\phi^*)'\bigl(\phi'(0)+\lambda(y-x_0)\bigr) \dy
		= \frac1\lambda \int_{\phi'(0)}^{\phi'(0)+\lambda(x-x_0)} (\phi^*)'(z) \dz \\
		&= \frac1\lambda \bigl[\phi^*\bigl(\phi'(0)+\lambda(x-x_0)\bigr)-\phi^*\bigl(\phi'(0)\bigr)\bigr]
		= \frac1\lambda \phi^*\bigl(\phi'(0)+\lambda(x-x_0)\bigr).
	\end{split}
\end{equation}

\paragraph{Formula for the resulting functional}
Setting $\tau=\phi'(t)$ in \eqref{eq:Fenchel-Young-equality} shows that
\[
	\phi(t) = t\phi'(t) - \phi^*\bigl(\phi'(t)\bigr).
\]
Evaluating it at $t=u'(x)$ and using \eqref{eq:optimal-profile-ode-simplified} and
\eqref{eq:optimal-profile-explicit} leads to
\[
	\begin{split}
		\phi\bigl(u'(x)\bigr) + \lambda u(x)
		&= u'(x)\phi'\bigl(u'(x)\bigr) - \phi^*\bigl(\phi'\bigl(u'(x)\bigr)\bigr)
		+ \phi^*\bigl(\phi'(0)+\lambda(x-x_0)\bigr) \\
		&= u'(x) \bigl(\phi'(0)+\lambda(x-x_0)\bigr).
	\end{split}
\]
Using \eqref{eq:optimal-profile-ode-ae}, we get
\[
	\begin{split}
		\int_{x_0}^L \left\{ \phi\bigl(u'(x)\bigr) + \lambda u(x) \right\} \dx
		&= \int_{x_0}^L \bigl(\phi'(0)+\lambda(x-x_0)\bigr)(\phi^*)'\bigl(\phi'(0)+\lambda(x-x_0)\bigr) \dx \\
		&= \frac1\lambda \int_{\phi'(0)}^{\phi'(0)+\lambda(L-x_0)} y (\phi^*)'(y) \dy,
	\end{split}
\]
hence the functional becomes
\begin{equation}\label{eq:optimal-profile-functional}
	F(u)
	= \frac1\lambda \int_{\phi'(0)}^{\phi'(0)+\lambda(L-x_0)} y (\phi^*)'(y) \dy
	+ \theta\bigl(\abs{H-u(L)}\bigr),
\end{equation}
which has to be minimized over $x_0\in\left[L-\frac{1-\phi'(0)}\lambda,L\right]$.
In addition, we know that the minimizer satisfies $u(L)\leq H$;
therefore, let us find a more explicit condition on $x_0$ which ensures this property and allows to
remove the absolute value in the argument of $\theta$.

We claim that $\phi^*:[\phi'(0),1]\to[0,\phi^*(1)]$ is strictly increasing, continuous and bijective.
Suppose that $\phi^*(\tau)=0$ for some $\tau\in[\phi'(0),1]$; then
\[
	\phi(t)-\phi(0)
	= \phi(t)
	= \phi(t)+\phi^*(\tau)
	\geq \tau t,
\]
from which, dividing by $t$ and letting $t\to0^+$, we get $\phi'(0)\geq\tau$, hence $\tau=\phi'(0)$.
This proves that $\tau>\phi'(0)$ implies $\phi^*(\tau)>0$.
Since $\phi^*$ is convex, for every $\sigma>\tau$ we have
\[
	\phi^*(\sigma)
	\geq \phi^*(0) + \frac\sigma\tau \bigl(\phi^*(\tau)-\phi^*(0)\bigr)
	= \frac\sigma\tau \phi^*(\tau)
	> \phi^*(\tau),
\]
hence the claimed monotonicity. The continuity follows from the basic properties of the Legendre transform.
The bijectivety follows from continuity and strict monotonicity.

Using \eqref{eq:optimal-profile-explicit}, the condition $u(L)\leq H$ is therefore trivially
satisfied if $\lambda H\geq\phi^*(1)$, otherwise it is equivalent to
$\phi'(0)+\lambda(L-x_0) \leq (\phi^*)^{-1}(\lambda H)$, that is
\[
	x_0 \geq L - \frac{(\phi^*)^{-1}(\lambda H)-\phi'(0)}\lambda,
\]
where by $(\phi^*)^{-1}$ we denote the inverse of the strictly increasing restriction
$\phi^*\rvert_{[\phi'(0),1]}$. This latter inequality is correct in both cases if we interpret
$(\phi^*)^{-1}(\lambda H)=1$ when $\lambda H\geq\phi^*(1)$.
Notice that, since $(\phi^*)^{-1}(\lambda H)\leq 1$, this condition on $x_0$ is more restrictive
than the previously found lower bound $x_0\geq L-\frac{1-\phi'(0)}\lambda$.

\paragraph{Reduction to a one-dimensional minimization problem}
Finding the optimal profile $u$ has been reduced to minimizing the function
\begin{equation}\label{eq:function-x_0}
	f(x_0) \coloneqq \frac1\lambda \int_{\phi'(0)}^{\phi'(0)+\lambda(L-x_0)} y (\phi^*)'(y) \dy
	+ \theta\oleft(H - \frac1\lambda \phi^*\bigl(\phi'(0)+\lambda(L-x_0)\bigr)\right)
\end{equation}
under the constraint
\begin{equation}\label{eq:constraint-x_0}
	L - \frac{(\phi^*)^{-1}(\lambda H)-\phi'(0)}\lambda \leq x_0 \leq L.
\end{equation}

If $\phi$ is strictly convex then $\phi^*$ is differentiable everywhere, therefore we can compute
\[
	\begin{split}
		\frac{\d F(u)}{\dx_0}
		&= f'(x_0)
		= -y(\phi^*)'(y) \bigr\rvert_{y=\phi'(0)+\lambda(L-x_0)}
		+ \theta'\bigl(H-\phi^*(y)/\lambda\bigr) (\phi^*)'(y) \bigr\rvert_{y=\phi'(0)+\lambda(L-x_0)} \\
		&= (\phi^*)'(y) \bigl[\theta'\bigl(H-\phi^*(y)/\lambda\bigr) - y\bigr]
		\bigr\rvert_{y=\phi'(0)+\lambda(L-x_0)}.
	\end{split}
\]
If the minimizer $x_0$ is an interior point, it must satisfy $f'(x_0)=0$, which is equivalent to
\[
	\theta'\oleft(H-\frac1\lambda\phi^*\bigl(\phi'(0)+\lambda(L-x_0)\bigr)\right) = \phi'(0)+\lambda(L-x_0),
\]
because
\[
	(\phi^*)'\bigl(\phi'(0)+\lambda(L-x_0)\bigr)
	\geq \frac{\phi^*\bigl(\phi'(0)+\lambda(L-x_0)\bigr)-\phi^*\bigl(\phi'(0)\bigr)}{\lambda(L-x_0)}
	> 0
\]
cannot vanish by the strict monotonicity of $\phi^*$.

\paragraph{Acknowledgements}
The authors are members of GNAMPA, INdAM.
This work was partially supported by the European Union - Next Generation EU and Research Projects PRIN2022 PNRR ``Geometric-Analytic Methods for PDEs and Applications'' (2022SLTHCE, cup E53D23005880006), awarded with D.D.\ 104 - 02/02/2022 of the Italian Ministry for University and Research. This manuscript reflects only the authors' views and opinions and the Ministry cannot be considered responsible for them.

\paragraph{Conflict of interest}
The authors declare no conflict of interest.

\paragraph{Data availability statement}
We do not analyse or generate any datasets, because our work proceeds within a theoretical and mathematical approach.

\printbibliography[heading=bibintoc]

\end{document}

%% file: img/holder.pdf_tex
\begingroup%
\makeatletter%
\providecommand\color[2][]{%
	\errmessage{(Inkscape) Color is used for the text in Inkscape, but the package 'color.sty' is not loaded}%
	\renewcommand\color[2][]{}%
}%
\providecommand\transparent[1]{%
	\errmessage{(Inkscape) Transparency is used (non-zero) for the text in Inkscape, but the package 'transparent.sty' is not loaded}%
	\renewcommand\transparent[1]{}%
}%
\providecommand\rotatebox[2]{#2}%
\newcommand*\fsize{\dimexpr\f@size pt\relax}%
\newcommand*\lineheight[1]{\fontsize{\fsize}{#1\fsize}\selectfont}%
\ifx\svgwidth\undefined%
	\setlength{\unitlength}{250.37471981bp}%
	\ifx\svgscale\undefined%
		\relax%
	\else%
		\setlength{\unitlength}{\unitlength * \real{\svgscale}}%
	\fi%
\else%
	\setlength{\unitlength}{\svgwidth}%
\fi%
\global\let\svgwidth\undefined%
\global\let\svgscale\undefined%
\makeatother%
\begin{picture}(1,1)%
	\lineheight{1}%
	\setlength\tabcolsep{0pt}%
	\put(0,0){\includegraphics[width=\unitlength,page=1]{img/holder.pdf}}%
	\put(0.1622591,0.04025279){\makebox(0,0)[lt]{\lineheight{1.25}\smash{\begin{tabular}[t]{l}$a$\end{tabular}}}}%
	\put(0.86919921,0.04025279){\makebox(0,0)[lt]{\lineheight{1.25}\smash{\begin{tabular}[t]{l}$b$\end{tabular}}}}%
	\put(-0.00199727,0.20512689){\makebox(0,0)[lt]{\lineheight{1.25}\smash{\begin{tabular}[t]{l}$g(a)$\end{tabular}}}}%
	\put(-0.00199727,0.8761211){\makebox(0,0)[lt]{\lineheight{1.25}\smash{\begin{tabular}[t]{l}$g(b)$\end{tabular}}}}%
	\put(-0.00199727,0.36688485){\makebox(0,0)[lt]{\lineheight{1.25}\smash{\begin{tabular}[t]{l}$u(a)$\end{tabular}}}}%
	\put(-0.00199727,0.69039921){\makebox(0,0)[lt]{\lineheight{1.25}\smash{\begin{tabular}[t]{l}$u(b)$\end{tabular}}}}%
	\put(0.40201131,0.3830681){\makebox(0,0)[lt]{\lineheight{1.25}\smash{\begin{tabular}[t]{l}$g$\end{tabular}}}}%
	\put(0.3293594,0.48328708){\makebox(0,0)[lt]{\lineheight{1.25}\smash{\begin{tabular}[t]{l}$u$\end{tabular}}}}%
	\put(0.25212449,0.04025279){\makebox(0,0)[lt]{\lineheight{1.25}\smash{\begin{tabular}[t]{l}$a+(h/C)^{1/\alpha}$\end{tabular}}}}%
	\put(0.581631,0.04025279){\makebox(0,0)[lt]{\lineheight{1.25}\smash{\begin{tabular}[t]{l}$b-(h/C)^{1/\alpha}$\end{tabular}}}}%
\end{picture}%
\endgroup%

%% file: biblio.bib
@book{AmbrosioBrueSemola,
	title = {Lectures on {{Optimal Transport}}},
	author = {Ambrosio, Luigi and Brué, Elia and Semola, Daniele},
	date = {2021},
	series = {{{UNITEXT}}},
	volume = {130},
	publisher = {Springer International Publishing},
	location = {Cham},
	doi = {10.1007/978-3-030-72162-6},
	url = {https://link.springer.com/10.1007/978-3-030-72162-6},
	isbn = {978-3-030-72161-9, 978-3-030-72162-6},
	langid = {english},
}

@inbook{CCN_survey,
	author = "Caselles, V. and Chambolle, A. and Novaga, M.",
	editor = "Scherzer, Otmar",
	title = "Total Variation in Imaging",
	bookTitle = "Handbook of Mathematical Methods in Imaging",
	year = "2015",
	publisher = "Springer New York",
	address = "New York, NY",
	pages = "1455--1499",
	isbn = "978-1-4939-0790-8",
	doi = "10.1007/978-1-4939-0790-8_23",
	url = "https://doi.org/10.1007/978-1-4939-0790-8_23",
}

@article{ROF,
	title = {Nonlinear total variation based noise removal algorithms},
	journal = {Physica D: Nonlinear Phenomena},
	volume = {60},
	number = {1},
	pages = {259-268},
	year = {1992},
	issn = {0167-2789},
	doi = {https://doi.org/10.1016/0167-2789(92)90242-F},
	url = {https://www.sciencedirect.com/science/article/pii/016727899290242F},
	author = {Leonid I. Rudin and Stanley Osher and Emad Fatemi},
}

@article{Chambolle2004,
	author = "Chambolle, Antonin",
	title = "An algorithm for total variation minimization and applications",
	journal = "Journal of Mathematical Imaging and Vision",
	year = "2004",
	month = "Jan",
	day = "01",
	volume = "20",
	number = "1",
	pages = "89--97",
	issn = "1573-7683",
	doi = "10.1023/B:JMIV.0000011325.36760.1e",
	url = "https://doi.org/10.1023/B:JMIV.0000011325.36760.1e",
}

@article{Chambolle1997,
	author = "Chambolle, Antonin and Lions, Pierre-Louis",
	title = "Image recovery via total variation minimization and related
	         problems",
	journal = "Numerische Mathematik",
	year = "1997",
	month = "Apr",
	day = "01",
	volume = "76",
	number = "2",
	pages = "167--188",
	issn = "0945-3245",
	doi = "10.1007/s002110050258",
	url = "https://doi.org/10.1007/s002110050258",
}

@book{AFP,
	Author = {Ambrosio, Luigi and Fusco, Nicola and Pallara, Diego},
	Title = {Functions of bounded variation and free discontinuity problems},
	isbn = {9780198502456},
	lccn = {99046602},
	series = {Oxford Science Publications},
	year = {2000},
	publisher = {Clarendon Press},
}

@book{boyd_convex_2023,
	author = {Boyd, Stephen P. and Vandenberghe, Lieven},
	title = {Convex Optimization},
	publisher = {Cambridge University Press},
	date = {2023},
	location = {Cambridge New York Melbourne New Delhi Singapore},
	edition = {Version 29},
	isbn = {978-0-521-83378-3},
	url = {https://web.stanford.edu/%7Eboyd/cvxbook/bv_cvxbook.pdf},
	pagetotal = {716},
	langid = {english},
}

@article{Talenti2016,
	author = "Talenti, Giorgio",
	title = "The Art of Rearranging",
	journal = "Milan Journal of Mathematics",
	year = "2016",
	month = "6",
	day = "01",
	volume = "84",
	number = "1",
	pages = "105--157",
	issn = "1424-9294",
	doi = "10.1007/s00032-016-0253-6",
	url = "https://link.springer.com/article/10.1007/s00032-016-0253-6",
}

@book{Baernstein,
	author = {Baernstein II, Albert},
	title = {Symmetrization in Analysis},
	year = {2019},
	series = {New Mathematical Monographs},
	collection = {New Mathematical Monographs},
	publisher = {Cambridge University Press},
	place = {Cambridge},
}

@book{Rakotoson,
	title = {Relative {{Rearrangement}}},
	author = {Rakotoson, Jean Michel},
	date = {2025},
	series = {Lecture {{Notes}} in {{Mathematics}}},
	volume = {2376},
	publisher = {Springer Nature Switzerland},
	location = {Cham},
	doi = {10.1007/978-3-032-02228-8},
	url = {https://link.springer.com/10.1007/978-3-032-02228-8},
	isbn = {978-3-032-02227-1, 978-3-032-02228-8},
	langid = {english},
}

@article{Optim.jl,
	author = {Mogensen, Patrick K. and Riseth, Asbjørn N.},
	title = {Optim: A mathematical optimization package for Julia},
	journal = {Journal of Open Source Software},
	publisher = {The Open Journal},
	year = {2018},
	volume = {3},
	number = {24},
	pages = {615},
	doi = {10.21105/joss.00615},
	url = {https://joss.theoj.org/papers/10.21105/joss.00615},
}

@article{Julia,
	author = {Bezanson, Jeff and Edelman, Alan and Karpinski, Stefan and Shah,
	          Viral B.},
	title = {Julia: A Fresh Approach to Numerical Computing},
	journal = {SIAM Review},
	year = {2017},
	volume = {59},
	number = {1},
	pages = {65-98},
	doi = {10.1137/141000671},
	url = {https://epubs.siam.org/doi/10.1137/141000671},
}
